\documentclass[11pt,reqno]{amsart}

\newtheorem{thm}{Theorem}[section]
\newtheorem*{thm*}{Theorem}
\newtheorem{lem}[thm]{Lemma}
\newtheorem*{lem*}{Lemma}

\newtheorem{cor}[thm]{Corollary}
\newtheorem{claim}[thm]{Claim}
\newtheorem{prop}[thm]{Proposition}

\theoremstyle{definition}

\newtheorem{assump}[thm]{Assumption}
\newtheorem*{case*}{Case}

\newtheorem{defn}[thm]{Definition}
\newtheorem*{defn*}{Definition}
\newtheorem{exmp}[thm]{Example}
\newtheorem*{exmp*}{Example}

\newtheorem{hyp}[thm]{Hypothesis}

\newtheorem{step}{Step}\renewcommand{\thestep}{}

\theoremstyle{remark}

\newtheorem{case}{Case}\renewcommand{\thecase}{}

\newtheorem{rmk}[thm]{Remark}
\newtheorem*{rmk*}{Remark}

\makeatletter
\def\alphenumi{
  \def\theenumi{\alph{enumi}}
  \def\p@enumi{\theenumi}
  \def\labelenumi{(\@alph\c@enumi)}}
\makeatother

\makeatletter
\def\thecase{\@arabic\c@case}
\makeatother

\makeatletter
\def\thestep{\@arabic\c@step}
\makeatother

\newcount\hh
\newcount\mm
\mm=\time
\hh=\time
\divide\hh by 60
\divide\mm by 60
\multiply\mm by 60
\mm=-\mm
\advance\mm by \time
\def\hhmm{\number\hh:\ifnum\mm<10{}0\fi\number\mm}

\let\oldmarginpar\marginpar
\renewcommand\marginpar[1]{\-\oldmarginpar[\raggedleft\footnotesize #1]%
{\raggedright\footnotesize #1}}

\renewcommand\emptyset{\varnothing}

\newcommand\HH{\mathbb{H}}

\newcommand\RR{\mathbb{R}}

\newcommand\sA{{\mathscr{A}}}
\newcommand\sB{{\mathscr{B}}}
\newcommand\sC{{\mathscr{C}}}

\newcommand\eps{\varepsilon}

\newcommand\less{\setminus}

\newcommand\diag{\operatorname{diag}}

\DeclareMathOperator{\mydirac}{\slashed{\partial}}

\DeclareMathOperator{\Int}{int}

\newcommand\loc{\operatorname{loc}}

\newcommand\supp{\operatorname{supp}}

\newcommand\tr{\operatorname{tr}}

\numberwithin{equation}{section}

\usepackage{amssymb}
\usepackage{hyperref}
\usepackage{mathrsfs}
\usepackage{slashed}
\usepackage[usenames]{color}
\usepackage{url}
\usepackage{verbatim}

\hypersetup{pdftitle={A Schauder approach to degenerate-parabolic partial differential equations with unbounded coefficients}}
\hypersetup{pdfauthor={Paul M. N. Feehan and Camelia A. Pop}}

\begin{document}

\title[Schauder approach to degenerate-parabolic equations]{A Schauder approach to degenerate-parabolic partial differential equations with unbounded coefficients}

\author[P. M. N. Feehan]{Paul M. N. Feehan}
\address[PF]{Department of Mathematics, Rutgers, The State University of New Jersey, 110 Frelinghuysen Road, Piscataway, NJ 08854-8019, United States}
\email[PF]{feehan@math.rutgers.edu}

\author[C. A. Pop]{Camelia A. Pop}
\address[CP]{Department of Mathematics, University of Pennsylvania, 209 South 33rd Street, Philadelphia, PA 19104-6395, United States}
\email{cpop@math.upenn.edu}

\date{August 8, 2013. Incorporates final galley proof corrections corresponding to published version: Journal of Differential Equations (2013) \textbf{254}, 4401--4445, \url{dx.doi.org/10.1016/j.jde.2013.03.006}.}

\begin{abstract}
Motivated by applications to probability and mathematical finance, we consider a parabolic partial differential equation on a half-space whose coefficients are suitably H\"older continuous and allowed to grow linearly in the spatial variable and which become degenerate along the boundary of the half-space. We establish existence and uniqueness of solutions in weighted H\"older spaces which incorporate both the degeneracy at the boundary and the unboundedness of the coefficients. In our companion article \cite{Feehan_Pop_mimickingdegen_probability}, we apply the main result of this article to show that the martingale problem associated with a degenerate-elliptic partial differential operator is well-posed in the sense of Stroock and Varadhan.
\end{abstract}

%



\thanks{PF was partially supported by NSF grant DMS-1059206. CP was partially supported by a Rutgers University fellowship. }

\maketitle
\tableofcontents

\section{Introduction}
\label{sec:Intro}
Motivated by applications to probability theory and mathematical finance \cite{Antonov_Misirpashaev_Piterbarg_2009, Dupire1994, Gyongy, Piterbarg_markovprojection}, we use a Schauder approach to prove existence, uniqueness, and regularity of solutions to a \emph{degenerate}-parabolic partial differential equation with \emph{unbounded}, locally H\"older-continuous coefficients, $(a,b,c)$ with $a=(a^{ij})$ and $b=(b^i)$, generalizing both the \emph{Heston equation} \cite{Heston1993} and the model linear degenerate-parabolic equation used in the study of the
\emph{porous medium equation} \cite{DaskalHamilton1998, Daskalopoulos_Rhee_2003, Koch},
\begin{equation}
\label{eq:Problem}
\begin{aligned}
\begin{cases}
Lu=f & \hbox{ on  } \HH_T,\\
u(0,\cdot)=g & \hbox{  on   } \overline{\HH},
\end{cases}
\end{aligned}
\end{equation}
where $\HH := \RR^{d-1}\times(0,\infty)$ (with $d\geq 2$) denotes the half-space $\{x_d\geq 0\}$, and $\HH_T := (0,T) \times \HH$ is the open half-cylinder with $0<T<\infty$, and
\begin{equation}
\label{eq:Generator}
\begin{aligned}
-Lu = -u_t + \sum_{i,j=1}^d x_d a^{ij}u_{x_ix_j} + \sum_{i=1}^d b^iu_{x_i} + cu, \quad \forall\, u \in C^{1,2}(\HH_T).
\end{aligned}
\end{equation}
The operator $L$ becomes \emph{degenerate} along the boundary $\partial\HH = \{x_d=0\}$ of the half-space but in addition, unlike the model linear degenerate-parabolic equation considered in \cite{DaskalHamilton1998, Daskalopoulos_Rhee_2003, Koch}, the coefficients of \eqref{eq:Generator} are also permitted to grow linearly with $x$ as $x\to\infty$ and, even when the coefficients $b^i$ are constant, we do not require that $b^i=0$ when $i=1,\ldots,d-1$.

In our companion article \cite{Feehan_Pop_mimickingdegen_probability}, we apply the main result of the present article (Theorem \ref{thm:MainExistenceUniquenessPDE}) to prove that the martingale problem associated with the degenerate-elliptic partial differential operator acting on $v \in C^{2}(\overline\HH)$,
\begin{equation}
\label{eq:MartingaleGenerator}
\sA_tv(x) := \frac{1}{2}\sum_{i,j=1}^d x_da^{ij}(t,x)v_{x_ix_j}(x) + \sum_{i=1}^d b^i(t,x)v_{x_i}(x), \quad (t,x) \in [0,\infty)\times\HH,
\end{equation}
is well-posed in the sense of Stroock and Varadhan \cite{Stroock_Varadhan}. In \cite{Feehan_Pop_mimickingdegen_probability}, we then prove existence, uniqueness, and the strong Markov property for weak solutions to the associated stochastic differential equation with degenerate diffusion coefficient and unbounded diffusion and drift coefficients with suitable H\"older continuity properties. Finally, in \cite{Feehan_Pop_mimickingdegen_probability}, given an It\^o process with degenerate diffusion coefficient and unbounded but appropriately regular diffusion and drift coefficients, we prove existence of a strong Markov process, unique in the sense of probability law, whose one-dimensional marginal probability distributions match those of the given It\^o process.

\subsection{Summary of main results}
\label{subsec:Summary}
We describe our results outlined in the preamble to Section \ref{sec:Intro}. We shall seek a solution, $u$, to \eqref{eq:Problem} in a certain weighted H\"older space $\sC^{2+\alpha}_p(\overline{\HH}_T)$, given a source function, $f$, in a weighted H\"older space $\sC^{\alpha}_p(\overline{\HH}_T)$ and initial data, $g$, in a weighted H\"older space $\sC^{2+\alpha}_p(\overline{\HH})$. These weighted H\"older spaces generalize both the standard H\"older spaces as defined, for example, in \cite{Krylov_LecturesHolder, Lieberman} and the H\"older spaces defined with the cycloidal metric and introduced, independently, by Daskalopoulos and Hamilton \cite{DaskalHamilton1998} and Koch \cite{Koch}. We defer a detailed description of these H\"older spaces to Section \ref{subsec:DHKHolderSpaces}. However, the essential features of our H\"older spaces are that (i) near the boundary, $x_d=0$, of the half-space cylinder $\HH_T$, our H\"older spaces are equivalent to those of Daskalopoulos, Hamilton, and Koch and account for the
degeneracy of the operator $L$, (ii) polynomial weights in the definition of
our H\"older spaces allow for coefficients $(x_d a,b,c)$ in \eqref{eq:Generator} with up to linear growth near $x=\infty$ in the half-space cylinder $\HH_T$, and (iii) on compact subsets of the half-space cylinder $\HH_T$, our H\"older spaces are equivalent to standard H\"older spaces. We defer a detailed description of the conditions on the coefficients $(a,b,c)$ defining $L$ in \eqref{eq:Generator} to Section \ref{subsec:AssumptionsPDECoefficients} --- see Assumption \ref{assump:Coeff} on the properties of the coefficients of the parabolic differential operator. However, the essential features of the conditions on $(a,b,c)$ in Assumption \ref{assump:Coeff} are that (i) the matrix $a=(a^{ij})$ is \emph{uniformly elliptic}, so the \emph{degeneracy} in \eqref{eq:Generator} is captured by the common factor $x_d$ appearing in the $u_{x_ix_j}$ terms, (ii) the coefficients $(x_d a,b,c)$ have at most \emph{linear growth} with respect to $x\in\HH$ as $x\to\infty$, (iii) the coefficients $(a,b,c)$ are \emph{locally H\"
older continuous} on $\overline\HH_T$ with exponent $\alpha\in(0,1)$, (iv) the coefficient $c$ is \emph{bounded above} on $\HH_T$ by a constant, and (v) the coefficient $b^d$ is \emph{positive} when $x_d=0$. We can now state our first main result.

\begin{thm} [Existence and uniqueness of solutions to a degenerate-parabolic partial differential equation with unbounded coefficients]
\label{thm:MainExistenceUniquenessPDE}
Assume that the coefficients $(a,b,c)$ in \eqref{eq:Generator} obey the conditions in Assumption \ref{assump:Coeff}. Then there is a positive constant $p$, depending only on the H\"older exponent $\alpha\in(0,1)$, such that for any $T>0$, $f \in \sC^{\alpha}_p(\overline{\HH}_T)$ and $g \in \sC^{2+\alpha}_p(\overline{\HH})$, there exists a unique solution
$u \in \sC^{2+\alpha}(\overline{\HH}_T)$ to \eqref{eq:Problem}. Moreover, $u$ satisfies the a priori estimate
\begin{equation}
\label{eq:GlobalEstimate}
\|u\|_{\sC^{2+\alpha}(\overline{\HH}_T)}
\leq C \left(  \|f\|_{\sC^{\alpha}_p(\overline{\HH}_T)} +  \|g\|_{\sC^{2+\alpha}_p(\overline{\HH})} \right),
\end{equation}
where $C$ is a positive constant, depending only on $K$, $\nu$, $\delta$, $d$, $\alpha$ and $T$.
\end{thm}

One of the difficulties in establishing Theorem \ref{thm:MainExistenceUniquenessPDE} is that the coefficient, $x_da(t,x)$, becomes degenerate when $x_d=0$ and is allowed to have linear growth in $x$, instead of being uniformly elliptic and bounded as  in \cite[Hypothesis 2.1]{Krylov_Priola_2010}. To address the degeneracy of $x_da(t,x)$ as $x_d\downarrow 0$, we build on the results on \cite[Theorem I.1.1]{DaskalHamilton1998} by employing a localization procedure. To address the linear growth of the coefficients $(x_da,b,c)$ of the parabolic operator $L$ in \eqref{eq:Generator}, we augment previous definitions of weighted H\"older spaces \cite{DaskalHamilton1998, Koch}, by introducing a weight $(1+|x|)^p$, where $p$ is a positive constant depending only on the H\"older exponent $\alpha\in(0,1)$. The proof of existence does not follow by standard methods, for example, the method of continuity, because $L:\sC^{2+\alpha}(\overline{\HH}_T)\rightarrow \sC^{\alpha}_p(\overline{\HH}_T)$ is not a well-defined
operator. In general, the domain of definition of $L$ is a subspace of $\sC^{2+\alpha}(\overline{\HH}_T)$ which depends on the nature of the coefficients of $L$, a feature which is not encountered in the case of parabolic operators with bounded coefficients . To circumvent this difficulty, we first consider the case of similar degenerate operators with \emph{bounded} coefficients and then use an approximation procedure to obtain our solution. To obtain convergence of sequences to a solution of our parabolic differential equation \eqref{eq:Problem}, we prove a priori estimates in the weighted H\"older spaces $\sC^{\alpha}_p$ and $\sC^{2+\alpha}_p$.

The conditions in Assumption \ref{assump:Coeff} on the coefficients $(a,b,c)$ in \eqref{eq:Generator} are mild enough that they allow for many examples of interest in mathematical finance.

\begin{exmp}[Parabolic Heston partial differential equation]
\label{exmp:HestonPDE}
The conditions in Assumption \ref{assump:Coeff} are obeyed by the coefficients of the parabolic Heston partial differential operator,
\begin{equation}
\label{eq:HestonPDE}
-Lu = -u_t + \frac{y}{2}\left(u_{xx} + 2\varrho\sigma u_{xy} + \sigma^2 u_{yy}\right) + \left(r-q-\frac{y}{2}\right)u_x + \kappa(\theta-y)u_y - ru,
\end{equation}
where $q\geq 0, r\geq 0, \kappa>0, \theta>0, \sigma>0$, and $\varrho\in (-1,1)$ are constants.
\end{exmp}

Naturally, the conditions in Assumption \ref{assump:Coeff} on the coefficients $(a,b,c)$ in \eqref{eq:Generator} also allow for the model linear degenerate-parabolic equation used in the study of the porous medium equation.

\begin{exmp}[Model linear degenerate-parabolic equation]
\label{exmp:GeneralizedPorousMediumEquation}
In their landmark article, Daskalopoulos and Hamilton \cite{DaskalHamilton1998} proved existence and uniqueness of $C^\infty$ solutions, $u$, to the Cauchy problem for the porous medium equation \cite[p. 899]{DaskalHamilton1998} (when $d=2$),
\begin{equation}
\label{eq:PorousMediumEquation}
-u_t + \sum_{i=1}^d (u^m)_{x_ix_i} = 0 \quad\hbox{on }(0,T)\times\RR^d, \quad u(\cdot, 0) = g \quad\hbox{on }\RR^d,
\end{equation}
where $m>1$ and $g$ is non-negative, integrable, and compactly supported on $\RR^d$, together with $C^\infty$-regularity of its free boundary, $\partial\{u>0\}$. Their analysis is based on an extensive development of existence, uniqueness, and regularity results for the linearization of the porous medium equation near the free boundary and, in particular, their \emph{model linear degenerate operator} \cite[p. 901]{DaskalHamilton1998} (generalized from $d=2$ in their article),
\begin{equation}
\label{eq:DHModel}
-Lu = -u_t + x_d\sum_{i=1}^d u_{x_ix_i} + \nu u_{x_d},
\end{equation}
where $\nu$ is a positive constant. The same model linear degenerate operator (for $d\geq 2$), was studied independently by Koch \cite[Equation (4.43)]{Koch} and, in a remarkable Habilitation thesis, he obtained existence, uniqueness, and regularity results for solutions to \eqref{eq:PorousMediumEquation} which complement those of Daskalopoulos and Hamilton \cite{DaskalHamilton1998}.
Even when the coefficients in \eqref{eq:Generator} are constant, our operator \emph{cannot} be transformed by simple coordinate changes to one of the form \eqref{eq:DHModel}, but rather one of the form \eqref{eq:SimplerForm}. Similarly, the operator \eqref{eq:HestonPDE} \emph{cannot} be transformed by simple coordinate changes to one of the form \eqref{eq:DHModel}, even when the factor $y$ in the coefficients of $u_x$ and $u_y$ in \eqref{eq:HestonPDE} is (artificially) replaced by zero.
\end{exmp}

\subsection{Connections with previous research on degenerate partial differential equations}
We provide a brief survey of some related research by other authors on Schauder a priori estimates and regularity theory for solutions to degenerate-elliptic and degenerate-parabolic partial differential equations most closely related to the results described in our article.

The principal feature which distinguishes the Cauchy problem \eqref{eq:Problem}, when the operator $L$ is given by \eqref{eq:Generator}, from the linear, second-order, strictly parabolic operators in \cite{Krylov_LecturesHolder, LadyzenskajaSolonnikovUralceva, Lieberman} and their initial-boundary value problems, is the degeneracy of $L$ due to the factor, $x_d$, in the coefficients of $u_{x_ix_j}$ and, because the coefficient $b^d$ of $u_{x_d}$ in \eqref{eq:Generator} is positive, the fact that boundary conditions may be omitted along $x_d=0$ when we seek solutions, $u$, with sufficient regularity up to $x_d=0$.

The literature on degenerate-elliptic and parabolic partial differential equations is vast, with the well-known articles of Fabes, Kenig, and Serapioni \cite{Fabes_1982, Fabes_Kenig_Serapioni_1982a}, Fichera \cite{Fichera_1956, Fichera_1960}, Kohn and Nirenberg \cite{Kohn_Nirenberg_1967}, Murthy and Stampacchia \cite{Murthy_Stampacchia_1968, Murthy_Stampacchia_1968corr} and the monographs of Levendorski{\u\i} \cite{LevendorskiDegenElliptic} and Ole{\u\i}nik and Radkevi{\v{c}} \cite{Oleinik_Radkevic, Radkevich_2009a, Radkevich_2009b}, being merely the tip of the iceberg.

As far as the authors can tell, however, there has been relatively little prior work on a priori Schauder estimates and higher-order H\"older regularity of solutions up to the portion of the domain boundary where the operator becomes degenerate. In this context, the work of Daskalopoulos, Hamilton, and Rhee \cite{DaskalHamilton1998, Daskalopoulos_Rhee_2003, RheeThesis} and of Koch stands out in recent years because of their introduction of the cycloidal metric on the upper half-space, weighted H\"older norms, and weighted Sobolev norms which provide the key ingredients required to unlock the existence, uniqueness, and higher-order regularity theory for solutions to the porous medium equation \eqref{eq:PorousMediumEquation} and the linear degenerate-parabolic model equation \eqref{eq:DHModel} on the upper half-space.

While the Daskalopoulos-Hamilton Schauder theory for degenerate-parabolic operators has been adopted so far by relatively few other researchers, it has also been employed by De Simone, Giacomelli, Kn\"upfer, and Otto in \cite{DeSimone_Knupfer_Otto_2006, Giacomelli_Knupfer_Otto_2008, Giacomelli_Knupfer_2010} and by Epstein and Mazzeo in \cite{Epstein_Mazzeo_2010}.

\subsection{Extensions and future work}
Motivated by the results obtained in our related ``degenerate-elliptic'' article \cite{Feehan_Pop_elliptichestonschauder}, it is natural to consider higher-order, interior and boundary regularity and a priori Schauder estimates for solutions $u \in C^{k,2+\alpha}_s(\underline Q)\cap C(\bar Q)$ to an initial-boundary value problem,
\begin{equation}
\label{eq:InitialBoundaryProblem}
\begin{cases}
Lu = f &\hbox{on } Q,
\\
u = g &\hbox{on } \mydirac_1\!Q,
\end{cases}
\end{equation}
generalizing the Cauchy problem \eqref{eq:Problem} for the operator $L$ in \eqref{eq:Generator}. Here, the cylinder $(0,T)\times\HH$ has been replaced by a subdomain $Q\subset(0,T)\times\HH$ with non-empty ``degenerate boundary'' portion
$\mydirac_0\!Q := \Int(((0,T)\times\partial\HH)\cap \partial Q)$ of the parabolic boundary,
$$
\mydirac Q := \left((\{0\}\times\HH)\cap \bar Q\right) \cup \left(((0,T)\times\HH)\cap \partial Q\right),
$$
and $\mydirac_1\! Q := \mydirac Q \less \overline{\mydirac_0\!Q}$ denotes the ``non-degenerate boundary'' portion of the parabolic boundary, while $\underline Q := Q\cup \mydirac_0\!Q$.

For reasons we summarize in \cite[\S 1.3]{Feehan_Pop_higherregularityweaksoln}, the development of \emph{global} Schauder a priori estimates, regularity, and existence theory for solutions $u \in C^{k,2+\alpha}_s(\bar Q)$ to \eqref{eq:InitialBoundaryProblem} appears very difficult when the intersection, $\overline{\mydirac_0\!Q}\cap\overline{\mydirac_1\! Q}$, of the ``degenerate and non-degenerate boundary'' portions is non-empty. In fact, even the development of an existence theory for solutions $u$ to \eqref{eq:InitialBoundaryProblem} just belonging to $C^{2+\alpha}_s(\underline Q)\cap C(\bar Q)$ is already a challenging problem which is not addressed in \cite{DaskalHamilton1998, Daskalopoulos_Rhee_2003}.

While our a priori Schauder estimates rely on the specific form of the degeneracy factor, $x_d$, of the operator $L$ in \eqref{eq:Generator} on a subdomain of the half-space, we obtained weak and strong maximum principles for a much broader class of degenerate-elliptic operators in \cite{Feehan_maximumprinciple}. Thus, for degenerate-parabolic operators such as
$$
-Lv = -v_t + \vartheta\sum_{i,j=1}^d a^{ij}v_{x_ix_j} + \sum_{i=1}^d b^iv_{x_i} + cv \quad\hbox{on } Q, \quad v\in C^\infty(Q),
$$
where $\vartheta \in C^\alpha_{\loc}(\bar Q)$ and $\vartheta>0$ on a subdomain $Q\subset(0,T)\times\RR^d$ with non-empty ``degenerate boundary'' portion $\mydirac_0\!Q := \Int(\{(t,x)\in\partial Q:\vartheta(x)=0\})$ of the parabolic boundary,
$$
\mydirac Q := \left((\{0\}\times\RR^d)\cap \bar Q\right) \cup \left(((0,T)\times\RR^d)\cap \partial Q\right),
$$
we plan to develop a priori Schauder estimates, regularity, and existence theory in a subsequent article.

\subsection{Outline of the article}
\label{subsec:Guide}
In Section \ref{sec:WeightedHolderSpaces}, we define the H\"older spaces required to prove Theorem \ref{thm:MainExistenceUniquenessPDE} (existence and uniqueness of solutions to a degenerate-parabolic partial differential equation on a half-space with unbounded coefficients) and provide a detailed description of the conditions required of the coefficients $(a,b,c)$ in the statement of Theorem \ref{thm:MainExistenceUniquenessPDE}, which we then proceed to prove in Section \ref{sec:ExistenceUniquenessRegularityInhomoInitValProblem}. Appendices \ref{app:ExistenceUniquenessDegenParabolicPDEConstantCoefficients} and \ref{app:ProofKrylovResults} contain proofs for several results which are slightly more technical than those in the body of the article.

\subsection{Notation and conventions}
We adopt the convention that a condition labeled as an \emph{Assumption} is considered to be universal and in effect throughout this article and so not referenced explicitly in theorem and similar statements; a condition labeled as a \emph{Hypothesis} is only considered to be in effect when explicitly referenced.

\subsection{Acknowledgments} We would like to thank Panagiota Daskalopoulos and Peter Laurence for stimulating discussions on degenerate partial differential equations. We are very grateful to everyone who provided us with comments on a previous version of this article or related conference and seminar presentations. Finally, we thank the anonymous referee for a careful reading of our manuscript.

\section{Weighted H\"older spaces and coefficients of the differential operators}
\label{sec:WeightedHolderSpaces}
In Section \ref{subsec:DHKHolderSpaces}, we introduce the H\"older spaces required for the statement and proof of Theorem \ref{thm:MainExistenceUniquenessPDE}, while in Section \ref{subsec:AssumptionsPDECoefficients}, we describe the regularity and growth conditions required of the coefficients $(a, b, c)$ in Theorem \ref{thm:MainExistenceUniquenessPDE}.

\subsection{Weighted H\"older spaces}
\label{subsec:DHKHolderSpaces}
For $a>0$, we denote
$$
\HH_{a,T} := (0,T) \times\RR^{d-1}\times(0,a),
$$
and, when $T=\infty$, we denote $\HH_\infty = (0,\infty) \times\HH$ and $\HH_{a,\infty} = (0,\infty)\times\RR^{d-1}\times(0,a)$. We denote the usual closures these half-spaces and cylinders by $\overline{\HH}:=\RR^{d-1}\times[0,\infty)$, $\overline{\HH}_T:=[0,T]\times\overline\HH$, while $\overline{\HH}_{a,T}:=[0,T]\times\RR^{d-1}\times[0,a]$. We write points in $\HH$ as $x:=(x',x_d)$, where $x':=(x_1, x_2, \ldots, x_{d-1})\in \RR^{d-1}$. For $x^0 \in \overline{\HH}$ and $R>0$, we let
\begin{align*}
B_R(x^0) &:= \left\{x \in \HH: |x-x^0| < R \right\},
\\
Q_{R,T}(x^0) &:= (0,T)\times B_R(x^0),
\end{align*}
and denote their usual closures by $\bar B_R(x^0) := \{x \in \HH: |x-x^0| \leq R\}$ and $\bar Q_{R,T}(x^0) := [0,T]\times \bar B_R(x^0)$, respectively. We write $B_R$ or $Q_{R,T}$ when the center, $x^0$, is clear from the context or unimportant.

A parabolic partial differential equation with a degeneracy similar to that considered in this article arises in the study of the porous medium equation \cite{DaskalHamilton1998, Daskalopoulos_Rhee_2003, Koch}. The existence, uniqueness, and regularity theory for such equations is facilitated by the use of H\"older spaces defined by the \emph{cycloidal metric} on $\HH$ introduced by Daskalopoulos and Hamilton \cite{DaskalHamilton1998} and, independently, by Koch \cite{Koch}. See \cite[p. 901]{DaskalHamilton1998} for a discussion of this metric. Following \cite[p. $901$]{DaskalHamilton1998}, we define the \emph{cycloidal distance} between two points, $P_1=(t_1, x^1), P_2=(t_2, x^2) \in [0,\infty)\times\overline{\HH}$, by
\begin{equation}
\label{eq:CycloidalMetric}
\begin{aligned}
s(P_1, P_2) &:= \frac{\sum_{i=1}^d |x_i^1-x_i^2|}{\sqrt{x_d^1}+\sqrt{x_d^2}+\sqrt{\sum_{i=1}^{d-1} |x_i^1-x_i^2|}}+\sqrt{|t_1-t_2|}.
\end{aligned}
\end{equation}
Following \cite[p. 117]{Krylov_LecturesHolder}, we define the usual Euclidean distance between points $P_1, P_2 \in [0,\infty)\times\RR^d$ by
\begin{equation}
\label{eq:UsualMetric}
\begin{aligned}
\rho(P_1, P_2) &:= \sum_{i=1}^d |x_i^1-x_i^2| + \sqrt{|t_1-t_2|}.
\end{aligned}
\end{equation}

\begin{rmk}[Equivalence of the cycloidal and Euclidean distance functions on suitable subsets of $[0,\infty)\times\HH$]
\label{rmk:EquivalenceMetric}
The cycloidal and Euclidean distance functions, $s$ and $\rho$, are equivalent on sets of the form $[0,\infty)\times\RR^{d-1}\times[y_0,y_1]$, for any $0<y_0<y_1$.
\end{rmk}

Let $\Omega\subset (0,T)\times\HH$ be an open set and $\alpha \in (0,1)$. We denote by $C(\bar \Omega)$ the space of bounded, continuous functions on $\bar \Omega$, and by $C^{\infty}_0(\bar \Omega)$ the space of smooth functions with compact support in $\bar \Omega$. For a function $u:\bar \Omega\rightarrow\RR$, we consider the following norms and seminorms
\begin{align}
\label{eq:SupNorm}
 \|u\|_{C(\bar \Omega)}             &= \sup_{P \in \bar \Omega}  |u(P)|,\\
\label{eq:SeminormS}
 [u]_{C^{\alpha}_s(\bar \Omega)}      &= \sup_{\stackrel{P_1, P_2 \in \bar \Omega,} {P_1 \neq P_2}} \frac{|u(P_1)-u(P_2)|}{s^{\alpha}(P_1,P_2)},\\
\label{eq:SeminormRHO}
 [u]_{C^{\alpha}_{\rho}(\bar \Omega)} &= \sup_{\stackrel{P_1, P_2 \in \bar \Omega,} {P_1 \neq P_2}} \frac{|u(P_1)-u(P_2)|}{\rho^\alpha(P_1,P_2)}.
\end{align}
We say that $u \in C^{\alpha}_s(\bar \Omega)$ if $u \in  C(\bar \Omega)$ and
\[
\|u\|_{C^{\alpha}_s(\bar \Omega)} = \|u\|_{C(\bar \Omega)}+[u]_{C^{\alpha}_s(\bar \Omega)}  < \infty.
\]
Analogously, we define the H\"older space $C^{\alpha}_{\rho}(\bar \Omega)$ of functions $u$ which satisfy
\[
\|u\|_{C^{\alpha}_{\rho}(\bar \Omega)} = \|u\|_{C(\bar \Omega)}+[u]_{C^{\alpha}_{\rho}(\bar \Omega)}  < \infty.
\]
We say that $u \in C^{2+\alpha}_s(\bar \Omega)$ if
\[
\|u\|_{C^{2+\alpha}_s(\bar \Omega)}
:= \|u\|_{C^{\alpha}_s(\bar \Omega)}+\|u_t\|_{C^{\alpha}_s(\bar \Omega)}
+\max_{1\leq i\leq d}\|u_{x_i}\|_{C^{\alpha}_s(\bar \Omega)}
+\max_{1\leq i,j\leq d}\|x_du_{x_ix_j}\|_{C^{\alpha}_s(\bar \Omega)} <\infty,
\]
and $u \in C^{2+\alpha}_\rho(\bar \Omega)$ if
\[
\|u\|_{C^{2+\alpha}_{\rho}(\bar \Omega)}
= \|u\|_{C^{\alpha}_{\rho}(\bar \Omega)}+\|u_t\|_{C^{\alpha}_{\rho}(\bar \Omega)}
+ \max_{1\leq i\leq d} \|u_{x_i}\|_{C^{\alpha}_{\rho}(\bar \Omega)}
+ \max_{1\leq i,j\leq d} \|u_{x_ix_j}\|_{C^{\alpha}_{\rho}(\bar \Omega)}  <\infty.
\]
We denote by $C^{\alpha}_{s, \loc}(\bar \Omega)$ the space of functions $u$ with the property that for any compact set $K\subseteq \bar \Omega$, we have $u \in C^{\alpha}_s(K)$. Analogously, we define the spaces $C^{2+\alpha}_{s, \loc}(\bar \Omega)$, $C^{\alpha}_{\rho, \loc}(\bar \Omega)$ and $C^{2+\alpha}_{\rho, \loc}(\bar \Omega)$.

We prove existence, uniqueness and regularity of solutions for a parabolic operator \eqref{eq:Generator} whose second-order coefficients are degenerate on $\partial\HH$. For this purpose, we will make use of the
following H\"older spaces,

\begin{equation*}
\begin{aligned}
\sC^{\alpha}(\overline{\HH}_T)
&:= \left\{u: u \in C^{\alpha}_s(\overline{\HH}_{1,T}) \cap C^{\alpha}_{\rho}(\overline{\HH}_T\less\HH_{1,T}) \right\},\\
\sC^{2+\alpha}(\overline{\HH}_T)
&:= \left\{u: u \in C^{2+\alpha}_s(\overline{\HH}_{1,T}) \cap C^{2+\alpha}_{\rho}(\overline{\HH}_T\less\HH_{1,T}) \right\}.
\end{aligned}
\end{equation*}
We define $\sC^{\alpha}(\overline{\HH})$ and $\sC^{2+\alpha}(\overline{\HH})$ in the analogous manner.

The coefficient functions $x_d a^{ij}(t,x)$, $b^i(t,x)$ and $c(t,x)$ of the parabolic operator \eqref{eq:Generator} are allowed to have linear growth in $|x|$. To account for the unboundedness of the coefficients, we  augment our definition of H\"older spaces by introducing weights of the form $(1+|x|)^q$, where $q\geq 0$ will be suitably chosen in the sequel. For $q\geq0$, we define
\begin{equation}
\label{eq:DefinitionX_0_q}
\begin{aligned}
\|u\|_{\sC^0_q(\overline{\HH})}
&:= \sup_{x \in \overline{\HH}}\left(1+ |x|\right)^q |u(x)|,
\end{aligned}
\end{equation}
and, given $T>0$, we define
\begin{align}
\label{eq:DefinitionX_0_q_T}
\|u\|_{\sC^0_q(\overline{\HH}_T)}
&:= \sup_{(t,x)\in\overline{\HH}_T} \left(1+|x|\right)^q|u(t,x)|.
\end{align}
Moreover, given $\alpha\in(0,1)$, we define
\begin{align}
\label{eq:DefinitionX_Alpha_q_T}
\|u\|_{\sC^{\alpha}_q(\overline{\HH}_T)}
&:= \|u\|_{\sC^0_q(\overline{\HH}_T)} + [(1+|x|)^q u]_{C^{\alpha}_s(\overline{\HH}_{1,T})} + [(1+|x|)^q u]_{C^{\alpha}_{\rho}(\overline{\HH}_T\less\HH_{1,T})},\\
\label{eq:DefinitionX_Two_Aplha_q_T}
\|u\|_{\sC^{2+\alpha}_q(\overline{\HH}_T)}
&:=  \|u\|_{\sC^{\alpha}_q(\overline{\HH}_T)}+ \|u_t\|_{\sC^{\alpha}_q(\overline{\HH}_T)} + \|u_{x_i}\|_{\sC^{\alpha}_q(\overline{\HH}_T)}+ \|x_du_{x_ix_j}\|_{\sC^{\alpha}_q(\overline{\HH}_T)}.
\end{align}
The vector spaces
\begin{align*}
\sC^0_q(\overline{\HH}_T)
&:= \left\{u \in C(\overline{\HH}_T): \|u\|_{\sC^{0}_q(\overline{\HH}_T)} <\infty \right\},\\
\sC^{\alpha}_q(\overline{\HH}_T)
&:= \left\{u \in \sC^{\alpha}(\overline{\HH}_T): \|u\|_{\sC^{\alpha}_q(\overline{\HH}_T)} <\infty \right\},\\
\sC^{2+\alpha}_q(\overline{\HH}_T)
&:= \left\{u \in \sC^{2+\alpha}(\overline{\HH}_T): \|u\|_{\sC^{2+\alpha}_q(\overline{\HH}_T)} <\infty \right\},
\end{align*}
can be shown to be Banach spaces with respect to the norms \eqref{eq:DefinitionX_0_q_T}, \eqref{eq:DefinitionX_Alpha_q_T} and \eqref{eq:DefinitionX_Two_Aplha_q_T}, respectively. We define the vector spaces $\sC^0_q(\overline{\HH})$, $\sC^{\alpha}_q(\overline{\HH})$, and $\sC^{2+\alpha}_q(\overline{\HH})$ similarly, and each can be shown to be a Banach space when equipped with the corresponding norm.

We let $\sC^{2+\alpha}_{q, \loc}(\overline{\HH}_T)$ denote the vector space of functions $u$ such that for any compact set $K \subset \overline{\HH}_T$, we have $u \in \sC^{2+\alpha}_q(K)$, for $q \geq 0$.

When $q=0$, the subscript $q$ is omitted in the preceding definitions.

\subsection{Coefficients of the differential operators}
\label{subsec:AssumptionsPDECoefficients}
Unless other conditions are explicitly substituted, we require in this article that the coefficients $(a,b,c)$ of the parabolic differential operator $L$ in \eqref{eq:Generator} satisfy the conditions in the following

\begin{assump}[Properties of the coefficients of the parabolic differential operator]
\label{assump:Coeff}
There are constants $\delta>0$, $K>0$, $\nu>0$ and $\alpha \in (0,1)$ such that the following hold.
\begin{enumerate}
\item The coefficients $c$ and $b^d$ obey
\begin{align}
\label{eq:ZerothOrderTermUpperBound}
c(t,x) &\leq K, \quad \forall\, (t,x)\in \overline{\HH}_{\infty},
\\
\label{eq:CoeffBD}
b^d(t,x',0) &\geq \nu, \quad \forall\, (t,x')\in [0,\infty)\times\RR^{d-1}.
\end{align}
\item On $\overline{\HH}_{2,\infty}$ (that is, near $x_d=0$), we require that
\begin{gather}
\label{eq:NonDegeneracyNearBoundary}
\sum_{i,j=1}^d a^{ij} (t,x) \eta_i \eta_j \geq \delta |\eta|^2, \quad \forall\, \eta \in \mathbb{R}^d, \quad\forall\, (t,x)\in \overline{\HH}_{2,\infty},
\\
\label{eq:BoundednessNearBoundary}
\max_{1\leq i,j\leq d}\|a^{ij}\|_{C(\overline{\HH}_{2,\infty})} + \max_{1\leq i\leq d}\|b^i\|_{C(\overline{\HH}_{2,\infty})} + \|c\|_{C(\overline{\HH}_{2,\infty})} \leq K,
\end{gather}
and, for all $P_1, P_2 \in \overline{\HH}_{2,\infty}$ such that $P_1 \neq P_2$ and $s(P_1, P_2) \leq 1$,
\begin{equation}
\label{eq:LocalHolderS}
\begin{aligned}
\max_{1\leq i,j\leq d}\frac{|a^{ij}(P_1)-a^{ij}(P_2)|}{s^{\alpha}(P_1,P_2)} &\leq K,
\\
\max_{1\leq i\leq d}\frac{|b^i(P_1)-b^i(P_2)|}{s^{\alpha}(P_1,P_2)} &\leq K,
\\
\frac{|c(P_1)-c(P_2)|}{s^{\alpha}(P_1,P_2)} &\leq K.
\end{aligned}
\end{equation}
\item On $\overline{\HH}_{\infty} \less \HH_{2,\infty}$ (that is, farther away from $x_d=0$), we require that
\begin{align}
\label{eq:NonDegeneracyInterior}
\sum_{i,j=1}^d x_da^{ij} (t,x) \eta_i \eta_j \geq \delta |\eta|^2, \quad \forall\, \eta \in \mathbb{R}^d, \quad\forall\, (t,x)\in \overline{\HH}_{\infty} \less \HH_{2,\infty},
\end{align}
and, for all $P_1, P_2 \in \overline{\HH}_{\infty} \less \HH_{2,\infty}$ such that $P_1 \neq P_2$ and $\rho(P_1, P_2) \leq 1$,
\begin{equation}
\label{eq:LocalHolderRho}
\begin{aligned}
\max_{1\leq i,j\leq d}\frac{|x_d^1a^{ij}(P_1)-x_d^2a^{ij}(P_2)|}{\rho^{\alpha}(P_1,P_2)} &\leq K,
\\
\max_{1\leq i\leq d}\frac{|b^i(P_1)-b^i(P_2)|}{\rho^{\alpha}(P_1,P_2)} &\leq K,
\\
\frac{|c(P_1)-c(P_2)|}{\rho^{\alpha}(P_1,P_2)} &\leq K.
\end{aligned}
\end{equation}
\end{enumerate}
\end{assump}

\begin{rmk}[Local H\"older conditions on the coefficients]
The local H\"older conditions \eqref{eq:LocalHolderS} and \eqref{eq:LocalHolderRho} are similar to those in \cite[Hypothesis 2.1]{Krylov_Priola_2010}.
\end{rmk}

\begin{rmk}[Linear growth of the coefficients of the parabolic differential operator]
\label{rmk:LinearGrowth}
Conditions \eqref{eq:BoundednessNearBoundary} and \eqref{eq:LocalHolderRho} imply that the coefficients $x_d a^{ij}(t,x)$, $b^i(t,x)$ and $c(t,x)$ can have at most linear growth in $x$. In particular, we may choose the constant $K$ large enough such that
\begin{equation}
\label{eq:LinearGrowth}
\sum_{i,j=1}^d|x_da^{ij}(t,x)| + \sum_{i=1}^d|b^i(t,x)| + |c(t,x)| \leq K(1+|x|), \quad \forall\, (t,x) \in \overline{\HH}_{\infty}.
\end{equation}
\end{rmk}

\section{Existence, uniqueness and regularity of the inhomogeneous initial value problem}
\label{sec:ExistenceUniquenessRegularityInhomoInitValProblem}
In this section, we prove Theorem \ref{thm:MainExistenceUniquenessPDE}. We begin by reviewing the boundary properties and establishing the interpolation inequalities (Lemma \ref{lem:InterpolationIneqS}) suitable
for functions in $C^{2+\alpha}_{s}(\overline{\HH}_T)$. Then, we prove two versions of the maximum principle (Proposition \ref{prop:MaximumPrinciple}) that, when combined with the a priori local H\"older estimates at the boundary (Theorem \ref{thm:AprioriBoundaryEst}) and in the interior (Proposition \ref{prop:InteriorEst}), allow us to prove Theorem \ref{thm:MainExistenceUniquenessPDE}.

\subsection{Boundary properties of functions in Daskalopoulos-Hamilton-Koch H{\"o}lder spaces}
\label{subsec:BoundaryPropertiesDHKHolderFunctions}
The following result was proved as \cite[Proposition I.12.1]{DaskalHamilton1998} when $d=2$ and the proof when $d\geq 2$ follows by a similar argument; we include a proof for the cases omitted in \cite[Proposition I.12.1]{DaskalHamilton1998}.

\begin{lem}[Boundary properties of functions in Daskalopoulos-Hamilton-Koch H{\"o}lder spaces]
\label{lem:PropSecondOrderDeriv}
\cite[Proposition I.12.1]{DaskalHamilton1998}
Let $u \in C^{2+\alpha}_{s,\loc}(\overline{\HH}_T)$. Then, for all $\bar P\in [0,T]\times\partial \HH$,
\begin{equation}
\label{eq:PropSecondOrderDeriv}
\lim_{\overline{\HH}_T \ni P \rightarrow \bar P} x_d u_{x_ix_j}( P) = 0, \quad
 i,j=1,\ldots,d.
\end{equation}
\end{lem}

\begin{proof}
First, we consider the case $1 \leq i,j\leq d-1$. Because the seminorm $[x_du_{x_ix_j}]_{C^{\alpha}_{s,\loc}(\overline{\HH}_T)}$ is finite, the function $x_du_{x_ix_j}$ is uniformly continuous on compact subsets of $\overline{\HH}_T$, and so, the limit in \eqref{eq:PropSecondOrderDeriv} exists. We assume, to obtain a contradiction, that
\begin{equation}
\label{eq:PositiveLimit}
\lim_{\overline{\HH}_T \ni P \rightarrow \bar P} x_d u_{x_ix_j}( P) = a \neq 0,
\end{equation}
and we can further assume, without loss of generality, that this limit is positive. Then, there is a constant, $\eps>0$, such that for all $P=(t,x',x_d)\in\overline{\HH}_T$ satisfying
\begin{equation}
\label{eq:BallEps}
0<x_d<\eps, \quad |t-\bar t|<\eps, \quad |x'-\bar x'| < \eps,
\end{equation}
we have
\begin{equation}
\label{eq:PropSecondOrderDeriv1}
\begin{aligned}
\frac{a}{2x_d} \leq u_{x_ix_j}(t,x',x_d).
\end{aligned}
\end{equation}
Let $P_1=(t,x^1)$ and $P_2=(t, x^2)$ be points satisfying \eqref{eq:BallEps} and such that all except the $x_i$-coordinates are identical. Then, by integrating \eqref{eq:PropSecondOrderDeriv1} with respect to $x_i$, we obtain
\[
\frac{a(x_i^2-x_i^1)}{2x_d} \leq u_{x_j}(P_2)-u_{x_j}(P_1),
\]
and thus,
\begin{equation}
\label{eq:PropSecondOrderDeriv2}
\begin{aligned}
\frac{a(x_i^2-x_i^1)}{2x_d s^{\alpha}(P_1,P_2)} \leq \frac{u_{x_j}(P_2)-u_{x_j}(P_1)}{s^{\alpha}(P_1,P_2)}.
\end{aligned}
\end{equation}
We can choose $P_1$, $P_2$ such that $x_i^2-x_i^1=\eps/2$, for all $0<x_d <\eps/2$. Then, by taking the limit as $x_d$ goes to zero, the left-hand side of \eqref{eq:PropSecondOrderDeriv2} diverges, while the right-hand side is finite since $[u_{x_j}]_{C^{\alpha}_s(\HH_T)}$ is bounded. This contradicts \eqref{eq:PositiveLimit} and so \eqref{eq:PropSecondOrderDeriv} holds.

The case where $i=d$ or $j=d$ can be treated as in the proof of \cite[Proposition I.12.1]{DaskalHamilton1998}.
\end{proof}

Next, we establish the analogue of \cite[Theorem~8.8.1]{Krylov_LecturesHolder} for the H\"older space $C^{2+\alpha}_s(\overline{\HH}_T)$.

\begin{lem} [Interpolation inequalities for Daskalopoulos-Hamilton-Koch H{\"o}lder spaces]
\label{lem:InterpolationIneqS}
Let $R>0$. Then there are positive constants $m=m(d,\alpha)$ and $C=C(T, R, d, \alpha)$ such that for any $u \in C^{2+\alpha}_s(\overline{\HH}_T)$ with compact support in $[0,\infty)\times\bar{B}_R(x^0)$, for some $x^0 \in \partial \HH$, and any $\eps \in (0,1)$, we have
\begin{align}
\label{eq:InterpolationIneqS1}
\|u\|_{C^{\alpha}_s(\overline{\HH}_T)} &\leq \eps \|u\|_{C^{2+\alpha}_s(\overline{\HH}_T)} + C \eps^{-m} \|u\|_{C(\overline{\HH}_T)},\\
\label{eq:InterpolationIneqS2}
\|u_{x_i}\|_{C(\overline{\HH}_T)} &\leq \eps \|u\|_{C^{2+\alpha}_s(\overline{\HH}_T)} + C \eps^{-m} \|u\|_{C(\overline{\HH}_T)},\\
\label{eq:InterpolationIneqS3}
\|x_d u_{x_i}\|_{C^{\alpha}_s(\overline{\HH}_T)} &\leq \eps \|u\|_{C^{2+\alpha}_s(\overline{\HH}_T)} + C \eps^{-m} \|u\|_{C(\overline{\HH}_T)},\\
\label{eq:InterpolationIneqS4}
\|x_d u_{x_i x_j}\|_{C(\overline{\HH}_T)} &\leq \eps \|u\|_{C^{2+\alpha}_s(\overline{\HH}_T)} + C \eps^{-m} \|u\|_{C(\overline{\HH}_T)}.
\end{align}
\end{lem}

\begin{rmk}
Notice that Lemma \ref{lem:InterpolationIneqS} does not establish the analogue of \cite[Inequality (8.8.4)]{Krylov_LecturesHolder}, that is,
\[
[u_{x_i}]_{C^{\alpha}_{\rho}(\overline{\HH}_T)} \leq \eps \|u\|_{C^{2+\alpha}_{\rho}(\overline{\HH}_T)} + C \eps^{-m} \|u\|_{C(\overline{\HH}_T)}.
\]
This is replaced by the weighted inequality \eqref{eq:InterpolationIneqS3}.
\end{rmk}

\begin{proof} [Proof of Lemma \ref{lem:InterpolationIneqS}]
We consider $\eta\in(0,1)$, to be suitably chosen during the proofs of each of the desired inequalities.

\begin{step}[Proof of inequality \eqref{eq:InterpolationIneqS1}]
We only need to show that the first inequality \eqref{eq:InterpolationIneqS1} holds for the seminorm $[u]_{C^{\alpha}_s(\overline{\HH}_T)}$. It is enough to consider differences, $u(P_1)-u(P_2)$, where all except one of the coordinates of the points $P_1, P_2 \in \overline{\HH}_T$ are identical. We outline the proof when the $x_i$-coordinates of $P_1$ and $P_2$ differ, but the case of the $t$-coordinate can be treated in a similar manner. We consider two situations: $|x^1_i-x^2_i|\leq \eta$ and $|x^1_i-x^2_i|> \eta$.

\begin{case}[Points with $x_i$-coordinates close together]
Assume $|x^1_i-x^2_i|\leq \eta$. We have
\begin{equation}
\label{eq:InterpolationIneqS1_1}
\begin{aligned}
|u(P_1)-u(P_2)| &\leq |x^1_i-x^2_i| \|u_{x_i}\|_{C(\overline{\HH}_T)} \\
                &\leq \eta \frac{|x^1_i-x^2_i|}{\eta} \|u\|_{C^{2+\alpha}_s(\overline{\HH}_T)}\\
                &\leq \eta \left(\frac{|x^1_i-x^2_i|}{\eta} \right)^{\alpha} \|u\|_{C^{2+\alpha}_s(\overline{\HH}_T)}\\
                &\leq \eta^{1-\alpha} \left(2\sqrt{x_d}+\sqrt{|x^1_i-x^2_i|}\right)^{\alpha} s^{\alpha}(P_1,P_2)\|u\|_{C^{2+\alpha}_s(\overline{\HH}_T)},
\end{aligned}
\end{equation}
where in the last line we used the fact that, by \eqref{eq:CycloidalMetric},
\begin{equation}
\label{eq:FormulaS1}
s(P_1,P_2) = \frac{|x^1_i-x^2_i|}{ 2\sqrt{x_d} + \sqrt{|x^1_i-x^2_i|}}.
\end{equation}
Because $u$ has compact support in the spatial variable, we obtain in \eqref{eq:InterpolationIneqS1_1} that there exists a positive constant $C=C(\alpha, R)$ such that
\begin{equation}
\label{eq:InterpolationIneqS1_2}
\begin{aligned}
\frac{|u(P_1)-u(P_2)|}{s^{\alpha}(P_1,P_2)} \leq C\eta^{1-\alpha}  \|u\|_{C^{2+\alpha}_s(\overline{\HH}_T)},
\end{aligned}
\end{equation}
which concludes this case.
\end{case}

\begin{case}[Points with $x_i$-coordinates farther apart]
Assume $|x^1_i-x^2_i|> \eta$. By \eqref{eq:FormulaS1}, we have
\[
1<\left(\frac{|x^1_i-x^2_i|}{\eta} \right)^{\alpha} = \eta^{-\alpha}\left(2\sqrt{x_d}+\sqrt{|x^1_i-x^2_i|}\right)^{\alpha} s^{\alpha}(P_1,P_2).
\]
Because it suffices to consider points $P_1$ and $P_2$ in the support of $u$, there is a positive constant $C$, depending at most on $\alpha$ and $R$, such that
\[
1 \leq C\eta^{-\alpha}  s^{\alpha}(P_1,P_2).
\]
Therefore,
\[
|u(P_1)-u(P_2)| \leq 2 \|u\|_{C(\overline{\HH}_T)} \leq  C\eta^{-\alpha} s^{\alpha}(P_1,P_2) \|u\|_{C(\overline{\HH}_T)},
\]
which is equivalent to
\begin{equation}
\label{eq:InterpolationIneqS1_3}
\begin{aligned}
\frac{|u(P_1)-u(P_2)|}{s^{\alpha}(P_1,P_2)} &\leq C\eta^{-\alpha} \|u\|_{C^{0}(\overline{\HH}_T)},
\end{aligned}
\end{equation}
which concludes this case.
\end{case}
By combining  \eqref{eq:InterpolationIneqS1_2} and \eqref{eq:InterpolationIneqS1_3}, we obtain
\[
[u]_{C^{\alpha}_s(\overline{\HH}_T)} \leq C \eta^{1-\alpha} \|u\|_{C^{2+\alpha}_s(\overline{\HH}_T)} + C \eta^{-\alpha} \|u\|_{C^)(\overline{\HH}_T)}.
\]
Since $\eps\in(0,1)$, we may choose $\eta\in(0,1)$ such that $\eps=C\eta^{1-\alpha}$. The preceding inequality then gives \eqref{eq:InterpolationIneqS1}.
\end{step}

\begin{step}[Proof of inequality \eqref{eq:InterpolationIneqS2}]
Let $P\in \overline{\HH}_T$. Then, for any $\eta>0$, we have
\begin{align*}
|u_{x_i}(P)| &\leq \left|u_{x_i}(P)-\eta^{-1} \left(u(P+\eta e_i) - u(P)\right)\right| + 2 \eta^{-1} \|u\|_{C(\overline{\HH}_T)}\\
             &=    |u_{x_i}(P)- u_{x_i}(P+\eta\theta e_i)| + 2 \eta^{-1} \|u\|_{C(\overline{\HH}_T)}\\
             &=    \frac{|u_{x_i}(P)- u_{x_i}(P+\eta\theta e_i)|}
             {s^{\alpha}(P,P+\eta\theta e_i)}s^{\alpha}(P,P+\eta\theta e_i)
                    + 2 \eta^{-1} \|u\|_{C(\overline{\HH}_T)},
\end{align*}
for some constant $\theta \in [0,1]$.  Using
\begin{equation}
\label{eq:ADDS}
s(P,P+\eta\theta e_i) \leq  \eta^{1/2},\quad \forall\, P\in \overline{\HH}_T,
\end{equation}
we have
\begin{equation}
\label{eq:PartC}
|u_{x_i}(P)| \leq \eta^{\alpha/2}[u_{x_i}]_{C^{\alpha}_s(\overline{\HH}_T)}
                    + 2 \eta^{-1} \|u\|_{C(\overline{\HH}_T)}, \quad \forall\, P\in \overline{\HH}_T.
\end{equation}
Since $\eps\in(0,1)$, we may choose $\eta\in(0,1)$ such that $\eps=\eta^{\alpha/2}$. Then \eqref{eq:InterpolationIneqS2} follows from \eqref{eq:PartC}.
\end{step}

\begin{step}[Proof of inequality \eqref{eq:InterpolationIneqS3}]
Because $u$ has compact support in the spatial variable, then \eqref{eq:InterpolationIneqS2} gives, for some positive constant $C=C(\alpha,R)$,
\begin{equation}
\label{eq:InterpolationIneqS2_1}
\begin{aligned}
\|x_du_{x_i}\|_{C(\overline{\HH}_T)} \leq C \eps \|u\|_{C^{2+\alpha}_s(\overline{\HH}_T)} +  C \eps^{-m} \|u\|_{C(\overline{\HH}_T)}.
\end{aligned}
\end{equation}
This gives the desired bound in \eqref{eq:InterpolationIneqS3} for the term $\|x_du_{x_i}\|_{C(\overline{\HH}_T)}$.  It remains to prove the estimate \eqref{eq:InterpolationIneqS3} for the H\"older seminorm $[x_d u_{x_i}]_{C^{\alpha}_s(\overline{\HH}_T)}$.  As in the proof of \eqref{eq:InterpolationIneqS1}, it suffices to consider the differences $x_d^1u_{x_i}(P_1)-x_d^2u_{x_i}(P_2)$, where all except one of the coordinates of the points $P_1, P_2 \in \overline{\HH}_T$ are identical.

First, we consider the case when only the $x_d$-coordinates of the points $P_1$ and $P_2$ differ. We denote $P_k=(t,x',x_d^k)$, $k=1,2$.

\setcounter{case}{0}
\begin{case}[Points with $x_d$-coordinates close together] Assume $|x^1_d-x^2_d|\leq \eta$. Using
\[
(x_d u_{x_i})_{x_d} = x_d u_{x_i x_d} + u_{x_i}
\]
and the mean value theorem, there is a point $P^*$ on the line segment connecting $P_1$ and $P_2$ such that,
\[
x_d^1u_{x_i}(P_1)-x_d^2u_{x_i}(P_2)  = \left(x_d^* u_{x_i x_d}(P^*) + u_{x_i}(P^*)\right) (x_d^1-x_d^2),
\]
and so,
\begin{equation*}
\begin{aligned}
|x_d^1u_{x_i}(P_1)-x_d^2u_{x_i}(P_2)| &\leq \eta \left(\frac{|x^1_d-x^2_d|}{\eta} \right)^{\alpha} \|u\|_{C^{2+\alpha}_s(\overline{\HH}_T)}\\
                &\leq \eta^{1-\alpha} \left(\sqrt{x^1_d}+\sqrt{x^2_d}+\sqrt{|x^1_d-x^2_d|}\right)^{\alpha} s^{\alpha}(P_1,P_2)\|u\|_{C^{2+\alpha}_s(\overline{\HH}_T)}.
\end{aligned}
\end{equation*}
Because $u$ has compact support in the spatial variable, there is a positive constant $C=C(\alpha, R)$ such that
\begin{equation}
\label{eq:InterpolationIneqS2_2}
\begin{aligned}
\frac{|x_d^1u_{x_i}(P_1)-x_d^2u_{x_i}(P_2)|}{s^{\alpha}(P_1,P_2)} &\leq C \eta^{1-\alpha} \|u\|_{C^{2+\alpha}_s(\overline{\HH}_T)},
\end{aligned}
\end{equation}
which concludes this case.
\end{case}

\begin{case}[Points with $x_d$-coordinates
farther apart] Assume $|x^1_d-x^2_d|> \eta$. We have
\begin{equation*}
\begin{aligned}
\frac{|x_d^1u_{x_i}(P_1)-x_d^2u_{x_i}(P_2)|}{s^{\alpha}(P_1,P_2)}
&\leq 2\frac{\|x_d u_{x_i}\|_{C(\overline{\HH}_T)}}{|x^1_d-x^2_d|^{\alpha}} \left(\sqrt{x^1_d} + \sqrt{x^2_d} + \sqrt{|x^1_d-x^2_d|}\right)^{\alpha}\\
&\leq C\eta^{-\alpha} \|x_d u_{x_i}\|_{C(\overline{\HH}_T)}.
\end{aligned}
\end{equation*}
Since $\eps\in(0,1)$, we may choose $\eta$ such that $\eps=\eta^{\alpha+1}$ in \eqref{eq:InterpolationIneqS2_1}. We obtain
\begin{equation}
\label{eq:InterpolationIneqS2_xdfarapart}
\frac{|x_d^1u_{x_i}(P_1)-x_d^2u_{x_i}(P_2)|}{s^{\alpha}(P_1,P_2)}
\leq C \eta \|u\|_{C^{2+\alpha}_s(\overline{\HH}_T)} + C \eta^{-m(1+\alpha)-\alpha} \|u\|_{C(\overline{\HH}_T)},
\end{equation}
which concludes this case.
\end{case}
Combining \eqref{eq:InterpolationIneqS2_2} and \eqref{eq:InterpolationIneqS2_xdfarapart} gives
\begin{equation}
\label{eq:InterpolationIneqS2_3}
\begin{aligned}
\frac{|x_d^1u_{x_i}(P_1)-x_d^2u_{x_i}(P_2)|}{s^{\alpha}(P_1,P_2)}
\leq C\eta^{1-\alpha} \|u\|_{C^{2+\alpha}_s(\overline{\HH}_T)} + C\eta^{-m(1+\alpha)-\alpha} \|u\|_{C(\overline{\HH}_T)}.
\end{aligned}
\end{equation}
A similar argument, when only the $x_i$-coordinates of the points $P_1$ and $P_2$ differ, $1\leq i\leq d-1$, also yields \eqref{eq:InterpolationIneqS2_3}.

Next, we consider the case when only the $t$-coordinates of the points $P_1$ and $P_2$ differ. We denote $P_k=(x, t_k)$, $k=1,2$. We shall only describe the proof of the interpolation
inequality for $u_{x_i}$ when $i\neq d$, as the case $i=d$ follows by a similar argument. We denote $\delta=\sqrt{|t_1-t_2|}$.

\begin{case}[Points with $t$-coordinates close together] Assume $|t_1-t_2|<\eta$. We have
\begin{equation*}
\begin{aligned}
|u_{x_i}(P_1) - u_{x_i}(P_2)|
&\leq
\left|u_{x_i}(x,t_1) - \frac{1}{\delta} \left( u(x+\delta e_i,t_1) - u(x,t_1)\right)\right| \\
&\quad +\left|u_{x_i}(x,t_2) - \frac{1}{\delta} \left( u(x+\delta e_i,t_2) - u(x,t_2)\right)\right| \\
&\quad+ \frac{1}{\delta} |u(x+\delta e_i,t_1) - u(x+\delta e_i,t_2)| +
   \frac{1}{\delta} |u(x,t_1) - u(x,t_2)|.
\end{aligned}
\end{equation*}
By the mean value theorem, there are points $P^*_k\in\overline{\HH}_T$, $k=1,2$, such that
\begin{equation*}
\begin{aligned}
|u_{x_i}(P_1) - u_{x_i}(P_2)|
&= |u_{x_i}(x,t_1) - u_{x_i}(x+\theta_1\delta e_i,t_1)| +
   |u_{x_i}(x,t_2) - u_{x_i}(x+\theta_2\delta e_i,t_2)| \\
&\quad+  \frac{|t_1-t_2|}{\delta} |u_t(x+\delta e_i,t^*_1)| +
    \frac{|t_1-t_2|}{\delta} |u_t(x,t^*_2)|\\
&\leq |u_{x_ix_i}(P_1^*,t_1)| \delta + |u_{x_ix_i}(P_2^*,t_2) | \delta\\
      &\quad + \frac{|t_1-t_2|}{\delta} |u_t(x+\delta e_i,t^*_1)| +
       \frac{|t_1-t_2|}{\delta} |u_t(x,t^*_2)|.
\end{aligned}
\end{equation*}
Notice that $s(P_1,P_2) = \sqrt{|t_1-t_2|}=\delta$ and so, by multiplying the preceding inequality by $x_d$ and using the fact that $u$ has compact support, we obtain
\begin{equation*}
\frac{|x_du_{x_i}(P_1) - x_du_{x_i}(P_2)|}{s^{\alpha}(P_1,P_2)}
\leq 2\|x_d u_{x_ix_i}\|_{C^{0}(\overline{\HH}_T)} |t_1-t_2|^{\frac{1-\alpha}{2}}+
      2|t_1-t_2|^{1-\frac{1+\alpha}{2}}  \|x_du_t\|_{C^{0}(\overline{\HH}_T)},
\end{equation*}
and thus
\begin{equation}
\label{eq:InterpolationIneqS2_4}
\begin{aligned}
\frac{|x_du_{x_i}(P_1) - x_du_{x_i}(P_2)|}{s^{\alpha}(P_1,P_2)}
\leq C\eta^{\frac{1-\alpha}{2}}  \|u\|_{C^{2+\alpha}_s(\overline{\HH}_T)},
\end{aligned}
\end{equation}
where $C$ is a positive constant depending only on $R$.
\end{case}

\begin{case}[Points with $t$-coordinates
farther apart] Assume  $|t_1-t_2|\geq\eta$. This case is easier, as usual, because
\begin{equation}
\label{eq:InterpolationIneqS2_5}
\begin{aligned}
\frac{|x_du_{x_i}(P_1) - x_du_{x_i}(P_2)|}{s^{\alpha}(P_1,P_2)}
&\leq& 2 \eta^{-\frac{\alpha}{2}} \|x_du_{x_i}\|_{C(\overline{\HH}_T)},
\end{aligned}
\end{equation}
which concludes this case.
\end{case}
By combining inequalities \eqref{eq:InterpolationIneqS2_4} and \eqref{eq:InterpolationIneqS2_5}, we obtain
\begin{equation}
\label{eq:InterpolationIneqS2_6}
\begin{aligned}
\frac{|x_du_{x_i}(P_1) - x_du_{x_i}(P_2)|}{s^{\alpha}(P_1,P_2)}
&\leq C\eta^{\frac{1-\alpha}{2}}  \|u\|_{C^{2+\alpha}_s(\overline{\HH}_T)} + 2 \eta^{-\frac{\alpha}{2}} \|x_du_{x_i}\|_{C(\overline{\HH}_T)}.
\end{aligned}
\end{equation}
By \eqref{eq:InterpolationIneqS2_3} and \eqref{eq:InterpolationIneqS2_6}, we have
\[
\left[x_du_{x_i}\right]_{C^{\alpha}_s(\overline{\HH}_T)}
\leq C\eta^{\alpha_0}  \|u\|_{C^{2+\alpha}_s(\overline{\HH}_T)} + 2 \eta^{-m_0} \|x_du_{x_i}\|_{C(\overline{\HH}_T)},
\]
where $\alpha_0:=\min\{\alpha, 1-\alpha, (1-\alpha)/2\}$ and $m_0:=4+\alpha$. Without loss of generality, we may assume $C\geq1$. Since $\eps\in(0,1)$, we may choose $\eta\in (0,1)$ such that $\eps=C\eta^{\alpha_0}$ in the preceding inequality, and so we obtain the estimate \eqref{eq:InterpolationIneqS3} for $[x_du_{x_d}]_{C^{\alpha}_s(\HH_T)}$. This concludes the proof of \eqref{eq:InterpolationIneqS3}.
\end{step}

\begin{step}[Proof of inequality \eqref{eq:InterpolationIneqS4}]
For any $P=(t,x)\in \overline{\HH}_T$, we can find $\theta \in [0,1]$ such that
\begin{equation*}
\begin{aligned}
|x_du_{x_ix_j}(P)|&\leq\left|x_du_{x_ix_j}(P) - \left(x_du_{x_i}(P+\eta e_j) -x_d u_{x_i}(P) \right)\right|
                      + 2 \|x_du_{x_i}\|_{C(\overline{\HH}_T)},
\end{aligned}
\end{equation*}
and thus
\begin{equation}
\label{eq:ADD1}
\begin{aligned}
|x_du_{x_ix_j}(P)|&\leq  \left|x_du_{x_ix_j}(P) - x_du_{x_ix_j}(P+\theta \eta e_j) \right| + 2 \|x_du_{x_i}\|_{C(\overline{\HH}_T)},
\end{aligned}
\end{equation}
where $1\leq i,j\leq d$. If $j \neq d$, we have
\begin{equation*}
\begin{aligned}
|x_du_{x_ix_j}(P)|&\leq\frac{|x_du_{x_ix_j}(P) - x_du_{x_ix_j}(P+\theta \eta e_j)|}{ s^{\alpha}(P,P+\theta\eps e_j)}
                          s^{\alpha}(P,P+\theta\eta e_j) + 2 \|x_du_{x_i}\|_{C(\overline{\HH}_T)} \\
                    &\leq  C \eta^{\alpha/2}[x_du_{x_ix_j}]_{C^{\alpha}_s(\overline{\HH}_T)}
                             + 2 \|x_du_{x_i}\|_{C(\overline{\HH}_T)}, \hbox{  (by \eqref{eq:ADDS}).}
\end{aligned}
\end{equation*}
Because $\eps\in(0,1)$, we may choose $\eta\in(0,1)$ such that $\eps=C\eta^{\alpha/2}$ in the preceding inequality and combining the resulting inequality with \eqref{eq:InterpolationIneqS3}, we see that the estimate \eqref{eq:InterpolationIneqS4} for $\|x_du_{x_ix_j}\|_{C(\overline{\HH}_T)}$ holds for all $j \neq d$.

Next, we consider the case $j=d$. For brevity, we denote $P'=P+\theta\eta e_d=(t,x',x'_d)$ and $P^{''}=(t,x',0) $. We consider two distinct cases depending on whether $\eta < x'_d/2$ or $\eta \geq x'_d/2$.

\setcounter{case}{0}
\begin{case}[Points with $x_d$-coordinates farther apart] Assume $\eta < x'_d/2$. By \eqref{eq:ADD1}, we obtain
\begin{equation}
\label{eq:InterpolationIneq4_11}
\begin{aligned}
|x_du_{x_ix_d}(P)| &\leq  \frac{|x_du_{x_ix_d}(P) - x'_du_{x_ix_d}(P')|}{ s^{\alpha}(P,P')} s^{\alpha}(P,P')
\\
&\quad + |(x'_d-x_d) u_{x_ix_d}(P')| + 2 \|x_du_{x_i}\|_{C(\overline{\HH}_T)},
\end{aligned}
\end{equation}
and so, using \eqref{eq:ADDS} and the fact that $|x'_d-x_d|\leq \eta$, by definitions of points $P$ and $P'$,
\begin{equation*}
|x_du_{x_ix_d}(P)| \leq  \eta^{\alpha/2}[x_du_{x_ix_d}]_{C^{\alpha}_s(\overline{\HH}_T)}
                             + \frac{\eta}{x'_d}|x'_d u_{x_ix_d}(P')|
                             + 2 \|x_du_{x_i}\|_{C(\overline{\HH}_T)},
\end{equation*}
which gives, by our assumption that $\eta<x'_d/2$,
\begin{equation}
\label{eq:InterpolationIneqS4_1}
|x_du_{x_ix_d}(P)| \leq  \eta^{\alpha/2}[x_du_{x_ix_d}]_{C^{\alpha}_s(\overline{\HH}_T)}
                             + \frac{1}{2}\|x_d u_{x_ix_d}\|_{C(\overline{\HH}_T)}
                             + 2\|x_du_{x_i}\|_{C(\overline{\HH}_T)}.
\end{equation}
As \eqref{eq:InterpolationIneqS4_1} holds for all $P\in\overline{\HH}_T$, we have
\begin{equation*}
\begin{aligned}
\|x_du_{x_ix_d}\|_{C(\overline{\HH}_T)}
&\leq   \frac{1}{2}\|x_d u_{x_ix_d}\|_{C(\overline{\HH}_T)} + \eta^{\alpha/2}[x_du_{x_ix_d}]_{C^{\alpha}_s(\overline{\HH}_T)}
+ 2\|x_du_{x_i}\|_{C(\overline{\HH}_T)},
\end{aligned}
\end{equation*}
or
\begin{equation}
\label{eq:Adaugare1}
\begin{aligned}
\|x_du_{x_ix_d}\|_{C(\overline{\HH}_T)}
&\leq  2\eta^{\alpha/2}[x_du_{x_ix_d}]_{C^{\alpha}_s(\overline{\HH}_T)}
+ 4\|x_du_{x_i}\|_{C(\overline{\HH}_T)},
\end{aligned}
\end{equation}
which concludes this case.
\end{case}

\begin{case}[Points with $x_d$-coordinates close together] Assume $\eta \geq x'_d/2$. Recall that $x'_d = x_d + \theta\eta$, for some $\theta\in [0,1]$, so that $|x'_d-x_d| \leq x'_d$. From Lemma \ref{lem:PropSecondOrderDeriv}, we have
\[
x_d u_{x_ix_d} \rightarrow 0, \hbox{  as  } x_d \rightarrow 0.
\]
Therefore, we obtain
\begin{align*}
|(x'_d-x_d) u_{x_ix_d}(P')| &\leq |x'_du_{x_ix_d}(P')|
= \frac{|x'_du_{x_ix_d}(P') - 0|}{s^{\alpha}(P',P^{''})} s^{\alpha}(P',P^{''}) \\
&\leq [x_d u_{x_ix_d}]_{C^{\alpha}_s(\overline{\HH}_T)} (2\eta)^{\alpha/2},
\end{align*}
where the second inequality follows from the fact that
\[
s(P',P^{''}) \leq \sqrt{x'_d}\leq\sqrt{2\eta}.
\]
By a calculation similar to that which led to \eqref{eq:InterpolationIneq4_11}, we obtain
\begin{equation*}
\begin{aligned}
|x_du_{x_ix_d}(P)| &\leq  \frac{|x_du_{x_ix_d}(P) - x'_du_{x_ix_d}(P')|}{ s^{\alpha}(P,P')} s^{\alpha}(P,P')
\\
&\quad + |(x'_d-x_d) u_{x_ix_d}(P')| + 2 \|x_du_{x_i}\|_{C(\overline{\HH}_T)},
\end{aligned}
\end{equation*}
and hence
\begin{equation}
\label{eq:InterpolationIneqS4_2}
\begin{aligned}
|x_du_{x_ix_d}(P)| \leq  C \eta^{\alpha/2}[x_du_{x_ix_d}]_{C^{\alpha}_s(\overline{\HH}_T)}
                             +(2\eta) ^{\alpha/2} [x_du_{x_ix_d}]_{C^{\alpha}_s(\overline{\HH}_T)}
                             +2\|x_du_{x_i}\|_{C(\overline{\HH}_T)},
\end{aligned}
\end{equation}
which concludes this case.
\end{case}

By combining inequalities \eqref{eq:InterpolationIneqS4_1} and \eqref{eq:InterpolationIneqS4_2}, we obtain, for all $P\in\overline{\HH}_T$,
\[
|x_du_{x_ix_d}(P)| \leq \frac{1}{2}\|x_d u_{x_ix_d}\|_{C(\overline{\HH}_T)}+C \eta^{\alpha/2} [x_du_{x_ix_d}]_{C^{\alpha}_s(\overline{\HH}_T)} +  2\|x_du_{x_i}\|_{C(\overline{\HH}_T)},
\]
which is equivalent to
\[
\|x_du_{x_ix_d}\|_{C(\overline{\HH}_T)}
\leq \frac{1}{2}\|x_d u_{x_ix_d}\|_{C(\overline{\HH}_T)}
     +C \eta^{\alpha/2} [x_du_{x_ix_d}]_{C^{\alpha}_s(\overline{\HH}_T)} +  2\|x_du_{x_i}\|_{C(\overline{\HH}_T)}.
\]
Rearranging terms yields
\begin{equation}
\label{eq:Adaugare2}
\|x_du_{x_ix_d}\|_{C(\overline{\HH}_T)} \leq 2C \eta^{\alpha/2} [x_du_{x_ix_d}]_{C^{\alpha}_s(\overline{\HH}_T)}
+  4\|x_du_{x_i}\|_{C(\overline{\HH}_T)}.
\end{equation}
Since $\eps\in (0,1)$, we may choose $\eta \in (0,1)$ in \eqref{eq:Adaugare1} and \eqref{eq:Adaugare2} such that
$\eps= 4(C+1)\eta^{\alpha/2}$ and so we obtain
\begin{equation*}
\|x_du_{x_ix_d}\|_{C(\overline{\HH}_T)} \leq
\frac{\eps}{2} [x_du_{x_ix_d}]_{C^{\alpha}_s(\overline{\HH}_T)}
+  4\|x_du_{x_i}\|_{C(\overline{\HH}_T)}.
\end{equation*}
Combining the preceding inequality with \eqref{eq:InterpolationIneqS3} applied with $\eps$ replaced by $\eps/8$, we conclude that \eqref{eq:InterpolationIneqS4} holds.
\end{step}
This completes the proof of Lemma \ref{lem:InterpolationIneqS}.
\end{proof}

\subsection{Maximum principle and its applications}
\label{subsec:MaximumPrinciple}
In this subsection, we prove a variant of the classical maximum principle (see \cite[Section 8.1]{Krylov_LecturesHolder} and \cite[Theorem I.3.1]{DaskalHamilton1998}) for parabolic operators, $L$, of the form \eqref{eq:Generator}.

\begin{lem}[Maximum principle]
\label{lem:MaximumPrinciple}
We relax the requirements stated in Assumption \ref{assump:Coeff} on the coefficients $a=(a^{ij}), b=(b^i), c$ of the operator $L$ in \eqref{eq:Generator} to those stated here. Require that the coefficients $a^{ij}$, $b^i$, $c$ be defined on $(0,T]\times\overline\HH$, and the matrix $(a^{ij})$ is non-negative definite on $(0,T]\times\overline\HH$, and $b^d\geq 0$ when $x_d=0$, and $c$ obeys \eqref{eq:ZerothOrderTermUpperBound}, and
\begin{equation}
\label{eq:CoeffMaxPrincTraceGrowth}
\tr(x_d a(t,x)) + x\cdot b(t,x) \leq K(1+|x|^2), \quad \forall\, (t,x)\in\overline\HH_T,
\end{equation}
where $K>0$.
Suppose $u\in C^{1,2}(\HH_T)\cap C(\overline{\HH}_T)$ obeys
\begin{equation}
\label{eq:CondUMaxPrinc1}
u_t, u_{x_i}, x_d u_{x_ix_j} \in C_{\loc}((0,T]\times\overline{\HH}), \quad 1 \leq i,j \leq d,
\end{equation}
and
\begin{equation}
\label{eq:CondUMaxPrinc2}
x_d u_{x_ix_j}=0 \hbox{  on } (0,T]\times\partial{\HH},  \quad 1 \leq i,j \leq d.
\end{equation}
If
\begin{align}
\label{eq:Subsolution}
Lu &\leq 0 \quad\hbox{on } (0,T)\times\HH,
\\
\label{eq:NonPositiveInitialCondition}
u(0,\cdot) &\leq 0 \quad\hbox{on } \overline{\HH},
\end{align}
then
\begin{equation}
\label{eq:SubsolutionUpperBound}
u \leq 0 \quad\hbox{on } [0,T]\times\overline{\HH}.
\end{equation}
\end{lem}

\begin{proof}
We apply an argument similar to that used in the proofs of \cite[Theorem 2.9.2, Exercises 2.9.4 and 2.9.5]{Krylov_LecturesHolder} (maximum principle for elliptic equations on unbounded domains) and
\cite[Theorems 8.1.2 and 8.1.4]{Krylov_LecturesHolder} (maximum principle for parabolic equations on unbounded domains); see also \cite[Theorem I.3.1]{DaskalHamilton1998}.

We consider the transformation
\begin{align}
\label{eq:FunctionTransformed}
u(t,x)=e^{\lambda t} \tilde u(t,x), \quad
\forall\, (t,x) \in [0,T]\times\overline\HH,
\end{align}
where the constant $\lambda>0$ will be suitably chosen below. The conclusion of the lemma follows if and only if \eqref{eq:SubsolutionUpperBound} holds for $\tilde u$. By \eqref{eq:Subsolution} and definition \eqref{eq:FunctionTransformed}, we have
\[
e^{\lambda t} \left(L+\lambda\right) \tilde u = Lu \leq 0 \quad \hbox{ on } (0,T)\times\HH.
\]
Therefore, by \eqref{eq:Subsolution} and \eqref{eq:NonPositiveInitialCondition}, the function $\tilde u$ satisfies
\begin{align}
\label{eq:SubsolutionTransformed}
\left(L+\lambda\right)\tilde u &\leq 0 \quad\hbox{on } (0,T)\times\HH,
\\
\label{eq:NonPositiveInitialConditionTransformed}
\tilde u(0,\cdot) &\leq 0 \quad\hbox{on } \overline{\HH}.
\end{align}
We may suppose without loss of generality that
\begin{equation}
\label{eq:NonNegMaximum}
m := \sup_{\HH_T} \tilde u \geq 0 ,
\end{equation}
as if $m<0$ we are done; we will show that $m=0$. Define an auxiliary function,
\begin{equation}
\label{eq:DefinitionH}
h(t,x) := 1+|x|^2, \quad \forall\, (t,x) \in \overline{\HH}_T.
\end{equation}
By direct calculation,
\begin{align*}
-\left(L+\lambda \right)h &= \sum_{i,j=1}^d x_da^{ij}h_{x_ix_j} + \sum_{i=1}^d b^ih_{x_i} + (c-\lambda)h - h_t
\\
&=  2x_d\sum_{i=1}^d a^{ii} + 2\sum_{i=1}^d b^ix_i + (c-\lambda)(1+|x|^2)
\\
&\leq \left(2K+c-\lambda\right)(1+|x|^2) \quad\hbox{on } (0,T)\times\HH, \quad \hbox{(by \eqref{eq:CoeffMaxPrincTraceGrowth})}
\end{align*}
By choosing
\begin{equation}
\label{eq:ChoiceLambda}
\lambda \geq 3K,
\end{equation}
we notice that condition \eqref{eq:ZerothOrderTermUpperBound}, gives
\begin{equation}
\label{eq:ZerothOrderTermNegative}
2K+c(t,x)-\lambda \leq 0, \quad \forall\, (t,x) \in\overline\HH_T,
\end{equation}
and so we have
\begin{equation}
\label{eq:AuxiliaryFunction}
\left(L+\lambda\right)h \geq 0 \quad\hbox{on } (0,T)\times\HH.
\end{equation}
Fix $\delta\in(0,1)$ and define another auxiliary function
\begin{equation}
\label{eq:DefnwR}
w := \tilde u - \delta mh.
\end{equation}
From \eqref{eq:SubsolutionTransformed} and \eqref{eq:AuxiliaryFunction}, we have $(L+\lambda)w \leq 0$ on $(0,T)\times\HH$ and thus
\begin{equation}
\label{eq:wRSubsolution}
\left(L+\lambda\right)w \leq 0 \quad\hbox{on } (0,T]\times\overline{\HH},
\end{equation}
since $w_t, w_{x_i}, x_dw_{x_ix_j}$ extend continuously from $(0,T)\times\HH$ to $(0,T]\times\overline{\HH}$ because these continuity properties are true of $u$ by hypothesis \eqref{eq:CondUMaxPrinc1} (and trivially true for $h$) and thus also true for $w$.

\begin{claim}
\label{claim:AuxiliaryMaxPrincipleBall}
There is a constant, $R_0=R_0(\delta) > 0$, such that
\begin{equation}
\label{eq:Adaugare4}
w \leq 0 \quad\hbox{on  } [0,T]\times \bar B_R, \quad \forall\, R \geq R_0(\delta).
\end{equation}
\end{claim}

\begin{proof}
Since $w \in C([0,T]\times \bar B_R)$, the function $w$ attains its maximum at some point $P \in [0,T]\times \bar B_R$. If $P \in (0,T]\times B_R$, then
$$
w_t(P) \geq 0, \quad w_{x_i}(P) = 0, \quad (w_{x_ix_j}(P)) \leq 0.
$$
Therefore, noting that $(a^{ij}(P))\geq 0$ by hypothesis,
\begin{align*}
-\left(L+\lambda\right)w(P) &= \sum_{i,j=1}^d x_da^{ij}(P)w_{x_ix_j}(P) + \sum_{i=1}^d b^i(P)w_{x_i}(P) + (c(P)-\lambda)w(P) - w_t(P)
\\
&\leq (c(P)-\lambda)w(P).
\end{align*}
If $P \in (0,T]\times (\bar B_R\cap\{x_d=0\})$, then
$$
w_t(P) \geq 0, \quad w_{x_d}(P)\leq 0, \quad w_{x_i}(P) = 0\quad (i\neq d),\quad x_dw_{x_ix_j}(P) = 0,
$$
where we use the fact that $u$, and thus $w$, obey \eqref{eq:CondUMaxPrinc1} and \eqref{eq:CondUMaxPrinc2}. Therefore,
\begin{align*}
-\left(L+\lambda\right)w(P) &= \sum_{i,j=1}^d x_da^{ij}(P)w_{x_ix_j}(P) + \sum_{i=1}^d b^i(P)w_{x_i}(P) + c(P)w(P) - w_t(P)
\\
&\leq b^d(P)w_{x_d}(P) + (c(P)-\lambda)w(P)
\\
&\leq (c(P)-\lambda)w(P) \quad\hbox{(by hypothesis that $b^d\geq 0$ on $\{x_d=0\}$).}
\end{align*}
Hence, for $P \in (0,T]\times B_R$ or $(0,T]\times (\bar B_R\cap\{x_d=0\})$, we obtain
$$
-(c(P)-\lambda)w(P) \leq Lw(P).
$$
But $Lw(P) \leq 0$ by \eqref{eq:wRSubsolution} and therefore $w(P) \leq 0$ since $c\leq K$ by \eqref{eq:ZerothOrderTermUpperBound} and $\lambda \geq 3K$ by \eqref{eq:ChoiceLambda} .

Now suppose $P$ lies in one of the remaining two components of the boundary of $(0,T)\times B_R$,
$$
\sB^0_R := \{0\}\times \bar B_R \quad\hbox{or}\quad \sB^1_R := (0,T] \times \left(\{x_d>0\}\cap\partial B_R\right).
$$
The definition \eqref{eq:DefinitionH} of $h$, definition \eqref{eq:DefnwR} of $w$, and \eqref{eq:NonPositiveInitialConditionTransformed} yield
\begin{equation}
\label{eq:NonPositiveInitialConditionwR}
w(0,\cdot) \leq 0 \hbox{  on  } \bar B_R, \quad \forall\, R>0,
\end{equation}
and thus, $w(P) \leq 0$ if $P \in \sB^0_R$, for $R>0$. If $P\in \sB^1_R$, then $|x|=R$ and we see that \eqref{eq:NonNegMaximum}, \eqref{eq:DefinitionH}, and \eqref{eq:DefnwR} give
\begin{align*}
w(P) &= \tilde u(P) - \delta mh(P)
\\
&\leq m - \delta m (1+R^2)
\\
&= m(1-\delta(1+R^2)).
\end{align*}
But $1-\delta(1+R^2) \leq 0$ provided $R\geq R_0(\delta) := (\delta^{-1}-1)^{1/2}>0$ and so $w(P) \leq 0$ for all $R\geq R_0(\delta)$. This completes the proof of Claim \ref{claim:AuxiliaryMaxPrincipleBall}.
\end{proof}

By \eqref{eq:Adaugare4}, we see that
$$
w = \tilde u - \delta mh \leq 0 \quad\hbox{on } \overline{\HH}_T,
$$
for all $\delta \in (0,1)$ and thus, letting $\delta\downarrow 0$, we obtain \eqref{eq:SubsolutionUpperBound}.
\end{proof}

Lemma \ref{lem:MaximumPrinciple} immediately leads to the following comparison principle.

\begin{cor}[Comparison principle]
\label{cor:Comparison}
Assume that the coefficients of $L$ in \eqref{eq:Generator} obey the hypotheses of Lemma \ref{lem:MaximumPrinciple}. If $u, v \in C^{1,2}(\HH_T)\cap C(\overline{\HH}_T)$ obey \eqref{eq:CondUMaxPrinc1}, \eqref{eq:CondUMaxPrinc2}, and
\begin{align}
\label{eq:RelativeSubsolution}
Lu &\leq Lv \quad\hbox{on } (0,T)\times\HH,
\\
\label{eq:RelativeNonPositiveInitialCondition}
u(0,\cdot) &\leq v(0,\cdot) \quad\hbox{on } \overline{\HH},
\end{align}
then
\begin{equation}
\label{eq:RelativeSubsolutionUpperBound}
u \leq v \quad\hbox{on } [0,T]\times\overline{\HH}.
\end{equation}
\end{cor}

Note that if \eqref{eq:RelativeSubsolution} and \eqref{eq:RelativeNonPositiveInitialCondition} are strengthened to
\begin{equation}
\label{eq:MaxPrinc1}
|Lu| \leq Lv \quad\hbox{on }(0,T)\times\HH \quad\hbox{and}\quad  |u(0,\cdot)| \leq v(0,\cdot) \quad\hbox{on }\overline{\HH},
\end{equation}
then Corollary \ref{cor:Comparison} yields
\begin{equation}
\label{eq:MaxPrinc2}
|u| \leq v \quad\hbox{on } [0,T]\times\overline{\HH}.
\end{equation}
We can now turn our attention to the

\begin{prop}[Application of the maximum principle]
\label{prop:MaximumPrinciple}
Assume that the coefficients of $L$ in \eqref{eq:Generator} obey the hypotheses of Lemma \ref{lem:MaximumPrinciple}, except that \eqref{eq:CoeffMaxPrincTraceGrowth} is replaced by the stronger condition
\begin{equation}
\label{eq:CoeffMaxPrincGrowth}
\sum_{i,j=1}^d x_d |a^{ij}(t,x)| + |x\cdot b(t,x)| \leq K(1+|x|^2), \quad \forall\, (t,x)\in\overline\HH_T.
\end{equation}
Suppose that $u\in C^{1,2}(\HH_T)\cap C(\overline{\HH}_T)$ solves \eqref{eq:Problem} and obeys \eqref{eq:CondUMaxPrinc1} and \eqref{eq:CondUMaxPrinc2}.
\begin{itemize}
\item[(a)] If $f\in C(\overline{\HH}_T)$ and $g \in C(\overline{\HH})$, then
           \begin{equation}
           \label{eq:MaximumPrinciple1}
           \begin{aligned}
           \|u\|_{C(\overline{\HH}_T)} \leq  e^{KT}\left( T \|f\|_{C(\overline{\HH}_T)} + \|g\|_{C(\overline{\HH})} \right).
           \end{aligned}
           \end{equation}
\item[(b)] If $q > 0$, $f\in \sC^0_q(\overline{\HH}_T)$, and $g\in \sC^0_q(\overline{\HH})$, then
           \begin{equation}
           \label{eq:MaximumPrinciple3}
           \begin{aligned}
           \|u\|_{\sC^{0}_q(\overline{\HH}_T)} \leq  e^{(1+q(q+4)K)T}\left(\|f\|_{\sC^0_q(\overline{\HH}_T)} + \|g\|_{\sC^0_q(\overline{\HH})}\right).
           \end{aligned}
           \end{equation}
\end{itemize}
\end{prop}

\begin{proof}
To obtain \eqref{eq:MaximumPrinciple1} and \eqref{eq:MaximumPrinciple3}, we make specific choices of the function $v$ in Corollary \ref{cor:Comparison}. To establish \eqref{eq:MaximumPrinciple1}, we choose
\[
v_1(t,x) := e^{K t}\left(t \|f\|_{C(\overline{\HH}_T)} + \|g\|_{C(\overline{\HH})}\right), \quad \forall\, (t,x)\in\overline{\HH}_T,
\]
Direct calculation gives
\begin{align*}
Lv_1 &= (-c+K) v_1 +  e^{K t} \|f\|_{C(\overline{\HH}_T)}
\\
&\quad \geq \|f\|_{C(\overline{\HH}_T)} \quad\hbox{on } (0,T)\times\HH \quad\hbox{(by \eqref{eq:ZerothOrderTermUpperBound})}.
\end{align*}
Therefore, since $Lu=f$ on $(0,T)\times\HH$ by \eqref{eq:Problem},
\[
|Lu| \leq L v_1 \quad \hbox{on } (0,T)\times\HH,
\]
and so $v_1$ satisfies condition \eqref{eq:MaxPrinc1}. Thus, by \eqref{eq:MaxPrinc2}, we obtain \eqref{eq:MaximumPrinciple1}.

Next, we prove \eqref{eq:MaximumPrinciple3}. For this purpose, we choose
\begin{equation}
\label{def:DefinitionV2}
v_2(t,x) :=  e^{\lambda t}\frac{\left(\|f\|_{\sC^0_{q}(\overline{\HH}_T)} + \|g\|_{\sC^0_{q}(\overline{\HH})}\right)}{(1+|x|^2)^{q/2}},\quad \forall\, (t,x)\in\overline{\HH}_T,
\end{equation}
where $\lambda>0$ will be suitably chosen below. First, we verify that $v_2$ satisfies the first inequality in \eqref{eq:MaxPrinc1}. Direct calculation gives
\begin{equation*}
\begin{aligned}
L v_2 &=
v_2\left[-c(t,x)+\lambda + q \sum_{i=1}^d\frac{b^i(t,x)x_i}{1+|x|^2} -q(q+2) \sum_{i,j=1}^d\frac{a^{ij}(t,x) x_i x_j x_d}{(1+|x|^2)^2}
+q \sum_{i=1}^d\frac{a^{ii}(t,x) x_d}{1+|x|^2} \right].
\end{aligned}
\end{equation*}
Conditions \eqref{eq:CoeffMaxPrincGrowth} and \eqref{eq:ZerothOrderTermUpperBound}, imply that
\begin{align*}
L v_2 \geq v_2\left(K+\lambda - qK - q(q+2)K - qK \right).
\end{align*}
By choosing
\[
\lambda = 1+q(q+4) K > 0,
\]
we obtain
\[
L v_2 \geq v_2 \geq \frac{\|f\|_{\sC^0_{q}(\overline{\HH}_T)}}{{(1+|x|^2)^{q/2}}} \quad\hbox{on } (0,T)\times\HH.
\]
By the definition \eqref{eq:DefinitionX_0_q_T} of the  norm $\|\cdot\|_{\sC^0_q(\overline{\HH}_T)}$,  we have
\[
\left(1+|x|^2\right)^{q/2} |f(t,x)| \leq \|f\|_{\sC^0_{q}(\overline{\HH}_T)}, \quad\forall\, (t,x) \in [0,T]\times\overline{\HH},
\]
and so, using $Lu=f$ on $\HH_T$ by \eqref{eq:Problem}, we obtain the first inequality in \eqref{eq:MaxPrinc1}, that is,
\begin{equation}
\label{eq:MaxPrincADD1}
|Lu| \leq Lv_2 \quad\hbox{on } (0,T)\times\HH.
\end{equation}
Similarly, by the definition \eqref{eq:DefinitionX_0_q} of the norm $\|\cdot\|_{\sC^0_q(\overline{\HH})}$, we have
\[
\left(1+|x|^2\right)^{q/2} |g(x)| \leq \|g\|_{\sC^0_{q}(\overline{\HH})}, \quad\forall\, x \in  \overline{\HH}.
\]
Since $u(0, \cdot)=g$ on $\overline{\HH}$, it is immediate that
\begin{equation}
\label{eq:MaxPrincADD2}
|u(0,\cdot)|\leq v_2(0,\cdot) \quad\hbox{on } \overline{\HH}.
\end{equation}
Therefore, by \eqref{eq:MaxPrincADD1} and \eqref{eq:MaxPrincADD2}, $v_2$ obeys conditions \eqref{eq:MaxPrinc1}, and so we obtain \eqref{eq:MaximumPrinciple3} from the definition \eqref{def:DefinitionV2} of $v_2$.
\end{proof}

\subsection{Local a priori boundary estimates}
\label{subsec:LocalAPrioriBoundaryEstimates}
We have the following analogue of \cite[Theorem 8.11.1]{Krylov_LecturesHolder}.

\begin{thm}[A priori boundary estimates]
\label{thm:AprioriBoundaryEst}
There is constant a $R^*=R^*(d,\alpha,K,\delta, \nu)$, such that for any $0<R\leq R^*$, we can find a positive constant $C=C(d,\alpha,K,\delta,\nu,R)$, such that for any  $x^0 \in \partial \HH$, $T\in (0, R]$ and $u \in C^{2+\alpha}_s(\bar Q_{3R/2,T}(x^0))$ that satisfies
\begin{equation}
\label{eq:AprioriBoundaryPDE}
\begin{aligned}
\begin{cases}
         L u = f  & \mbox{ on } Q_{3R/2,T}(x^0), \\
         u(0,\cdot)=g  & \mbox{ on } \bar B_{3R/2}(x^0),
\end{cases}
\end{aligned}
\end{equation}
the following estimate holds
\begin{equation}
\label{eq:AprioriBoundaryEst}
\begin{aligned}
\|u\|_{C^{2+\alpha}_s(\bar Q_{R,T}(x^0))} \leq C \left(\|f\|_{C^{\alpha}_s(\bar Q_{3R/2,T}(x^0))} + \|g\|_{C^{2+\alpha}_s(\bar B_{3R/2}(x^0))} + \|u\|_{C(\bar Q_{3R/2,T}(x^0))}\right).
\end{aligned}
\end{equation}
\end{thm}

\begin{proof}
The proof is a blend of the localizing technique used in \cite[Theorem 8.11.1]{Krylov_LecturesHolder} and the method of freezing the coefficients. Fix $R>0$ and $T\in(0,R]$. Let $\varphi:\RR\rightarrow [0,1]$ be a smooth function such that $\varphi(t)=0$ for $t<0$, and $\varphi(t)=1$ for $t>1$. Let
\[
R_n=R\sum_{k=0}^{n} \frac{1}{3^k},
\]
and consider the sequence of smooth cutoff functions $\{\varphi_n\}_{n\geq 1}\subset C^\infty(\bar\RR^d)$ defined by
\[
\varphi_n(x):=\varphi\left(\frac{R_{n+1}-|x|}{R_{n+1}-R_n}\right), \quad \forall\, x\in\overline{\HH},
\]
so that $0\leq \varphi_n\leq 1$ and $\varphi_n|_{B_{R_n}} \equiv 1$ and $\varphi_n|_{B^c_{R_{n+1}}}\equiv 0$, where $B^c_{R_{n+1}}$ denotes the complement of $B_{R_{n+1}}$ in $\RR^d$. Also, by direct calculation, we can find a positive constant $c$, independent of $n$ and $R$, such that
\begin{equation}
\label{eq:PropCutOffFunction}
\begin{aligned}
\|\varphi_n\|_{C^{\alpha}_s(\overline{\HH})},
\|(\varphi_n)_{x_i}\|_{C^{\alpha}_s(\overline{\HH})},
\|x_d(\varphi_n)_{x_ix_j}\|_{C^{\alpha}_s(\overline{\HH})},
\|(\varphi_n)_{x_ix_j}\|_{C^{\alpha}_s(\overline{\HH})}
\leq c 3^{3n}R^{-3}.
\end{aligned}
\end{equation}
We denote $r:=3^{-3}<1$ and set
\begin{equation}
\label{eq:DefinitionAlpha}
\begin{aligned}
\alpha_n &:= \|u\varphi_n\|_{C^{2+\alpha}_s(\overline{\HH}_T)}.
\end{aligned}
\end{equation}
We denote by $L_0$ the operator with constant coefficients obtained by freezing the coefficients of $L$ at $(0,x^0)$. Proposition \ref{prop:ConstantCoeff} shows there exists a positive constant $C$, depending only on $K$, $\delta$ and $\nu$, such that
\begin{equation}
\label{eq:ADDLocalAprioriBoundaryEst}
\alpha_n
=\|u\varphi_n\|_{C^{2+\alpha}_s(\overline{\HH}_T)}
\leq C \left( \|L_0(u\varphi_n)\|_{C^{\alpha}_s(\overline{\HH}_T)} + \|g\varphi_n\|_{C^{2+\alpha}_s(\overline{\HH})}\right),
\end{equation}
and so
\begin{equation}
\label{eq:ConstantCoeff1}
\begin{aligned}
\alpha_n
&\leq C \left( \|L(u\varphi_n)\|_{C^{\alpha}_s(\overline{\HH}_T)}
+\|(L-L_0)(u\varphi_n)\|_{C^{\alpha}_s(\overline{\HH}_T)}
+ \|g\varphi_n\|_{C^{2+\alpha}_s(\overline{\HH})}\right).
\end{aligned}
\end{equation}
We have $L(u\varphi_n) = \varphi_n Lu - [L,\varphi_n]u$, where, by direct calculation,
\begin{equation}
\label{eq:ConstantCoeff2}
\begin{aligned}
\left[L,\varphi_n\right]u &= \sum_{i,j=1}^d  2x_d a^{ij}(t,x) u_{x_i} (\varphi_n)_{x_j} + \sum_{i=1}^d b^i(t,x) u (\varphi_n)_{x_i} + \sum_{i,j=1}^dx_d a^{ij}(t,x) u (\varphi_n)_{x_ix_j}.
\end{aligned}
\end{equation}
By the analogue of the \cite[Inequality (4.7)]{GilbargTrudinger} for standard H\"older norms, we have
\begin{equation*}
\|\varphi_n Lu\|_{C^{\alpha}_s(\overline{\HH}_T)}
\leq c \|Lu\|_{C^{\alpha}_s(\bar Q_{R_{n+1},T})} \|\varphi_n\|_{C^{\alpha}_s(\overline{\HH})},
\end{equation*}
and by \eqref{eq:PropCutOffFunction}, there is a positive constant, $c$, such that
\begin{equation}
\label{eq:ConstantCoeff3}
\begin{aligned}
\|\varphi_n Lu\|_{C^{\alpha}_s(\overline{\HH}_T)}
\leq c r^{-n}R^{-3} \|f\|_{C^{\alpha}_s(\bar Q_{3R/2,T})}.
\end{aligned}
\end{equation}
From properties \eqref{eq:BoundednessNearBoundary} and \eqref{eq:LocalHolderS} of the coefficients $a^{ij}$, $b^i$ and $c$ on $\overline{\HH}_{2,T}$, we can find a positive constant $C$, depending only on $K$ and $d$, such that
\begin{equation}
\label{eq:ConstantCoeff3_2}
\begin{aligned}
\|[L,\varphi_n]u\|_{C^{\alpha}_s(\overline{\HH}_T)}
&\leq C r^{-n}R^{-3} \left(\|x_d (u\varphi_{n+1})_{x_i}\|_{C^{\alpha}_s(\overline{\HH}_T)} + \|u\varphi_{n+1}\|_{C^{\alpha}_s(\overline{\HH}_T)} \right).
\end{aligned}
\end{equation}
The interpolation inequality \eqref{eq:InterpolationIneqS3} in Lemma \ref{lem:InterpolationIneqS} gives us, for any $\eps\in(0,1)$,
\begin{equation}
\label{eq:ConstantCoeff4}
\begin{aligned}
\|x_d (u\varphi_{n+1})_{x_i}\|_{C^{\alpha}_s(\overline{\HH}_T)} + \|u\varphi_{n+1}\|_{C^{\alpha}_s(\overline{\HH}_T)}
\leq \eps \|u\varphi_{n+1}\|_{C^{2+\alpha}_s(\overline{\HH}_T)} + C \eps^{-m} \|u\varphi_{n+1}\|_{C(\overline{\HH}_T)}.
\end{aligned}
\end{equation}
Hence, the preceding inequality together with \eqref{eq:ConstantCoeff3} and \eqref{eq:ConstantCoeff3_2} give us
\begin{equation}
\label{eq:ConstantCoeff5}
\begin{aligned}
\|L(u\varphi_n)\|_{C^{\alpha}_s(\overline{\HH}_T)}
&\leq
C r^{-n}R^{-3} \left(\|f\|_{C^{\alpha}_s(\bar Q_{3R/2,T})}+\eps \|u\varphi_{n+1}\|_{C^{2+\alpha}_s(\overline{\HH}_T)} \right.
\\
&\quad + \left. \eps^{-m} \|u\varphi_{n+1}\|_{C(\overline{\HH}_T)} \right).
\end{aligned}
\end{equation}
Next, we estimate the term $(L-L_0)(u\varphi_n)$ in \eqref{eq:ConstantCoeff1}, that is,
\begin{equation}
\label{eq:EqLMinusL_0}
\begin{aligned}
-(L-L_0)(u\varphi_n)
&=
\sum_{i,j=1}^d x_d \left(a^{ij}(t,x)-a^{ij}(0,x^0)\right)(u\varphi_n)_{x_ix_j} \\
&\quad
 + \sum_{i=1}^d \left(b^i(t,x)-b^i(0,x^0)\right)(u\varphi_n)_{x_i}
 + \left(c(t,x)-c(0,x^0)\right)(u\varphi_n).
\end{aligned}
\end{equation}
We have

\begin{claim}
\label{claim:LocalBoundaryClaim1}
There is a constant $C=C(K, R^*, d, \alpha)$ such that, for any $\eps\in (0,1)$, we have
\begin{equation}
\label{eq:ADDLminusL_0}
\begin{aligned}
\|(L-L_0)(u\varphi_n)\|_{C^{\alpha}_s(\overline{\HH}_T)}
&\leq C\left(R^{\alpha/2} + r^{-n}R^{-3} \eps\right)\|u\varphi_{n+1}\|_{C^{2+\alpha}_s(\overline{\HH}_T)}
\\
&\quad + C r^{-n}R^{-3} \eps^{-m} \|u\varphi_{n+1}\|_{C(\overline{\HH}_T)},
\end{aligned}
\end{equation}
where $m$ is the constant appearing in Lemma \ref{lem:InterpolationIneqS}.
\end{claim}

\begin{proof}[Proof of Claim \ref{claim:LocalBoundaryClaim1}]
From the H\"older continuity \eqref{eq:LocalHolderS} and boundedness \eqref{eq:BoundednessNearBoundary} of the coefficients $a^{ij}$ on $\overline{\HH}_{2,T}$, we can find a positive constant $C$, depending only on $K$ and $d$, such that
\begin{equation}
\label{eq:ConstantCoeff6}
\begin{aligned}
&\|x_d \left(a^{ij}(t,x)-a^{ij}(0,x^0)\right)(u\varphi_n)_{x_ix_j}\|_{C^{\alpha}_s(\overline{\HH}_T)} \\
&\qquad \leq CR^{\alpha/2} \|x_d(u\varphi_n)_{x_ix_j}\|_{C^{\alpha}_s(\overline{\HH}_T)}
 + C \|x_d(u\varphi_n)_{x_ix_j}\|_{C(\overline{\HH}_T)}.
\end{aligned}
\end{equation}
Using the following calculation in the preceding inequality,
\begin{equation*}
\begin{aligned}
\|x_d(u\varphi_n)_{x_ix_j}\|_{C^{\alpha}_s(\overline{\HH}_T)}
&\leq
\|x_du_{x_ix_j}\varphi_n\|_{C^{\alpha}_s(\overline{\HH}_T)}
+ \|x_du_{x_i}(\varphi_n)_{x_j}\|_{C^{\alpha}_s(\overline{\HH}_T)}
+ \|x_d u(\varphi_n)_{x_ix_j}\|_{C^{\alpha}_s(\overline{\HH}_T)}\\
&\leq
[x_d(u\varphi_{n+1})_{x_ix_j}]_{C^{\alpha}_s(\overline{\HH}_T)}
+c r^{-n}R^{-3} \left(\|x_d(u\varphi_{n+1})_{x_ix_j}\|_{C(\overline{\HH}_T)}\right.\\
&\left.\quad+\|x_d(u\varphi_{n+1})_{x_i}\|_{C^{\alpha}_s(\overline{\HH}_T)} + \|x_du\varphi_{n+1}\|_{C^{\alpha}_s(\overline{\HH}_T)}\right),
\end{aligned}
\end{equation*}
together with the interpolation inequality \eqref{eq:InterpolationIneqS4} in Lemma \ref{lem:InterpolationIneqS} applied to $u\varphi_{n+1}$,
\begin{equation*}
\begin{aligned}
&\|x_d(u\varphi_{n+1})_{x_ix_j}\|_{C(\overline{\HH}_T)}+\|x_d (u\varphi_{n+1})_{x_i}\|_{C^{\alpha}_s(\overline{\HH}_T)} + \|u\varphi_{n+1}\|_{C^{\alpha}_s(\overline{\HH}_T)}\\
&\qquad \leq \eps \|u\varphi_{n+1}\|_{C^{2+\alpha}_s(\overline{\HH}_T)} + C \eps^{-m} \|u\varphi_{n+1}\|_{C(\overline{\HH}_T)},
\end{aligned}
\end{equation*}
we obtain from \eqref{eq:ConstantCoeff6} that
\begin{equation*}
\begin{aligned}
&\|x_d \left(a^{ij}(t,x)-a^{ij}(0,x^0)\right)(u\varphi_n)_{x_ix_j}\|_{C^{\alpha}_s(\overline{\HH}_T)} \\
&\quad \leq CR^{\alpha/2} [x_d(u\varphi_{n+1})]_{C^{\alpha}_s(\overline{\HH}_T)}
 + C r^{-n}R^{-3} \eps \|u\varphi_{n+1}\|_{C^{2+\alpha}_s(\overline{\HH}_T)}
 + C r^{-n}R^{-3} \eps^{-m} \|u\varphi_{n+1}\|_{C(\overline{\HH}_T)}\\
&\quad \leq C\left(R^{\alpha/2} + r^{-n}R^{-3} \eps\right)\|u\varphi_{n+1}\|_{C^{2+\alpha}_s(\overline{\HH}_T)}
 + C r^{-n}R^{-3} \eps^{-m} \|u\varphi_{n+1}\|_{C(\overline{\HH}_T)}.
\end{aligned}
\end{equation*}
A similar argument gives us
\begin{equation*}
\begin{aligned}
& \|\left(b^i(t,x)-b^i(0,x^0)\right)(u\varphi_n)_{x_i}\|_{C^{\alpha}_s(\overline{\HH}_T)}
 + \|\left(c(t,x)-c(0,x^0)\right)(u\varphi_n)\|_{C^{\alpha}_s(\overline{\HH}_T)}\\
&\qquad \leq
 C r^{-n}R^{-3} \eps \|u\varphi_{n+1}\|_{C^{2+\alpha}_s(\overline{\HH}_T)} + C r^{-n}R^{-3} \eps^{-m} \|u\varphi_{n+1}\|_{C(\overline{\HH}_T)},
\end{aligned}
\end{equation*}
and so, using the preceding inequalities in \eqref{eq:EqLMinusL_0}, we obtain the estimate \eqref{eq:ADDLminusL_0}.
\end{proof}

Combining  \eqref{eq:ConstantCoeff5}, \eqref{eq:ADDLminusL_0} and \eqref{eq:ConstantCoeff1}, we obtain
\begin{equation}
\label{eq:ConstantCoeff7}
\begin{aligned}
\alpha_n
&\leq
C r^{-n} R^{-3}\left(\|f\|_{C^{\alpha}_s(\bar Q_{3R/2,T})} + \|g\|_{C^{2+\alpha}_s(\bar B_{3R/2})} \right)\\
&\quad+ C \left(R^{\alpha/2} + r^{-n}R^{-3} \eps\right)  \alpha_{n+1}
+ C r^{-n}R^{-3} \eps^{-m} \|u\|_{C(\bar Q_{3R/2,T})}.
\end{aligned}
\end{equation}
We multiply the inequality \eqref{eq:ConstantCoeff7} by $\delta^n$, where $\delta>0$ is chosen such that
\begin{equation}
\label{eq:Choicedelta}
r^{-(m+1)}\delta \leq 1/2.
\end{equation}
Next, we choose $R^*>0$ such that $CR^{* \alpha/2}=\delta/2$. For $R \in (0,R^*]$, we choose $\eps=\eps(n,R)\in(0,1)$ such that $Cr^{-n} R^{-3}\eps=\delta/2$. With this choice of $\delta$, $R^*$ and $\eps$, inequality \eqref{eq:ConstantCoeff7} yields, for all $R \in (0,R^*]$,
\begin{equation*}
\begin{aligned}
\delta^n \alpha_n &\leq
 C R^{-3}(r^{-1}\delta)^{n} \left(\|f\|_{C^{\alpha}_s(\bar Q_{3R/2,T})}  + \|g\|_{C^{2+\alpha}_s(\bar B_{3R/2})} \right)\\
&\quad + \delta^{n+1} \alpha_{n+1}
+ (2C)^{m+1} R^{-3(m+1)} \delta^{-m} (r^{-(m+1)}\delta)^{n} \|u\|_{C(\bar B_{3R/2,T})}.
\end{aligned}
\end{equation*}
By \eqref{eq:Choicedelta}, we also have $r^{-1}\delta \leq 1/2$. Then, by choosing
\[
C_1:=\max\left\{C R^{-3}, (2C)^{m+1} R^{-3(m+1)} \delta^{-m}\right\},
\]
we obtain
\begin{equation}
\label{eq:ConstantCoeff8}
\begin{aligned}
\delta^n \alpha_n &\leq
 C_1 \frac{1}{2^n} \left(\|f\|_{C^{\alpha}_s(\bar Q_{3R/2,T})}  + \|g\|_{C^{2+\alpha}_s(\bar B_{3R/2})} \right)
+ \delta^{n+1} \alpha_{n+1}
+ C_1 \frac{1}{2^n} \|u\|_{C(\bar Q_{3R/2,T})}.
\end{aligned}
\end{equation}
Summing inequality \eqref{eq:ConstantCoeff8} yields
\begin{equation*}
\begin{aligned}
\sum_{n=0}^{\infty} \delta^n \alpha_n
&\leq
C_1 \left(\|f\|_{C^{\alpha}_s(\bar Q_{3R/2,T})}  + \|g\|_{C^{2+\alpha}_s(\bar B_{3R/2})} \right)  \sum_{n=0}^{\infty} \frac{1}{2^n} \\
&\quad + \sum_{n=0}^{\infty} \delta^{n+1} \alpha_{n+1}
+ C_1 \|u\|_{C(\bar Q_{3R/2,T})} \sum_{n=0}^{\infty} \frac{1}{2^n} .
\end{aligned}
\end{equation*}
The sum $\sum_{n=0}^{\infty} \delta^n \alpha_n$ is well-defined because we assumed $u \in C^{2+\alpha}_s(\bar Q_{3R/2,T})$, for all $R\in(0,R^*]$ and $T\in(0, R]$, while $\delta \in (0,1)$. By subtracting the term $\sum_{n=1}^{\infty} \delta^n \alpha_n$ from both sides of the preceding inequality, we obtain the desired inequality \eqref{eq:AprioriBoundaryEst}.
\end{proof}

\subsection{Local a priori interior estimates}
\label{subsec:LocalAprioriInteriorEstimates}
In order to establish the local interior estimates, we need to track the dependency of the constant $N$ appearing in \cite[Lemma 9.2.1 and Theorem 9.2.2]{Krylov_LecturesHolder} on the constant of uniform ellipticity and on the supremum and H\"older norms of the coefficients. Lemma \ref{lem:ClassicalEstConstCoeff} and Proposition \ref{prop:ClassicalEstVarCoeff} apply to a parabolic operator,
\begin{equation}
\label{eq:ClassicalGen}
\begin{aligned}
-\bar{L} u := -u_t + \sum_{i,j=1}^d \bar{a}_{ij}u_{x_ix_j} + \sum_{i=1}^d \bar{b}_iu_{x_i} + \bar{c}u,
\end{aligned}
\end{equation}
whose coefficients obey

\begin{hyp}
\label{hyp:ClassicalCoeff}
There are positive constants $\delta_1$, $K_1$ and $\lambda_1$ such that
\begin{enumerate}
\item $(\bar a^{ij}(t,x))$ is a symmetric, positive-definite matrix, for all $t\in [0,T]$ and $x \in \mathbb{R}^d$.
\item The diffusion matrix $\bar a$ is non-degenerate,
\begin{equation}
\label{eq:NonDegeneracyBoundedCoeff}
\sum_{i,j=1}^d\bar{a}_{ij}(t,x) \xi_i \xi_j \geq \delta_1 |\xi|^2,\quad  \forall\, \xi \in \mathbb{R}^d, t\in [0,T], x \in \mathbb{R}^d.
\end{equation}
\item The coefficients $\bar a^{ij}$, $\bar b^i$ and $\bar c$ are uniformly H\"older continuous on $[0,T]\times\mathbb{R}^d$,
\begin{equation}
\label{eq:GlobalHolderContinuityBoundedCoeff}
\|\bar{a}_{ij}\|_{C^{\alpha}_{\rho}([0,T]\times\mathbb{R}^d)}+
       \|\bar{b}_i\|_{C^{\alpha}_{\rho}([0,T]\times\mathbb{R}^d)}+
       \|\bar{c}\|_{C^{\alpha}_{\rho}([0,T]\times\mathbb{R}^d)}
       \leq K_1.
\end{equation}
\item The zeroth-order coefficient, $\bar c$, is bounded from above,
\begin{equation}
\label{eq:UpperBoundBarC}
\bar c(t,x) \leq \lambda_1 \quad \forall\, t\in [0,T], x \in \mathbb{R}^d.
\end{equation}
\end{enumerate}
\end{hyp}

The difference between the statements of Lemma \ref{lem:ClassicalEstConstCoeff} and \cite[Lemmas 9.2.1 and 8.9.1]{Krylov_LecturesHolder} is that we explicitly give the dependency of the constant $N_2$ on $\delta_1$ and $K_1$; the proofs are the same except that at each step we explicitly determine the dependency of the constants appearing in the estimate \eqref{eq:SchauderEstConstCoeff} on $\delta_1$ and $K_1$.

\begin{lem} [A priori estimate for a simple parabolic operator with constant coefficients]
\label{lem:ClassicalEstConstCoeff}
Assume that  $(\bar{a}_{ij})$ in \eqref{eq:ClassicalGen} is a constant matrix obeying \eqref{eq:NonDegeneracyBoundedCoeff}, $\bar{b}_i=0$, and $\bar{c}=0$. Then there are positive constants,
\begin{align}
\label{eq:GrowthConstCoeff1}
N_1 &= N_1(d, \alpha, T), \\
\label{eq:GrowthConstCoeff2}
N_2 &= N_1 \max\{1,\delta_1^{-1}\} \max\{1, K_1\} (1+\delta_1^{-\alpha/2}) (1+K_1^{\alpha/2}),
\end{align}
such that, for any solution $u\in C^{2+\alpha}_{\rho}([0,T]\times\mathbb{R}^d)$ to
\begin{equation}
\label{eq:NonHomIniValConstCoeff}
\begin{aligned}
\begin{cases}
         \bar{L} u = f & \mbox{ on } (0,T)\times\RR^d, \\
         u(0,\cdot)=g  & \mbox{ on } \RR^d,
\end{cases}
\end{aligned}
\end{equation}
with $f \in C^{\alpha}_{\rho}([0,T]\times\mathbb{R}^d)$ and $g \in C^{2+\alpha}_{\rho}(\mathbb{R}^d)$, we have
\begin{equation}
\label{eq:SchauderEstConstCoeff}
\begin{aligned}
\|u\|_{C^{2+\alpha}_{\rho}([0,T]\times\mathbb{R}^d)}
&\leq N_2 \left(\|f\|_{C^{\alpha}_{\rho}([0,T]\times\mathbb{R}^d)} + \|g\|_{C^{2+\alpha}_{\rho}(\mathbb{R}^d)}\right).
\end{aligned}
\end{equation}
\end{lem}

The proof of Lemma \ref{lem:ClassicalEstConstCoeff} can be found in Appendix \ref{app:ProofKrylovResults}. The statement of Proposition \ref{prop:ClassicalEstVarCoeff} is the same as that of \cite[Theorems 9.2.2 and 8.9.2]{Krylov_LecturesHolder} except that in the estimate \eqref{eq:SchauderEstVarCoeff}, the dependency of the constant $N_4$ on $\delta_1$ and $K_1$ is made explicit in \eqref{eq:GrowthVarCoeffP}, \eqref{eq:GrowthVarCoeffN_3} and \eqref{eq:GrowthVarCoeffN_4}.

\begin{prop} [A priori estimate for a parabolic operator with variable coefficients]
\label{prop:ClassicalEstVarCoeff}
Assume Hypothesis \ref{hyp:ClassicalCoeff}. Then there are positive constants,
\begin{align}
\label{eq:GrowthVarCoeffP}
p   &= p(\alpha) \geq 1,\\
\label{eq:GrowthVarCoeffN_3}
N_3 &= N_3(d, \alpha, T), \\
\label{eq:GrowthVarCoeffN_4}
N_4 &= N_3 e^{\lambda_1 T}\left( 1+\delta_1^{-p} + K_1^p\right) ,
\end{align}
such that, for any solution $u\in C^{2+\alpha}_{\rho}([0,T]\times\mathbb{R}^d)$ to
\begin{equation*}
\begin{aligned}
\begin{cases}
         \bar{L} u = f & \mbox{ on } (0,T)\times\mathbb{R}^d, \\
         u(0,\cdot)=g  & \mbox{ on } \RR^d,
\end{cases}
\end{aligned}
\end{equation*}
we have
\begin{equation}
\label{eq:SchauderEstVarCoeff}
\begin{aligned}
\|u\|_{C^{2+\alpha}_{\rho}([0,T]\times\mathbb{R}^d)}
&\leq N_4 \left(\|f\|_{C^{\alpha}_{\rho}([0,T]\times\mathbb{R}^d)} + \|g\|_{C^{2+\alpha}_{\rho}(\mathbb{R}^d)}\right).
\end{aligned}
\end{equation}
\end{prop}

The proof of Proposition \ref{prop:ClassicalEstVarCoeff} can be found in Appendix \ref{app:ProofKrylovResults}. Next, we have the

\begin{prop} [Local estimates for parabolic operators with variable coefficients]
\label{prop:ClassicalEstVarCoeffLocal}
Assume Hypothesis \ref{hyp:ClassicalCoeff} and that $R>0$. Then there are positive constants,
\begin{align}
\label{eq:GrowthVarCoeffLocalP}
p   &= p(\alpha) \geq 1,\\
\label{eq:GrowthVarCoeffLocalN_3}
N_3 &= N_3(d, \alpha, T, R), \\
\label{eq:GrowthVarCoeffLocalN_4}
N_4 &= N_3 e^{\lambda_1 T}\left( 1+\delta_1^{-p} + K_1^p\right) ,
\end{align}
such that for any $x^0\in\RR^d$ and any solution $u\in C^{2+\alpha}_{\rho}(\bar Q_{2R,T}(x^0))$ to
\begin{equation*}
\begin{aligned}
\begin{cases}
         \bar{L} u = f & \mbox{ on }  Q_{2R,T}(x^0), \\
         u(0,\cdot)=g  & \mbox{ on } \bar B_{2R}(x^0),
\end{cases}
\end{aligned}
\end{equation*}
we have
\begin{equation}
\label{eq:SchauderEstVarCoeffLocal}
\begin{aligned}
\|u\|_{C^{2+\alpha}_{\rho}(\bar Q_{R,T}(x^0))}
&\leq N_4 \left(\|f\|_{C^{\alpha}_{\rho}(\bar Q_{2R,T}(x^0))} + \|g\|_{C^{2+\alpha}_{\rho}(\bar B_{2R}(x^0))} \right.\\
&\qquad\left. + \|u\|_{C(\bar Q_{2R,T}(x^0))}\right).
\end{aligned}
\end{equation}
\end{prop}

\begin{proof}
The proof follows by the same argument as in Theorem \ref{thm:AprioriBoundaryEst} with the following modifications:
\begin{itemize}
\item In inequality \eqref{eq:ADDLocalAprioriBoundaryEst}, instead of applying Proposition \ref{prop:ConstantCoeff}, we apply Proposition \ref{prop:ClassicalEstVarCoeff}.
\item We use the interpolation inequalities for classical H\"older spaces $C^{2+\alpha}_{\rho}$ (\cite[Theorem 8.8.1]{Krylov_LecturesHolder}), instead of the interpolation inequalities for the H\"older spaces $C^{2+\alpha}_s$ (Lemma \ref{lem:InterpolationIneqS}).
\end{itemize}
This completes the proof.
\end{proof}

We now consider estimates for the operator $L$ in \eqref{eq:Generator}.

\begin{prop} [Interior local estimates]
\label{prop:InteriorEst}
There is a positive constant $p=p(\alpha)$, and for any $R\in (0, R^*]$, with $R^*$ as in Theorem \ref{thm:AprioriBoundaryEst}, there is a positive constant $C=C(d,\alpha,T,K,\delta,R^*, R)$ such that, for any $x^0\in \HH$, satisfying $x^0_d-2R\geq R^*/2$ and for any solution $u \in C^{2+\alpha}_{\rho}(\bar Q_{2R,T}(x^0))$ to the inhomogeneous initial value problem,
\begin{equation*}
\begin{aligned}
\begin{cases}
         L u = f & \mbox{ on }  Q_{2R,T}(x^0), \\
         u(0,\cdot)=g  & \mbox{ on } \bar B_{2R}(x^0),
\end{cases}
\end{aligned}
\end{equation*}
we have
\begin{equation}
\label{eq:InteriorEst}
\begin{aligned}
\|u\|_{C^{2+\alpha}_{\rho}(\bar Q_{R,T}(x^0))}
&\leq C \left(\|f\|_{\sC^{\alpha}_p(\bar Q_{2R,T}(x^0))} + \|g\|_{\sC^{2+\alpha}_p(\bar B_{2R}(x^0))} \right.
\\
&\qquad + \left.\|u\|_{\sC^0_p(\bar Q_{2R,T}(x^0))}\right).
\end{aligned}
\end{equation}
\end{prop}

\begin{proof}
From Proposition \ref{prop:ClassicalEstVarCoeffLocal}, the linear growth estimate \eqref{eq:LinearGrowth}, and the fact that the matrix $(x_da^{ij}(t,x))$ is uniformly elliptic on $\overline{\HH}_T\less\HH_{R^* /2,T}$ by \eqref{eq:NonDegeneracyNearBoundary} and \eqref{eq:NonDegeneracyInterior}, we obtain
\begin{equation}
\label{eq:NicerExpression}
\begin{aligned}
\|u\|_{C^{2+\alpha}_{\rho}(\bar Q_{R,T}(x^0))} &\leq C_1 (1+|x^0|)^p \left(\|f\|_{C^{\alpha}_{\rho}(\bar Q_{2R,T}(x^0))} + \|g\|_{C^{2+\alpha}_{\rho}(\bar B_{2R}(x^0))} \right.
\\
&\quad\left. + \|u\|_{C(\bar Q_{2R,T}(x^0))}\right),
\end{aligned}
\end{equation}
where $C_1$ is a positive constant depending only on $T$, $K$, $\delta$, $R^*$ and $R$.

\begin{claim}
\label{claim:PassToNewSpaces}
Given a function $v \in C^{2+\alpha}_{\rho}(\bar Q_{2R,T}(x^0))$, there is a positive constant $C_2$, depending only on $R^*$, $p$ and $\alpha$, such that for all $R \in (0,R^*]$  and $x^0 \in \HH_T$, we have
\begin{equation}
\label{eq:PassToNewSpaces}
(1+|x^0|)^p \|v\|_{C^{\alpha}_{\rho}(\bar Q_{2R,T}(x^0))} \leq  C_2\|v\|_{\sC^{\alpha}_{p}(\bar Q_{2R,T}(x^0))}.
\end{equation}
\end{claim}

\begin{proof}[Proof of Claim \ref{claim:PassToNewSpaces}]
Recall that, by definition \eqref{eq:DefinitionX_Alpha_q_T},
\[
\|(1+|x|)^p v\|_{C^{\alpha}_{\rho}(\bar Q_{2R,T}(x^0))} = \|v\|_{\sC^{\alpha}_{p}(\bar Q_{2R,T}(x^0))}.
\]
We may write
\begin{equation*}
(1+|x^0|)^p  |v(t,x)| = \left(\frac{1+|x^0|}{1+|x|}\right)^p (1+|x|)^p |v(t,x)|, \quad \forall\, (t,x) \in \bar Q_{2R,T}(x^0).
\end{equation*}
We can find a constant $C_2=C_2(R^*, p)$ such that
\begin{equation}
\label{eq:PassToNewSpaces1}
 \left(\frac{1+|x^0|}{1+|x|}\right)^p \leq C_2, \quad \forall\, x \in \bar B_{2R}(x^0), \quad \forall\, 0<R<R^*,
\end{equation}
which implies
\begin{equation}
\label{eq:PassToNewSpaces2}
 (1+|x^0|)^p \|v\|_{C(\bar Q_{2R,T}(x^0))} \leq  C_2\|(1+|x|)^pv\|_{C(\bar Q_{2R,T}(x^0))}.
\end{equation}
Next, we have
\begin{align*}
(1+|x^0|)^p [v]_{C^{\alpha}_{\rho}(\bar Q_{2R,T}(x^0))}
&= (1+|x^0|)^p \left[\frac{1}{(1+|x|)^p}(1+|x|)^p v\right]_{C^{\alpha}_{\rho}(\bar Q_{2R,T}(x^0))}\\
&\leq
(1+|x^0|)^p \left[\frac{1}{(1+|x|)^p}\right]_{C^{\alpha}_{\rho}(\bar B_{2R}(x^0))} \|(1+|x|)^p v\|_{C(\bar Q_{2R,T}(x^0))}\\
&\quad
+(1+|x^0|)^p \left\|\frac{1}{(1+|x|)^p}\right\|_{C(\bar B_{2R}(x^0))} [(1+|x|)^p v]_{C^{\alpha}_{\rho}(\bar Q_{2R,T}(x^0))}.
\end{align*}
As in \eqref{eq:PassToNewSpaces1}, there is a (possibly larger) constant $C_2=C_2(R^*,p,\alpha)$ such that
\[
(1+|x^0|)^p \left[\frac{1}{(1+|x|)^p}\right]_{C^{\alpha}_{\rho}(\bar B_{2R}(x^0))} \leq C_2.
\]
Therefore, we obtain
\begin{align}
\label{eq:PassToNewSpaces3}
(1+|x^0|)^p [v]_{C^{\alpha}_{\rho}(\bar Q_{2R,T}(x^0))}
& \leq C_2 \|(1+|x|)^p v\|_{C(\bar Q_{2R,T}(x^0))} + C_2 [(1+|x|)^p v]_{C^{\alpha}_{\rho}(\bar Q_{2R,T}(x^0))}.
\end{align}
Combining inequalities \eqref{eq:PassToNewSpaces2} and \eqref{eq:PassToNewSpaces3} yields the desired inequality \eqref{eq:PassToNewSpaces}.
\end{proof}

Claim \ref{claim:PassToNewSpaces} implies that
\begin{equation*}
\begin{aligned}
(1+|x^0|)^p \|f\|_{C^{\alpha}_{\rho}(\bar Q_{2R,T}(x^0))} &\leq  C_2 \|f\|_{\sC^{\alpha}_p(\bar Q_{2R,T}(x^0))}, \\
(1+|x^0|)^p \|g\|_{C^{2+\alpha}_{\rho}(\bar B_{2R}(x^0))} &\leq  C_2\|g\|_{\sC^{2+\alpha}_p(\bar B_{2R}(x^0))}, \\
(1+|x^0|)^p \|u\|_{C(\bar Q_{2R,T}(x^0))} &\leq  C_2\|u\|_{\sC^0_p(\bar Q_{2R,T}(x^0))},
\end{aligned}
\end{equation*}
and so, the interior local estimate \eqref{eq:InteriorEst} follows from the preceding inequalities and \eqref{eq:NicerExpression}.
\end{proof}

\subsection{Global a priori estimates and existence of solutions}
\label{subsec:GlobalAprioriEstimatesExistence}
The goal of this subsection is to establish Theorem \ref{thm:MainExistenceUniquenessPDE}. For this purpose, we need to first prove the analogue of Theorem \ref{thm:MainExistenceUniquenessPDE} when the coefficients are uniformly H\"older continuous on $\HH_T\less\HH_{2,T} = (0,T)\times\RR^{d-1}\times[2,\infty)$.

\begin{hyp}
\label{hyp:CoeffBounded}
In addition to the conditions in Assumption \ref{assump:Coeff}, assume that there is a positive constant $K_2$  such that the coefficients of $L$ obey
\begin{equation}
\label{eq:GlobalHolderRho}
\|x_d a^{ij}\|_{C^{\alpha}_{\rho}(\overline{\HH_T\less\HH_{2,T}})}+\|b^i\|_{C^{\alpha}_{\rho}(\overline{\HH_T\less\HH_{2,T}})}+ \|c\|_{C^{\alpha}_{\rho}(\overline{\HH_T\less\HH_{2,T}})} \leq K_2.
\end{equation}
\end{hyp}

We first derive global a priori estimates of solutions in the case of bounded coefficients.

\begin{lem}[Global estimates in the case of parabolic operators with bounded coefficients]
\label{lem:GlobalEstBoundedCoeff}
Suppose Hypothesis \ref{hyp:CoeffBounded} is satisfied. There exists a positive constant $C=C(T,\alpha,d,K_2,\delta, \nu)$ such that for any solution $u \in \sC^{2+\alpha}_{\loc}(\overline{\HH}_T)$ to \eqref{eq:Problem}, such that $Lu \in \sC^{\alpha}(\overline{\HH}_T)$ and $u(0,\cdot) \in \sC^{2+\alpha}(\overline{\HH})$, then $u \in \sC^{2+\alpha}(\overline{\HH}_T)$ and $u$ satisfies the global estimate
\begin{equation}
\label{eq:GlobalEstBoundedCoeff}
\begin{aligned}
\|u\|_{\sC^{2+\alpha}(\overline{\HH}_T)}
&\leq C \left(  \|Lu\|_{\sC^{\alpha}(\overline{\HH}_T)} +  \|u(0, \cdot)\|_{\sC^{2+\alpha}(\overline{\HH})} \right).
\end{aligned}
\end{equation}
\end{lem}

\begin{proof}
We shall apply a covering argument with the aid of Theorem \ref{thm:AprioriBoundaryEst} and Proposition \ref{prop:ClassicalEstVarCoeffLocal}. It is enough to prove the statement for $T>0$ small. Let $R^*>0$ be defined as in Theorem \ref{thm:AprioriBoundaryEst} and choose $T \in (0,R^*]$. Let $\{z^k: k\geq 1\}$ be a sequence of points in $\partial\HH$ such that
\begin{equation}
\label{eq:Cover1}
\HH_{R^*/2,T} \subset \bigcup_{k \geq 1}  Q_{R^*,T}(z^k),
\end{equation}
and let $\{w^l: l\geq 1\}$ be a sequence of points in $\HH_T \less \HH_{R^*/2, T}$ such that
\begin{equation}
\label{eq:Cover2}
\HH_T \less \HH_{R^*/2, T} \subset \bigcup_{l \geq 1}  Q_{R^*/8,T}(w^l),
\end{equation}
and assume
\begin{equation}
\label{eq:Cover3}
 Q_{R^*/4,T}(w^l) \cap \HH_{R^*/4,T} = \emptyset, \quad \forall\, l \geq 1.
\end{equation}
We apply the a priori boundary estimate \eqref{eq:AprioriBoundaryEst} to $u$ with $R=R^*$, $f=Lu$ and $g=u(0,\cdot)$ on $Q_{R^*,T}(z^k)$. Then, we can find a positive constant $C_1$, depending at most on $R^*$, $K_2$, $\delta$, $\nu$, such that
\begin{equation*}
\begin{aligned}
\|u\|_{C^{2+\alpha}_{s}(\bar Q_{R^*,T}(z^k))}
&\leq
C_1 \left(  \|Lu\|_{C^{\alpha}_{s}(\bar Q_{3R^*/2,T}(z^k))} +  \|u(0,\cdot)\|_{C^{2+\alpha}_{s}(\bar Q_{3R^*/2,T}(z^k))} \right.
\\
&\quad + \left. \|u\|_{C(\bar Q_{3R^*/2,T}(z^k))}  \right).
\end{aligned}
\end{equation*}
Using definitions \eqref{eq:DefinitionX_Alpha_q_T} of $\sC^{\alpha}(\overline{\HH}_T)$, and \eqref{eq:DefinitionX_Two_Aplha_q_T} of $\sC^{2+\alpha}(\overline{\HH})$, with $q=0$, Remark \ref{rmk:EquivalenceMetric} and the hypotheses that $Lu \in \sC^{\alpha}(\overline{\HH}_T)$ and $u(0,\cdot)\in \sC^{2+\alpha}(\overline{\HH})$, we obtain
\begin{equation*}
\begin{aligned}
\|u\|_{C^{2+\alpha}_{s}(\bar Q_{R^*,T}(z^k))}
&\leq
C_1 \left(  \|Lu\|_{\sC^{\alpha}(\overline{\HH}_T)} +  \|u(0,\cdot)\|_{\sC^{2+\alpha}(\overline{\HH})}
+  \|u\|_{C(\overline{\HH}_T)}  \right),
\end{aligned}
\end{equation*}
and  inequality \eqref{eq:MaximumPrinciple1} ensures
\begin{equation}
\label{eq:GlobalEstBoundedCoeff1}
\begin{aligned}
\|u\|_{C^{2+\alpha}_{s}(\bar Q_{R^*,T}(z^k))}
&\leq C_1 \left(  \|Lu\|_{\sC^{\alpha}(\overline{\HH}_T)} +  \|u(0,\cdot)\|_{\sC^{2+\alpha}(\overline{\HH})} \right), \quad \forall\, k \geq 1.
\end{aligned}
\end{equation}
From our Hypothesis \ref{hyp:CoeffBounded}, the coefficients $x_d a^{ij}$, $b^i$ and $c$ are in $C^{\alpha}_{\rho}(\overline{\HH_T\less\HH_{2,T}})$. By Assumption \ref{assump:Coeff}, we have that $x_d a^{ij}$, $b^i$ and $c$ are in $C^{\alpha}_{s}(\overline{\HH_{2,T}\less\HH_{R^*/4,T}})$. Since the distance functions $s$ and $\rho$ are equivalent on $\RR\times[R^*/4,2]$, by Remark \ref{rmk:EquivalenceMetric}, there is a positive constant $K_1$, depending on $K_2$ and $R^*$, such that
\[
\|x_da^{ij}\|_{C^{\alpha}_{\rho}(\overline{\HH_T\less\HH_{R^*/4,T}})}+
\|b^{i}\|_{C^{\alpha}_{\rho}(\overline{\HH_T\less\HH_{R^*/4,T}})}+
\|c\|_{C^{\alpha}_{\rho}(\overline{\HH_T\less\HH_{R^*/4,T}})} \leq K_1,
\]
and so the conditions of Hypothesis \ref{hyp:ClassicalCoeff} are obeyed on $\overline{\HH_T\less\HH_{R^*/4,T}}$. This is enough to ensure that we may apply Proposition \ref{prop:ClassicalEstVarCoeffLocal} to $u$ with $f=Lu$ and $g=u(0, \cdot)$  on $Q_{R^*/8,T}(w^l)$ and so there is a positive constant $C_2$, depending at most on $R^*$, $K_1$, $\delta$, $\nu$, giving
\begin{equation*}
\begin{aligned}
\|u\|_{C^{2+\alpha}_{\rho}(\bar Q_{R^*/8,T}(w^l))}
&\leq
C_2 \left(  \|Lu\|_{C^{\alpha}_{\rho}(\bar Q_{R^*/4,T}(w^l))} +  \|u(0,\cdot)\|_{C^{2+\alpha}_{\rho}(\bar Q_{R^*/4,T}(w^l))} \right.
\\
&\quad + \left. \|u\|_{C(\bar Q_{R^*/4,T}(w^l))} \right), \quad \forall\, l \geq 1.
\end{aligned}
\end{equation*}
By \eqref{eq:Cover3} and Remark \ref{rmk:EquivalenceMetric}, we obtain
\begin{equation*}
\begin{aligned}
\|u\|_{C^{2+\alpha}_{\rho}(\bar Q_{R^*/8,T}(w^l))}
&\leq
C_2 \left(  \|Lu\|_{\sC^{\alpha}(\overline{\HH})} +  \|u(0,\cdot)\|_{\sC^{2+\alpha}(\overline{\HH}_T)} \right.
\\
&\quad + \left. \|u\|_{C(\bar Q_{R^*/4,T}(w^l))} \right), \quad \forall\, l \geq 1,
\end{aligned}
\end{equation*}
and, by inequality \eqref{eq:MaximumPrinciple1} applied to $\|u\|_{C(\bar Q_{R^*/4,T}(w^l))}$, it follows that
\begin{equation}
\label{eq:GlobalEstBoundedCoeff2}
\begin{aligned}
\|u\|_{C^{2+\alpha}_{\rho}(\bar Q_{R^*/8,T}(w^l))}
&\leq C_2 \left(  \|Lu\|_{\sC^{\alpha}(\overline{\HH})} +  \|u(0,\cdot)\|_{\sC^{2+\alpha}(\overline{\HH}_T)}  \right), \quad \forall\, l \geq 1.
\end{aligned}
\end{equation}
Combining inequalities \eqref{eq:GlobalEstBoundedCoeff1} and \eqref{eq:GlobalEstBoundedCoeff2} and making use of the inclusions \eqref{eq:Cover1} and \eqref{eq:Cover2},
we obtain the global estimate \eqref{eq:GlobalEstBoundedCoeff}.
\end{proof}

Next, we establish the a priori global estimates in the case of coefficients with at most linear growth.

\begin{lem}[Global estimates for coefficients with linear growth]
\label{lem:GlobalEst}
There exists a positive constant $C=C(T,\alpha,d,K,\delta,\nu)$ such that for any solution $u \in \sC^{2+\alpha}_{\loc}(\overline{\HH}_T)$ to \eqref{eq:Problem}, such that $Lu \in \sC^{\alpha}_p(\overline{\HH}_T)$ and $u(0, \cdot) \in \sC^{2+\alpha}_p(\overline{\HH})$, we have
\begin{equation}
\label{eq:GlobalEstimatePrim}
\begin{aligned}
\|u\|_{\sC^{2+\alpha}(\overline{\HH}_T)}
&\leq C \left(  \|Lu\|_{\sC^{\alpha}_p(\overline{\HH}_T)} +  \|u(0, \cdot)\|_{\sC^{2+\alpha}_p(\overline{\HH})} \right),
\end{aligned}
\end{equation}
where $p=p(\alpha)$ is the constant appearing in Proposition \ref{prop:InteriorEst}.
\end{lem}

\begin{proof}
We shall apply a covering argument with the aid of Theorem \ref{thm:AprioriBoundaryEst} and Proposition \ref{prop:InteriorEst}. As in the proof of Lemma \ref{lem:GlobalEstBoundedCoeff}, we may assume without loss of generality that  $0<T\leq R^*$, where $R^*>0$ is defined as in Theorem \ref{thm:AprioriBoundaryEst}. Let $\{z^k\}$ and $\{w^l\}$ be the sequences of points considered in the proof of Lemma \ref{lem:GlobalEstBoundedCoeff}. Then, by applying Theorem \ref{thm:AprioriBoundaryEst} to $u$ with $f=Lu$ and $g=u(0,\cdot)$ on $\bar Q_{R^*,T}(z^k)$,
we obtain, for all $k \geq 1$,
\begin{equation*}
\begin{aligned}
\|u\|_{C^{2+\alpha}_s(\bar Q_{R^*,T}(z^k))}
& \leq C \left(  \|Lu\|_{C^{\alpha}_s(\bar Q_{3R^*/2,T}(z^k))} +  \|u(0,\cdot)\|_{C^{2+\alpha}_s(\bar B_{3R^*/2}(z^k))}\right.\\
&\qquad \left.+ \|u\|_{C(\bar Q_{3R^*/2,T}(z^k))} \right).
\end{aligned}
\end{equation*}
We notice that
\begin{equation*}
\begin{aligned}
\|Lu\|_{C^{\alpha}_s(\bar Q_{3R^*/2,T}(z^k))}
& \leq  C_1\|(1+|x|)^p Lu\|_{C^{\alpha}_s(\bar Q_{3R^*/2,T}(z^k))}\\
& =  C_1\|Lu\|_{\sC^{\alpha}_p(\bar Q_{3R^*/2,T}(z^k))}, \\
\|u(0,\cdot)\|_{C^{2+\alpha}_{\rho}(\bar B_{3R^*/2}(z^k))}
& \leq  C_1\|(1+|x|)^p u(0,\cdot)\|_{C^{2+\alpha}_s(\bar B_{3R^*/2}(z^k))}\\
& =  C_1\|u(0,\cdot)\|_{\sC^{2+\alpha}_p(\bar B_{3R^*/2}(z^k))}, \\
\|u\|_{C(\bar Q_{3R^*/2,T}(z^k))}
& \leq  C_1 \|(1+|x|)^p u\|_{C(\bar Q_{3R^*/2,T}(z^k))}\\
& =  C_1 \|u\|_{\sC^0_p(\bar Q_{3R^*/2,T}(z^k))},
\end{aligned}
\end{equation*}
where the positive constant $C_1$ depends on $R^*$ and $p$, but not on $z^k$. Therefore, we obtain, for all $k \geq 1$,
\begin{equation*}
\begin{aligned}
\|u\|_{C^{2+\alpha}_s(\bar Q_{R^*,T}(z^k))}
&\leq C_2 \left(  \|Lu\|_{\sC^{\alpha}_p(\overline{\HH}_T)} +  \|u(0, \cdot)\|_{\sC^{2+\alpha}_p(\overline{\HH})} + \|u\|_{\sC^0_p(\overline{\HH}_T)}\right),
\end{aligned}
\end{equation*}
for a positive constant $C_2$ depending at most on $R^*$, $K$, $\delta$, $\nu$, $\alpha$, $d$. Because the collection of balls $\{Q_{R^*,T}(z^k): k \geq 1\}$ covers $\HH_{R^*/2,T}$ and as we may apply \eqref{eq:MaximumPrinciple3} to $u$ with $f=Lu$ and $g=u(0,\cdot)$ with $q=p$,
there is a positive constant $C_3$, satisfying the same dependency on constants as $C_2$,
such that
\begin{equation}
\label{eq:GlobalEst1}
\begin{aligned}
\|u\|_{\sC^{2+\alpha}(\overline{\HH}_{R^*/2,T})}
&\leq C_3 \left(  \|Lu\|_{\sC^{\alpha}_p(\overline{\HH}_T)} +  \|u(0, \cdot)\|_{\sC^{2+\alpha}_p(\overline{\HH})}  \right).
\end{aligned}
\end{equation}
By applying Proposition \ref{prop:InteriorEst} to $u$ with $f=Lu$ and $g=u(0,\cdot)$ on $\bar Q_{R^*/8, T}(w^l)$ , we obtain, for all $l \geq 1$,
\begin{equation}
\label{eq:GlobalEst2}
\begin{aligned}
\|u\|_{C^{2+\alpha}_{\rho}(\bar Q_{R^*/8,T}(w^l))}
&\leq C_4 \left(  \|Lu\|_{\sC^{\alpha}_p(\bar Q_{R^*/4,T}(w^l))} +  \|u(0, \cdot)\|_{\sC^{2+\alpha}_p(\bar B_{R^*/4}(w^l))} \right.
\\
&\quad + \left. \|u\|_{\sC^{0}_p(\bar Q_{R^*/4,T}(w^l))} \right).
\end{aligned}
\end{equation}
Because the collection of balls $\{Q_{R^*/8,T}(w^l): l \geq 1\}$ covers $\HH_T\less\HH_{R^*/2,T}$ and we may apply \eqref{eq:MaximumPrinciple3} to $u$ with $f=Lu$ and $g=u(0,\cdot)$ with $q=p$,
we obtain
\begin{equation}
\label{eq:GlobalEst3}
\begin{aligned}
\|u\|_{\sC^{2+\alpha}(\overline{\HH_T\less\HH_{R^*/2,T}})}
&\leq C_5 \left(  \|Lu\|_{\sC^{\alpha}_p(\overline{\HH}_T)} +  \|u(0, \cdot)\|_{\sC^{2+\alpha}_p(\overline{\HH})} \right).
\end{aligned}
\end{equation}
By combining inequalities \eqref{eq:GlobalEst1} and \eqref{eq:GlobalEst3}, we obtain the desired estimate \eqref{eq:GlobalEstimatePrim}.
\end{proof}

Next, we prove Theorem \ref{thm:MainExistenceUniquenessPDE} in the case of bounded coefficients.

\begin{prop} [Existence and uniqueness for bounded coefficients]
\label{prop:ExistenceUniquenessBoundedCoeff}
Suppose Hypothesis \ref{hyp:CoeffBounded} is satisfied. Let $f \in \sC^{\alpha}(\overline{\HH}_T)$ and $g \in \sC^{2+\alpha}(\overline{\HH})$. Then there exists a unique solution $u \in \sC^{2+\alpha}(\overline{\HH}_T)$ to \eqref{eq:Problem} and $u$ satisfies estimate \eqref{eq:GlobalEstBoundedCoeff}.
\end{prop}

\begin{proof}
The proof employs the method used in proving existence of solutions to parabolic partial differential equations outlined in  \cite[\S 10.2]{Krylov_LecturesHolder} or \cite[Theorem II.1.1]{DaskalHamilton1998}. We let $\hat \sC^{2+\alpha}(\overline{\HH}_T)$ denote the Banach space of functions $u\in \sC^{2+\alpha}(\overline{\HH}_T)$ such that $u(0,x)=0$, for all $x \in \overline{\HH}$. The spaces $\hat C^{2+\alpha}_s(\overline{\HH}_T)$ and $\hat C^{2+\alpha}_{\rho}([0,T]\times\RR^d)$ are defined similarly. Without loss of generality, we may assume $g=0$ because $Lg \in \sC^{\alpha}(\overline{\HH}_T)$, when Hypothesis \ref{hyp:CoeffBounded} holds, and so
\[
L:\hat \sC^{2+\alpha}(\overline{\HH}_T) \rightarrow \sC^{\alpha}(\overline{\HH}_T)
\]
is a well-defined operator. Our goal is to show that $L$ is invertible and we accomplish this by constructing a bounded linear operator $M:\sC^{\alpha}(\overline{\HH}_T) \rightarrow \hat \sC^{2+\alpha}(\overline{\HH}_T)$ such that
\begin{equation}
\label{eq:AlmostInverse}
\begin{aligned}
\left\|LM-I_{\sC^{\alpha}(\overline{\HH}_T)}\right\| <1.
\end{aligned}
\end{equation}
For this purpose, we fix $r>0$ and choose a sequence of points $\left\{x^n: n=1,2,\ldots\right\}$ such that the collection of balls $\left\{B_r(x^n): n=1,2,\ldots\right\}$ covers the strip $\{x=(x',x_d)\in\HH: 0<x_d<r/2 \}$. We may assume without loss of generality, that there exists a positive constant $N$, depending only on the dimension $d$, such that at most $N$ balls of the covering have non-empty intersection. Let $\left\{\varphi_n: n=0,1,\ldots\right\}$ be a partition of unity subordinate to the open cover
\[
\left(\HH\less\left\{0<x_d\leq r/4\right\}\right) \cup \bigcup_{n=1}^{\infty} B_r(x^n) = \HH,
\]
such that
\[
\supp \varphi_0 \subset \HH\less\left\{0<x_d < r/4\right\} \hbox{   and   }
\supp \varphi_n \subset \bar B_r(x^n), \quad \forall\, n \geq 1.
\]
Without loss of generality, we may choose $\{\varphi_n\}_{n \geq 0}$ such that there is a positive constant $c$, independent of $r$ and $n$, such that
\begin{equation}
\label{eq:Eq0}
\begin{aligned}
\|\varphi_n\|_{C^{2+\alpha}_{\rho}(\RR^d)} \leq c r^{-3}, \quad \forall\, n \geq 0.
\end{aligned}
\end{equation}
We choose a sequence of non-negative, smooth cutoff functions, $\{\psi_n\}_{n \geq 0}\subset C^\infty(\overline\HH)$, such that $0\leq \psi_n\leq 1$ on $\HH$, for all $n\geq 0$, and
$$
\psi_0(x) = \begin{cases} 0, &\hbox{for  } 0<x_d<r/8, \\ 1, &\hbox{for  } x_d>r/4,\end{cases}
$$
while for all $n \geq 1$,
$$
\psi_n(x) = \begin{cases} 1, &\hbox{for  } 0<x_d<1/2, \\ 0, &\hbox{for  } x_d>1.\end{cases}
$$
Then, we notice that $\psi_0$ satisfies \eqref{eq:Eq0}. For $r$ small enough, we have
\begin{equation}
\label{eq:Eq1}
\psi_n \varphi_n = \varphi_n, \hbox{   for all } n\geq 0.
\end{equation}
For $n=0$, let $L_0$ be a uniformly elliptic parabolic operator on $\RR^d$ with bounded, $C^{\alpha}_\rho(\HH_T)$-H\"older continuous coefficients, such that $L_0$ agrees with $L$ on the support of $\psi_0$. Define the operator
\begin{equation*}
\begin{aligned}
M_0: C^{\alpha}_{\rho}([0,T]\times\RR^d)\rightarrow \hat C^{2+\alpha}_{\rho}([0,T]\times\RR^d),
\end{aligned}
\end{equation*}
to be the inverse of $L_0$, as given by \cite[Theorem 8.9.2]{Krylov_LecturesHolder}. For $n=1,2,\ldots$, let $L_n$ be the degenerate-parabolic operator obtained by freezing the variable coefficients $a^{ij}(t,x)$, $b^i(t,x)$ and $c(t,x)$ at $(0,x^n)$. Define the operator
\begin{equation*}
\begin{aligned}
M_n: C^{\alpha}_{s}(\overline{\HH}_T)\rightarrow \hat C^{2+\alpha}_{s}(\overline{\HH}_T),
\end{aligned}
\end{equation*}
be the inverse of $L_n$, as given by Proposition \ref{prop:ConstantCoeff}. Define the operator
$$
M: \sC^{\alpha}(\overline{\HH}_T)\rightarrow \hat\sC^{2+\alpha}(\overline{\HH}_T)
$$
by setting
\begin{equation*}
\begin{aligned}
Mf := \sum_{n=0}^{\infty} \varphi_n M_n \psi_n f, \quad \hbox{for } f \in \sC^{\alpha}(\overline{\HH}_T).
\end{aligned}
\end{equation*}
Our goal is to show that \eqref{eq:AlmostInverse} holds, for small enough $r$ and $T$.  We have
\begin{equation*}
\begin{aligned}
LMf-f &= \sum_{n=0}^{\infty} L \varphi_n M_n \psi_n f - f\\
&= \sum_{n=0}^{\infty}  \varphi_n L M_n \psi_n f + \sum_{n=0}^{\infty} [L,\varphi_n] M_n \psi_n f - f,
\end{aligned}
\end{equation*}
where $[L,\varphi_n]$ is given by \eqref{eq:ConstantCoeff2}. Denoting
\begin{equation}
\label{eq:DefinitionSequenceUn}
u_n := M_n \psi_n f, \quad\hbox{for } n=0,1,2,\ldots,
\end{equation}
we have
\begin{equation*}
\begin{aligned}
L M_n \psi_n f &= (L-L_n) u_n + L_n M_n \psi_n f\\
&= (L-L_n) u_n + \psi_n f,
\end{aligned}
\end{equation*}
since $L_n M_n = I$, for all $n\geq 0$. This implies, by the identities \eqref{eq:Eq1} and $\sum_{n=0}^{\infty} \varphi_n \psi_n f =f$, that
\begin{equation}
\label{eq:Eq2}
\begin{aligned}
LMf-f &= \sum_{n=0}^{\infty}  \varphi_n (L-L_n) u_n + \sum_{n=0}^{\infty} [L,\varphi_n] u_n .
\end{aligned}
\end{equation}
First, we estimate the terms in the preceding equality indexed by $n=0$. Because $L_0=L$ on the support of $\psi_0$, obviously we have $\psi_0(L-L_0) u_0=0$. Next, using the identity \eqref{eq:ConstantCoeff2}, there is a positive constant $C$, depending only on $K_2$ in \eqref{eq:GlobalHolderRho},  such that
\begin{equation*}
\begin{aligned}
\left\|\left[L,\varphi_0\right]u_0\right\|_{C^{\alpha}_{\rho}([0,T]\times\RR^d)}
& \leq C \|u_0\|_{C^{1+\alpha}_{\rho}([0,T]\times\RR^d)} \|\psi_0\|_{C^{2+\alpha}_{\rho}([0,T]\times\RR^d)}\\
& \leq C r^{-3}\|u_0\|_{C^{1+\alpha}_{\rho}([0,T]\times\RR^d)} \quad\hbox{(by \eqref{eq:Eq0}).}
\end{aligned}
\end{equation*}
From the interpolation inequalities for standard H\"older spaces \cite[Theorem 8.8.1]{Krylov_LecturesHolder}, there is a positive constant $m$ such that, for all $\eps>0$, we have
\begin{equation}
\label{eq:Eq3}
\begin{aligned}
\left\|\left[L,\varphi_0\right]u_0\right\|_{C^{\alpha}_{\rho}([0,T]\times\RR^d)}
& \leq C r^{-3} \left(\eps \|u_0\|_{C^{1+\alpha}_{\rho}([0,T]\times\RR^d)} + \eps^{-m} \|u_0\|_{C([0,T]\times\RR^d)} \right).
\end{aligned}
\end{equation}
By \cite[Theorem 8.9.2]{Krylov_LecturesHolder}, the identity \eqref{eq:Eq1}, and the definition \eqref{eq:DefinitionSequenceUn} of $u_0$, we have
\begin{equation*}
\begin{aligned}
\|u_0\|_{C^{1+\alpha}_{\rho}([0,T]\times\RR^d)} &\leq C_1(r) \|\psi_0 f\|_{C^{\alpha}_{\rho}([0,T]\times\RR^d)}\\
&\leq C_1(r) \|f\|_{\sC^{\alpha}(\HH_T)},
\end{aligned}
\end{equation*}
for some positive constant $C_1(r)$. From \cite[Corollary 8.1.5]{Krylov_LecturesHolder}, there is a constant $C$, depending only on $K_2$, $T$ and $d$, such that
\[
\|u_0\|_{C([0,T]\times\RR^d)} \leq CT \|f\|_{C([0,T]\times\RR^d)}.
\]
Therefore, from \eqref{eq:Eq3} we obtain, for possibly a different constant $C_1(r)$,
\begin{equation}
\label{eq:Eq4}
\begin{aligned}
\left\|\left[L,\varphi_0\right]u_0\right\|_{C^{\alpha}_{\rho}([0,T]\times\RR^d)}
& \leq C_1(r) \left(\eps \|f\|_{\sC^{\alpha}(\overline{\HH}_T)} + \eps^{-m} T\|f\|_{C(\overline{\HH}_T)} \right).
\end{aligned}
\end{equation}
Next, we estimate the terms in \eqref{eq:Eq2} indexed by $n \geq 1$. We closely follow the argument used to prove Theorem \ref{thm:AprioriBoundaryEst}. First, we have
\begin{equation}
\label{eq:Eq5}
\begin{aligned}
\|\varphi_n (L-L_n) u_n\|_{C^{\alpha}_s(\overline{\HH}_T)}
&\leq [\varphi_n]_{C^{\alpha}_s(\overline{\HH}_T)} \|(L-L_n) u_n\|_{C([0,T]\times\supp\varphi_n)}\\
&\quad + \|(L-L_n) u_n\|_{C^{\alpha}_s([0,T]\times\supp\varphi_n)}.
\end{aligned}
\end{equation}
Using \eqref{eq:Eq0} and Lemma \ref{lem:InterpolationIneqS}, there are  positive constants $m$ and $C_1(r)$ such that
\begin{equation*}
\left[\varphi_n\right]_{C^{\alpha}_s(\overline{\HH}_T)} \|(L-L_n) u_n\|_{C([0,T]\times\supp\varphi_n)}
\leq C_1(r) \left(\eps \|u_n\|_{C^{2+\alpha}_s(\overline{\HH}_T)} + \eps^{-m} \|u_n\|_{C(\overline{\HH}_T)}\right).
\end{equation*}
By Proposition \ref{prop:ConstantCoeff}, \eqref{eq:MaximumPrinciple1} and the preceding inequality, we obtain
\begin{equation*}
\left[\varphi_n\right]_{C^{\alpha}_s(\overline{\HH}_T)} \|(L-L_n) u_n\|_{C([0,T]\times\supp\varphi_n)}
\leq C_1(r) \left(\eps \|\psi_n f\|_{C^{\alpha}_s(\overline{\HH}_T)} + \eps^{-m} T \|\psi_n f\|_{C(\overline{\HH}_T)} \right),
\end{equation*}
and thus,
\begin{equation}
\label{eq:Eq6}
\left[\varphi_n\right]_{C^{\alpha}_s(\overline{\HH}_T)} \|(L-L_n) u_n\|_{C([0,T]\times\supp\varphi_n)}
\leq C_1(r) \left(\eps \| f\|_{\sC^{\alpha}(\overline{\HH}_T)} + \eps^{-m} T \| f\|_{C(\overline{\HH}_T)} \right).
\end{equation}
By applying an argument which is the same as that used to prove Claim \ref{claim:LocalBoundaryClaim1}, we find that there are positive constants, $C$, independent of $r$, and $C_1(r)$, such that
\begin{equation*}
\|(L-L_n) u_n\|_{C^{\alpha}_s([0,T]\times\supp\varphi_n)}
\leq
C r^{\alpha/2}\|u_n \|_{C^{\alpha}_s(\overline{\HH}_T)} +C_1(r)  \|u_n \|_{C(\overline{\HH}_T)}.
\end{equation*}
By Proposition \ref{prop:ConstantCoeff}, \eqref{eq:MaximumPrinciple1} and the definition \eqref{eq:DefinitionSequenceUn} of $u_n$, it follows that
\begin{equation}
\label{eq:Eq7}
\|(L-L_n) u_n\|_{C^{\alpha}_s([0,T]\times\supp\varphi_n)}
\leq C r^{\alpha/2}\|f\|_{\sC^{\alpha}(\overline{\HH}_T)} + C_1(r)  T \| f\|_{C(\overline{\HH}_T)}.
\end{equation}
With the aid of inequalities \eqref{eq:Eq6} and \eqref{eq:Eq7}, the estimate \eqref{eq:Eq5} becomes
\begin{equation}
\label{eq:Eq8}
\begin{aligned}
\|\varphi_n (L-L_n) u_n\|_{C^{\alpha}_s(\overline{\HH}_T)}
&\leq C r^{\alpha/2}\|f\|_{\sC^{\alpha}(\overline{\HH}_T)} + C_1(r) \left(\eps \| f\|_{\sC^{\alpha}(\overline{\HH}_T)} + \eps^{-m} T \| f\|_{C(\overline{\HH}_T)} \right).
\end{aligned}
\end{equation}
Next, we estimate $[L,\varphi_n]u_n$, for $n \geq 1$, by employing a method similar to that used to estimate the term $[L,\varphi_0] u_0$. Using the identity \eqref{eq:ConstantCoeff2}, there is a positive constant $C$, depending only on $K$ appearing in \eqref{eq:BoundednessNearBoundary} and \eqref{eq:LocalHolderS},  such that
\begin{equation*}
\begin{aligned}
\left\|\left[L,\varphi_n\right]u_n\right\|_{C^{\alpha}_{s}([0,T]\times\HH)}
& \leq C r^{-3}\|u_n\|_{C^{1+\alpha}_{s}([0,T]\times\HH)} \quad\hbox{(by \eqref{eq:Eq0}).}
\end{aligned}
\end{equation*}
From Lemma \ref{lem:InterpolationIneqS}, there is a positive constant $m$ such that, for all $\eps\in(0,1)$, we have
\begin{equation*}
\begin{aligned}
\left\|\left[L,\varphi_n\right]u_n\right\|_{C^{\alpha}_{s}([0,T]\times\HH)}
& \leq C r^{-3} \left(\eps \|u_n\|_{C^{1+\alpha}_{s}([0,T]\times\HH)} + \eps^{-m} \|u_n\|_{C([0,T]\times\HH)} \right).
\end{aligned}
\end{equation*}
According to Proposition \ref{prop:ConstantCoeff} and \eqref{eq:MaximumPrinciple1}, there is a constant $C_1(r)$ so that
\begin{equation}
\label{eq:Eq9}
\begin{aligned}
\left\|\left[L,\varphi_n\right]u_n\right\|_{C^{\alpha}_{s}([0,T]\times\HH)}
& \leq C_1(r) \left(\eps \|f\|_{\sC^{\alpha}(\overline{\HH}_T)} + \eps^{-m} T\|f\|_{C(\overline{\HH}_T)} \right).
\end{aligned}
\end{equation}
Combining inequalities \eqref{eq:Eq4}, \eqref{eq:Eq8} and \eqref{eq:Eq9}, and using the fact that at most $N$ balls in the covering have non-empty intersection, the identity \eqref{eq:Eq2} yields
\begin{equation*}
\begin{aligned}
\|LMf-f\|_{\sC^{2+\alpha}(\overline{\HH}_T)}
&\leq C r^{\alpha/2}\|f\|_{\sC^{\alpha}(\overline{\HH}_T)} + C_1(r) \left(\eps \| f\|_{\sC^{\alpha}(\overline{\HH}_T)} + \eps^{-m} T \| f\|_{C(\overline{\HH}_T)} \right),
\end{aligned}
\end{equation*}
where $C$ is a positive constant independent of $r$, while $C_1(r)$ may depend on $r$. By choosing small enough $r$, then small enough $\eps$, and then small enough $T$, in that order, we find a positive constant $C_0<1$ such that
\begin{equation*}
\begin{aligned}
\|LMf-f\|_{\sC^{\alpha}(\overline{\HH}_T)}
&\leq C_0\|f\|_{\sC^{\alpha}(\overline{\HH}_T)}, \quad \forall\, f \in \sC^{\alpha}(\overline{\HH}_T),
\end{aligned}
\end{equation*}
and this gives \eqref{eq:AlmostInverse}.
\end{proof}

\begin{proof} [Proof of Theorem \ref{thm:MainExistenceUniquenessPDE}]
Uniqueness of solutions follows from Proposition \ref{prop:MaximumPrinciple}.

We notice that $\sC^{\alpha}_p(\overline{\HH}_T)\subset \sC^{\alpha}(\overline{\HH}_T)$ and
$\sC^{2+\alpha}_p(\overline{\HH})\subset \sC^{2+\alpha}(\overline{\HH})$. Let $\tilde{L}$ be any operator satisfying Hypothesis \ref{hyp:CoeffBounded}. Let $\{\varphi_n\}_{n \geq 1}$ be a sequence of non-negative, smooth cut-off functions such that
\[
0\leq \varphi_n \leq 1, \quad\varphi_n|_{B_n}=1, \quad \hbox{  and  } \varphi_n|_{B^c_{2n}}=0.
\]
We define
\[
L_n := \varphi_n L + (1-\varphi_n) \tilde{L}, \quad \forall\, n \geq 1.
\]
Then, each $L_n$ satisfies Hypothesis \ref{hyp:CoeffBounded} and, by Proposition \ref{prop:ExistenceUniquenessBoundedCoeff}, there exists a unique solution $u_n \in \sC^{2+\alpha}(\overline{\HH}_T)$ to \eqref{eq:Problem} with $L=L_n$. By Lemma \ref{lem:GlobalEst}, each solution $u_n$ satisfies the global estimate,
\begin{equation}
\label{eq:GlobalEstimate_n}
\begin{aligned}
\|u_n\|_{\sC^{2+\alpha}(\overline{\HH}_T)}
&\leq C \left(  \|f\|_{\sC^{\alpha}_p(\overline{\HH}_T)} +  \|g\|_{\sC^{2+\alpha}_p(\overline{\HH})} \right).
\end{aligned}
\end{equation}
For any bounded subdomain $U\subset\HH$ and denoting $U_T=(0,T)\times U$, the parabolic analogue, $C^{2+\alpha}_\rho(\bar U_T) \hookrightarrow C^2_\rho(\bar U_T) \equiv C^{1,2}(\bar U_T)$, of the compact embedding \cite[Theorem 1.31 (4)]{Adams_1975} of standard H\"older spaces, $C^{2+\alpha}(\bar U) \hookrightarrow C^2(\bar U)$, implies that the sequence $\{u_n\}_{n\geq 1}$ converges strongly in $C^{1,2}(\bar U_T)$ to the limit $u \in C^{1,2}(U_T)$, that is, $u_n\to u$ in $C^{1,2}(U_T)$, as $n\to \infty$ for every bounded subdomain $U\subset\HH$. It is now easily seen that $u$ solves \eqref{eq:Problem}. By the Arzel$\grave{\text{a}}$-Ascoli Theorem, we obtain that $u \in \sC^{2+\alpha}(\overline{\HH}_T)$ and $u$ satisfies \eqref{eq:GlobalEstimate}.
\end{proof}

\appendix

\section{Existence and uniqueness of solutions for a degenerate-parabolic operator with constant coefficients}
\label{app:ExistenceUniquenessDegenParabolicPDEConstantCoefficients}
In order to derive the local a priori boundary estimates in Theorem \ref{thm:AprioriBoundaryEst}, we need an analogue of \cite[Theorem I.1.1]{DaskalHamilton1998} when the coefficients of our operator $L$, $a^{ij}$, $b^i$ and $c$, are assumed \emph{constant}. To emphasize this fact in this appendix, we denote our parabolic operator by
\begin{equation}
\label{eq:ConstantCoeffGenDH}
\begin{aligned}
-L_0 u := -u_t + \sum_{i,j=1}^d x_d a^{ij} u_{x_ix_j} + \sum_{i=1}^d b^i u_{x_i}+cu \quad \hbox{ on } (0,T)\times\HH.
\end{aligned}
\end{equation}
We now have the following analogue of \cite[Theorem I.1.1]{DaskalHamilton1998}.

\begin{prop}[Existence and uniqueness of solutions for a degenerate parabolic operator with constant coefficients]
\label{prop:ConstantCoeff}
Let $K$, $\delta$ and $\nu$ be positive constants such that
\begin{align}
\label{assump:ConstantCoeffNonDeg}
 \sum_{i,j=1}^d a^{ij} \eta_i\eta_j &\geq \delta \|\eta\|^2, \quad \forall\, \eta \in \RR^d,\\
\label{assump:ConstantCoeffPositiveDrift}
 b^d &\geq \nu,\\
 \label{assump:ConstantCoeffBoundednedd}
 |a^{ij}|, |b^i|, |c| &\leq K, \quad 1\leq i,j\leq d.
\end{align}
Let $k$ be a non-negative integer, $T>0$, and $\alpha \in (0,1)$. Assume that $f \in C^{k, \alpha}_s(\overline{\HH}_T)$ and $g \in C^{k, 2+\alpha}_s(\overline{\HH})$ with both $f$ and $g$ compactly supported in $\overline{\HH}_T$ and $\overline{\HH}$, respectively. Then, the inhomogeneous initial value problem,
\begin{equation}
\label{eq:ConstantCoeffPDE}
\begin{aligned}
\begin{cases}
         L_0 u = f & \mbox{ on } (0,T)\times\HH, \\
         u(0,\cdot)=g  & \mbox{ on } \overline{\HH},
\end{cases}
\end{aligned}
\end{equation}
admits a unique solution $u\in C^{k, 2+\alpha}_s(\overline{\HH}_T)$. Moreover, there is a positive constant $C=C(T, K, \delta, \nu,\alpha, d, k)$ such that
\begin{equation}
\label{eq:ConstantCoeffEst}
\begin{aligned}
\|u\|_{C^{k, 2+\alpha}_s(\overline{\HH}_T)} \leq C \left(\|f\|_{C^{k, \alpha}_s(\overline{\HH}_T)} + \|g\|_{C^{k, 2+\alpha}_s(\overline{\HH})}\right).
\end{aligned}
\end{equation}
\end{prop}

\begin{proof}
We adapt the proof of \cite[Theorem I.1.1]{DaskalHamilton1998}. Because the proof of \cite[Theorem I.1.1]{DaskalHamilton1998} is lengthy, we only outline the modifications, noting that these modifications are straightforward. We remark that there is no simple change of variables that can be applied in order to bring the
constant-coefficient equation \eqref{eq:ConstantCoeffPDE} to the form of the model equation defined in \cite[p. 901]{DaskalHamilton1998}. Another difficulty is that our interpolation inequalities (Lemma \ref{lem:InterpolationIneqS}) do not allow us to treat the first order derivatives, $u_{x_i}$, in \eqref{eq:Generator} as lower order terms: in order to do that, we would need to have
\begin{equation*}
\begin{aligned}
\|u_{x_i}\|_{C^{\alpha}_s(\overline{\HH}_T)} \leq \eps \|u\|_{C^{2+\alpha}_s(\overline{\HH}_T)} + C \eps^{-m} \|u\|_{C(\overline{\HH}_T)},
\end{aligned}
\end{equation*}
instead of the interpolation inequality \eqref{eq:InterpolationIneqS3}. On the other hand, by simple changes of variables which we describe below and
which preserve the domain $\HH$ and its boundary $\partial\HH$, problem \eqref{eq:ConstantCoeffPDE} can be simplified to
\begin{equation}
\label{eq:SimplerForm}
\begin{aligned}
- L_0 u = -u_t + x_d\sum_{i=1}^d u_{x_ix_i} + \sum_{i=1}^d b^i u_{x_i} \quad \hbox{ on } (0,T)\times\HH,
\end{aligned}
\end{equation}
where the coefficient $b^d>0$ remains unchanged. In addition, the possibly new constant coefficients $b^i$ are bounded in absolute value by constants  which depend only on $\delta$ in \eqref{assump:ConstantCoeffNonDeg} and $K$ in \eqref{assump:ConstantCoeffBoundednedd}.

The simple changes of variables are as follows. As usual, we eliminate the zeroth-order term $cu$ by multiplying $u$ by $e^{ct}$ and so we may assume without loss of generality that $c=0$ in \eqref{eq:ConstantCoeffGenDH}. We define a function $\tilde u$ on $(0,T)\times\HH$ by choosing $y = (y_1,\ldots,y_d)$ and
$$
(y_1,\ldots,y_d) := (x_1+\alpha_1x_d, \ldots, x_{d-1}+\alpha_{d-1}x_d,x_d),  \quad\hbox{where } \alpha_i := -\frac{a^{id}}{a^{dd}}, \quad 1\leq i\leq d-1,
$$
and
\[
\tilde{u}(t, y) := u(t,x).
\]
Note that $a^{dd}>\delta$, by choosing $\eta=(0,0,\ldots,1)$ in \eqref{assump:ConstantCoeffNonDeg}. By direct calculations, we obtain (omitting the arguments of the functions $u$ and $\tilde u$ for brevity),
\begin{align*}
u_{x_i} &= \tilde u_{y_i}, \quad \forall\, i \neq d,\\
u_{x_d} &= \sum_{k\neq d} \alpha_k\tilde u_{y_k} + \tilde u_{y_d}, \\
u_{x_ix_j} &= \tilde u_{y_iy_j}, \quad \forall\, i,j \neq d,\\
u_{x_ix_d} &= \sum_{k\neq d} \alpha_k\tilde u_{y_iy_k} + \tilde u_{y_iy_d}, \quad \forall\, i \neq d,\\
u_{x_dx_d} &= \sum_{k,l\neq d} \alpha_k\alpha_l\tilde u_{y_ky_l} + 2\sum_{k\neq d} \alpha_k \tilde u_{y_ky_d} + \tilde u_{y_d y_d},
\end{align*}
from where it follows that,
\begin{align*}
L u &=
\tilde u_t - y_d a^{dd} \tilde{u}_{y_dy_d}
-2 y_d\sum_{i\neq d}\left(a^{dd}\alpha_i+a^{id}\right) \tilde{u}_{y_iy_d}
- y_d\sum_{i,j\neq d} \left(a^{dd}\alpha_i\alpha_j+a^{ij}+a^{id} \alpha_j\right)\tilde{u}_{y_iy_d}\\
&\quad-\sum_{i\neq d} \left(b^i+\alpha_ib^d\right)\tilde{u}_{y_i}- b^d \widetilde{u}_{y_d}.
\end{align*}
Since $\alpha_i=-a^{id}/a^{dd}$ for all $i \neq d$, we obtain
\begin{align*}
a^{dd}\alpha_i+a^{id} &= 0,\quad \forall\, i \neq d,\\
a^{dd}\alpha_i\alpha_j+a^{ij}+a^{id} \alpha_j &= a^{ij}, \quad \forall\, i,j \neq d,
\end{align*}
and so problem \eqref{eq:ConstantCoeffPDE} is reduced to the study of the operator $\widetilde L_0$ on $(0,T)\times\HH$ given by
\[
\widetilde L_0 \tilde{u} := \tilde{u}_t - y_d a^{dd} \tilde{u}_{y_dy_d}
- y_d\sum_{i,j\neq d} a^{ij}\tilde{u}_{y_iy_j}
-\sum_{i\neq d} \left(b^i+\alpha_ib^d\right)\tilde{u}_{y_i}- b^d \widetilde{u}_{y_d}.
\]
Because the $(d-1)\times(d-1)$ matrix $\bar a:=(a^{ij})_{i,j={1,\ldots, d-1}}$ is positive-definite and symmetric, there is an orthogonal matrix $P$ such that $P^* \bar a P = D$, where $D:= \diag\left(\lambda_1,\ldots,\lambda_{d-1}\right)$ and $\lambda_i$, $i=1,\ldots,d-1$, are the (positive) eigenvalues of $\bar a$. By setting
$$
Q:= \begin{pmatrix}PD^{-1/2} & 0 \\ 0 & (a^{dd})^{-1/2}\end{pmatrix}, \quad\hbox{where }D^{-1/2} = \diag\left(\lambda_1^{-1/2},\ldots,\lambda_{d-1}^{-1/2}\right),
$$
we notice that
$$
Q^*\begin{pmatrix}\bar a & 0 \\ 0 & a^{dd}\end{pmatrix}Q = I_d,
$$
where $I_d$ is the $d\times d$ identity matrix. Proceeding as in the \cite[Proof of Lemma 6.1]{GilbargTrudinger}, we choose $z := yQ$ and
\[
\bar u(t,z) := \tilde u(t, y), \quad \forall\, (t,y) \in (0,T)\times\HH.
\]
Then, direct calculations show that problem \eqref{eq:ConstantCoeffPDE} is reduced to the study of the operator $\bar L_0$  on $(0,T)\times\HH$ given by
\[
- \bar L_0 \bar u := -\bar u_t + z_d  \sum_{i=1}^d \bar u_{z_iz_i} + \sum_{i=1}^d\bar b^i\bar u_{z_i},
\]
where the constant coefficients $\bar b^i$ may differ from the coefficients $b^i$, for $i\neq d$, and the coefficient $\bar b^d:=\sqrt{a^{dd}}b^d>0$. Therefore, for the remainder of this section, we may assume without loss of generality that $L_0$ is of the simpler form \eqref{eq:SimplerForm}.

The primary change required in the proof of \cite[Theorem I.1.1]{DaskalHamilton1998} lies in \cite[\S I.4]{DaskalHamilton1998}. The arguments in the remainder of \cite[Part I]{DaskalHamilton1998} adapt almost line by line to our model operator \eqref{eq:SimplerForm}. The goal in \cite[\S I.4]{DaskalHamilton1998} is to derive local
estimates of derivatives and this is achieved by applying a comparison principle with barrier functions. First, we need to adapt the definition of the barrier function \cite[Definition I.4.1]{DaskalHamilton1998} to one which is suitable for use with \eqref{eq:SimplerForm}.

\begin{defn}
\label{defn:BarrierFunction}
Let $0<t_1<t_2$. We say $\varphi$ is a \emph{barrier function} for $L_0$ when $t\in [t_1,t_2]$, if there are positive constants $C$ and $c$ such that
\begin{equation}
\label{eq:EquationBarrierFunction}
L_0 \varphi > -C x_d \varphi^2 + c \varphi^{3/2} + c.
\end{equation}
\end{defn}

The barrier functions in \cite[Theorems I.4.5 and I.4.8]{DaskalHamilton1998} are also barrier functions in the sense of Definition \ref{defn:BarrierFunction}. The barrier function constructed in \cite[Theorem I.4.6]{DaskalHamilton1998} needs modification because the coefficients $b^i$, $i=1,\ldots,d-1$, are non-zero in general, unlike in \cite[Part I]{DaskalHamilton1998}. We have the following modification.

\begin{claim}
\label{claim:Claim1}
Assume $i\neq d$. For any $\gamma<1$ as in \cite[Definition I.4.2]{DaskalHamilton1998}, there are a positive constant $b$, depending only on $|b^i|$, and a positive constant $\Delta$, depending only on $|b^i|$, $b$, and $\gamma$, such that, for any $t_0\geq 0$,
\begin{equation}
\label{eq:ConstaCoeffBarrierFunction}
\varphi_i(t,x) := \frac{1}{(1+x_i-b(t-t_0))^2} + \frac{1}{(1-x_i-b(t-t_0))^2}
\end{equation}
is a valid barrier function satisfying \eqref{eq:EquationBarrierFunction}, for all $t \in [t_0, t_0+\Delta]$.
\end{claim}

\begin{proof}[Proof of Claim \ref{claim:Claim1}]
It suffices to consider separately the terms ${}^{+}\varphi_i$ and ${}^{-}\varphi_i$ defined by
\[
{}^{\pm}\varphi_i := \frac{1}{(1\pm x_i-b(t-t_0))^2},
\]
because the barrier functions form a cone by \cite[Theorem I.4.4]{DaskalHamilton1998}.
We prove that ${}^{+}\varphi_i$ satisfies \eqref{eq:EquationBarrierFunction}, and the proof follows similarly for ${}^{-}\varphi_i$. We denote  $\varphi:={}^{+}\varphi_i$ for simplicity. By direct calculation, we obtain
\begin{equation*}
\begin{aligned}
\varphi_t &= 2b \varphi^{ 3/2}, \\
\varphi_{x_i} &= - 2 \varphi^{ 3/2}, \\
\varphi_{x_ix_i} &= 6 \varphi^{ 2},
\end{aligned}
\end{equation*}
while $\varphi_{x_j}=0$ and $\varphi_{x_jx_k}=0$, unless $j=k=i$.
Then, we have
\begin{equation*}
\begin{aligned}
L_0 \varphi = 2(b  + b _i) \varphi^{ 3/2} -6 x_d  \varphi^{ 2}.
\end{aligned}
\end{equation*}
We impose $1-b(t-t_0) \geq \gamma$, for all $t \in [t_0, t_0+\Delta]$, so we choose $\Delta  < (1-\gamma)/b$.
By choosing $b=|b^i|+1$, we can find $C>0$ and $c>0$ such that
\[
L_0 \varphi \geq - x_d C \varphi^{2} + c\varphi^{ 3/2} +c,
\]
and so $\varphi$ satisfies the requirement \eqref{eq:EquationBarrierFunction}, for all $t\in[t_0,t_0+\Delta]$.
\end{proof}

Next, the arguments in \cite[\S I.5]{DaskalHamilton1998} adapt to our framework with the following observation. Because our barrier functions \eqref{eq:ConstaCoeffBarrierFunction} are not defined for all $t\in [0,1]$, we cover first the interval $[0,1]$ by a finite number of intervals of length $\Delta$, as given in Claim \ref{claim:Claim1}, and we apply the maximum principle on each of the resulting subintervals. This will yield local estimates analogous to \cite[Theorems I.5.1, I.5.4 and Corollary I.5.7]{DaskalHamilton1998}, on the small time subintervals of the finite covering. By combining the local derivative estimates over each subinterval, we obtain the required local estimates for all $t \in [0,1]$.
\end{proof}

\section{Proofs of Lemma \ref{lem:ClassicalEstConstCoeff} and Proposition \ref{prop:ClassicalEstVarCoeff}}
\label{app:ProofKrylovResults}
We begin with the

\begin{proof}[The proof of Lemma \ref{lem:ClassicalEstConstCoeff}]
The proof follows the argument used to prove \cite[Lemmas 9.2.1 and 8.9.1]{Krylov_LecturesHolder}, only we are careful about the dependencies of the constants appearing in the estimates on $\delta_1$ and $K_1$, given by \eqref{eq:NonDegeneracyBoundedCoeff} and \eqref{eq:GlobalHolderContinuityBoundedCoeff}, respectively.
Let $U$ be an orthogonal matrix such that $A = U \text{diag} (\lambda_i) U^{T}$, where
$\lambda_i \in [\delta_1, K_1]$ are the eigenvalues of the symmetric, positive-definite matrix, $(a^{ij})$. We denote $B = U \text{diag}(\sqrt{\lambda_i}) U^*$ and $v(t,x)=u(t,Bx)$, $\bar{f}(t, x)=f(t,Bx)$, and $\bar{g}( x)=g(Bx)$. Then, by defining
$w(t,x):=e^{-t}v(t,x)$, we see that $w\in C^{2+\alpha}_{\rho}([0,T]\times\mathbb{R}^d)$  solves the inhomogeneous heat equation,
\begin{align*}
\begin{cases}
         w_t-\Delta w +w = e^{t}\bar{f}& \mbox{ on $(0,T]\times \mathbb{R}^d$}, \\
         w(0,\cdot)=\bar{g}&        \mbox{ on $\mathbb{R}^d$}.
\end{cases}
\end{align*}
By applying
\cite[Theorem 9.2.1]{Krylov_LecturesHolder} to $w$, we obtain a constant
$\bar N_1=\bar N_1(\alpha, d)$ such that
\begin{equation*}
\|w\|_{C^{2+\alpha}_{\rho}([0,T]\times\mathbb{R}^d)}
\leq \bar N_1 \left(\|e^{t}\bar{f}\|_{C^{\alpha}_{\rho}([0,T]\times\mathbb{R}^d)} + \|\bar{g}\|_{C^{2+\alpha}_{\rho}(\mathbb{R}^d)}\right),
\end{equation*}
which gives us, for $v(t,x)=e^{t} w(t,x)$, the estimate
\begin{equation}
\label{eq:ClassicalHolderConstCoeff}
\|v\|_{C^{2+\alpha}_{\rho}([0,T]\times\mathbb{R}^d)}
\leq N_1 \left(\|\bar{f}\|_{C^{\alpha}_{\rho}([0,T]\times\mathbb{R}^d)} + \|\bar{g}\|_{C^{2+\alpha}_{\rho}(\mathbb{R}^d)}\right),
\end{equation}
where now $N_1=N_1(\alpha,d,T)$.

To obtain \eqref{eq:SchauderEstConstCoeff} from \eqref{eq:ClassicalHolderConstCoeff}, we need the following

\begin{claim}
\label{claim:EquivalenceHolderNorms}
There is a positive constant $C=C(d)$ such that, for any $w_1 \in C^{\alpha}_{\rho}([0,T]\times\RR^d)$ and any symmetric, positive-definite $d\times d$-matrix, M, with eigenvalues in $[\lambda_{\min},\lambda_{\max}]$, where $\lambda_{\max}>\lambda_{\min}>0$, we have
\begin{align}
\label{eq:EquivalenceHolderNorms1}
\|w_1\|_{C^{\alpha}_{\rho}([0,T]\times\RR^d)} &\leq C(1+\lambda^{-\alpha}_{\min}) \|w_2\|_{C^{\alpha}_{\rho}([0,T]\times\RR^d)},\\
\label{eq:EquivalenceHolderNorms2}
\|w_2\|_{C^{\alpha}_{\rho}([0,T]\times\RR^d)} &\leq C(1+\lambda_{\max}^{\alpha}) \|w_1\|_{C^{\alpha}_{\rho}([0,T]\times\RR^d)},
\end{align}
where $w_2(t,x):=w_1(t,Mx)$.
\end{claim}

\begin{proof}[Proof of Claim \ref{claim:EquivalenceHolderNorms}]
We first prove \eqref{eq:EquivalenceHolderNorms1}. Obviously, we have
\begin{equation}
\label{eq:EquivalenceHolderNorms3}
\|w_1\|_{C([0,T]\times\RR^d)} = \|w_2\|_{C([0,T]\times\RR^d)}.
\end{equation}
Next, it suffices to consider $|w_1(P^1)-w_1(P^2)|/\rho^\alpha(P,Q)$, for points $P_i=(t^i,x^i)\in [0,T]\times\RR^d$, $i=1,2$, where only one of the coordinates differs. Notice that when $x^1=x^2$, then
\[
\frac{|w_1(P^1)-w_1(P^2)|}{\rho^\alpha(P^1,P^2)} = \frac{|w_2(P^1)-w_2(P^2)|}{\rho^\alpha(P^1,P^2)},
\]
because the transformation $w_2(t,x):=w_1(t,Mx)$ acts only on the spatial variables. Therefore, we have
\begin{equation}
\label{eq:EquivalenceHolderNorms5}
\frac{|w_1(P^1)-w_1(P^2)|}{\rho^\alpha(P^1,P^2)} \leq [w_2]_{C^{\alpha}_{\rho}([0,T]\times\RR^d)},
\end{equation}
Next, we consider the case $t^1=t^2=t$. Then, we have by writing $w_1(t,x)=w_2(t,M^{-1}x)$,
\begin{equation*}
\frac{|w_1(P^1)-w_1(P^2)|}{\rho^\alpha(P^1,P^2)} = \frac{|w_2(t,M^{-1}x^1)-w_2(t,M^{-1}x^2)|}{|M(M^{-1}x^1-M^{-1}x^2)|^{\alpha}}
\end{equation*}
Using the fact that $M$ is a symmetric, positive-definite matrix with eigenvalues in the range $[ \lambda_{\min},\lambda_{\max}]$, it follows
\[
|M(M^{-1}x^1-M^{-1}x^2)| \geq \lambda_{\min} |M^{-1}x^1-M^{-1}x^2|, \quad \forall\, x^1, x^2\in \RR^d,
\]
and so, by the preceding two inequalities, we have
\begin{equation}
\label{eq:EquivalenceHolderNorms6}
\frac{|w_1(P^1)-w_1(P^2)|}{\rho^\alpha(P^1,P^2)} \leq \lambda_{\min}^{-\alpha}[w_2]_{C^{\alpha}_{\rho}([0,T]\times\RR^d)}.
\end{equation}
Combining inequalities \eqref{eq:EquivalenceHolderNorms3}, \eqref{eq:EquivalenceHolderNorms5} and \eqref{eq:EquivalenceHolderNorms6}, we obtain \eqref{eq:EquivalenceHolderNorms1}.

To obtain \eqref{eq:EquivalenceHolderNorms2}, we apply \eqref{eq:EquivalenceHolderNorms1} to $w_2$ in place of $w_1$. Then, the matrix $M$ is replaced by the symmetric, positive-definite matrix $M^{-1}$ with eigenvalues in $[\lambda_{\max}^{-1}, \lambda_{\min}^{-1}]$. Therefore, $\lambda_{\min}^{-1}$ in \eqref{eq:EquivalenceHolderNorms1} is replaced by $\lambda_{\max}$, and thus, we obtain \eqref{eq:EquivalenceHolderNorms2}.
\end{proof}

Notice that $B$ is a symmetric, positive-definite matrix with eigenvalues in $[\sqrt{\delta_1}, \sqrt{K_1}]$. Since $v(t,x)=u(t,Bx)$, we may apply  \eqref{eq:EquivalenceHolderNorms1} with $w_1=u$ and $w_2=v$ and $M=B$ to obtain
\begin{equation}
\label{eq:EquivalenceHolderNorms7}
\|u\|_{C^{\alpha}_{\rho}([0,T]\times\mathbb{R}^d)}
\leq C(1+\delta_1^{-\alpha/2}) \|v\|_{C^{\alpha}_{\rho}([0,T]\times\mathbb{R}^d)}.
\end{equation}
Because $v_t(t,x)=u_t(t,Bx)$, we have as above
\begin{equation}
\label{eq:EquivalenceHolderNorms8}
\|u_t\|_{C^{\alpha}_{\rho}([0,T]\times\mathbb{R}^d)}
\leq C(1+\delta_1^{-\alpha/2}) \|v_t\|_{C^{\alpha}_{\rho}([0,T]\times\mathbb{R}^d)}.
\end{equation}
To evaluate $u_{x_i}$, we denote by $L^i$ the $i$-th row of the matrix $B^{-1}$. Then, we have
\[
u_{x_i} = L^i \nabla v,
\]
and so,
\[
\|u_{x_i}\|_{C([0,T]\times\RR^d)} \leq \delta_1^{-1/2} \|\nabla v\|_{C([0,T]\times\RR^d)},
\]
where we have use the fact that $B^{-1}$  is a symmetric, positive-definite matrix and the eigenvalues of $B^{-1}$ are in
$
\left[K_1^{-1/2},\delta_1^{-1/2} \right].
$
Applying inequality \eqref{eq:EquivalenceHolderNorms1} to $u_{x_i}$, we obtain as above that
\begin{equation}
\label{eq:EquivalenceHolderNorms9}
\|u_{x_i}\|_{C^{\alpha}_{\rho}([0,T]\times\mathbb{R}^d)}
\leq C\delta_1^{-1/2}(1+\delta_1^{-\alpha/2}) \|v_{x_i}\|_{C^{\alpha}_{\rho}([0,T]\times\mathbb{R}^d)},
\end{equation}
and similarly, it follows for $u_{x_ix_j}$ that
\begin{equation}
\label{eq:EquivalenceHolderNorms10}
\|u_{x_ix_j}\|_{C^{\alpha}_{\rho}([0,T]\times\mathbb{R}^d)}
\leq \delta_1^{-1}(1+\delta_1^{-\alpha/2}) \|v_{x_ix_j}\|_{C^{\alpha}_{\rho}([0,T]\times\mathbb{R}^d)}.
\end{equation}
Applying \eqref{eq:EquivalenceHolderNorms2} for $\bar f(t,x)= f(t, Bx)$ with $w_1=f$ and $w_2=\bar f$ and $M=B$, we have
\begin{equation}
\label{eq:EquivalenceHolderNorms11}
\|\bar{f}\|_{C^{\alpha}_{\rho}([0,T]\times\mathbb{R}^d)}
\leq (1+K_1^{\alpha/2}) \|f\|_{C^{\alpha}_{\rho}([0,T]\times\mathbb{R}^d)},
\end{equation}
Similarly, for $\bar g(x)=g(Bx)$, we obtain
\begin{equation}
\begin{aligned}
\label{eq:EquivalenceHolderNorms12}
\|\bar{g}\|_{C^{\alpha}_{\rho}(\mathbb{R}^d)}
&\leq (1+K_1^{\alpha/2}) \|g\|_{C^{\alpha}_{\rho}([0,T]\times\mathbb{R}^d)}, \\
\|\bar{g}_{x_i}\|_{C^{\alpha}_{\rho}(\mathbb{R}^d)}
&\leq K_1^{1/2}(1+K_1^{\alpha/2}) \|g_{x_i}\|_{C^{\alpha}_{\rho}(\mathbb{R}^d)}, \\
\|\bar{g}_{x_ix_j}\|_{C^{\alpha}_{\rho}(\mathbb{R}^d)}
&\leq K_1(1+K_1^{\alpha/2}) \|g_{x_ix_j}\|_{C^{\alpha}_{\rho}(\mathbb{R}^d)}.
\end{aligned}\end{equation}
By combining the  inequalities \eqref{eq:EquivalenceHolderNorms7}, \eqref{eq:EquivalenceHolderNorms8}, \eqref{eq:EquivalenceHolderNorms9}, \eqref{eq:EquivalenceHolderNorms10}, \eqref{eq:EquivalenceHolderNorms11} and \eqref{eq:EquivalenceHolderNorms12} in \eqref{eq:ClassicalHolderConstCoeff}, we obtain \eqref{eq:SchauderEstConstCoeff}.
\end{proof}

Next, we give the proof of Proposition \ref{prop:ClassicalEstVarCoeff}. The estimate \eqref{eq:SchauderEstVarCoeff} is obtained exactly as in the proof of \cite[Theorems 9.2.2 and 8.9.2]{Krylov_LecturesHolder} using Lemma \ref{lem:ClassicalEstConstCoeff}, except that we again provide the details in order to obtain the precise dependencies of the coefficients.

\begin{proof}[Proof of Proposition \ref{prop:ClassicalEstVarCoeff}]
Due to the classical interpolation inequalities \cite[Theorem 8.8.1]{Krylov_LecturesHolder} and the classical maximum principle for unbounded domains \cite[Corollary 8.1.5]{Krylov_LecturesHolder}, it suffices to prove that the estimate \eqref{eq:SchauderEstVarCoeff} holds with
\[
\left[u_t\right]_{C^{\alpha}_{\rho}([0,T]\times\RR^d)} \text{  and  } \left[u_{x_ix_j}\right]_{C^{\alpha}_{\rho}([0,T]\times\RR^d)}
\]
on the left-hand side of the inequality. We will prove this for the $C^{\alpha}_{\rho}([0,T]\times\RR^d)$-seminorm of $u_t$, but the same argument can be applied for the $C^{\alpha}_{\rho}([0,T]\times\RR^d)$-seminorm of $u_{x_ix_j}$.

For simplicity of notation, we denote $Q:=(0,T)\times\RR^d$, and we omit the subscript $\rho$ in the definition of the H\"older spaces. We also use the simplified notation
\begin{equation}
\label{eq:HigherHolderNorm}
[u]_{C^{2+\alpha}(\bar Q)} := \left[u_t\right]_{C^{\alpha}(\bar Q)} + \left[u_{x_ix_j}\right]_{C^{\alpha}(\bar Q)}.
\end{equation}
Let $u \in C^{2+\alpha}(\bar Q)$ be a solution to problem \eqref{eq:NonHomIniValConstCoeff}. Then,
\begin{equation}
\label{eq:AuxDefinitionBarU}
\bar u:=e^{-\lambda_1 t} u
\end{equation}
is in $C^{2+\alpha}(\bar Q)$ and it solves
\begin{equation*}
\begin{aligned}
\begin{cases}
         \left(-\bar{L}- \lambda_1\right)\bar u = -e^{-\lambda_1 t}f & \mbox{ on } (0,T)\times\mathbb{R}^d, \\
         \bar u(0,\cdot)=g  & \mbox{ on } \RR^d,
\end{cases}
\end{aligned}
\end{equation*}
where $\lambda_1$ is the upper bound on the zeroth-order coefficient, $\bar c$, assumed in \eqref{eq:UpperBoundBarC}. We may apply \cite[Corollary 8.1.5]{Krylov_LecturesHolder}, because the zeroth-order term of the parabolic operator $-\bar L -\lambda_1$ is non-positive, and we obtain
\[
\|\bar u\|_{C(\bar Q)} \leq T\|e^{-\lambda_1 t} f\|_{C(\bar Q)} + \|g\|_{C(\bar Q)} \leq  T\| f\|_{C(\bar Q)} + \|g\|_{C(\bar Q)}.
\]
Thus, it follows by \eqref{eq:AuxDefinitionBarU} that
\begin{equation}
\label{eq:BoundOnTheSupNorm}
\|u\|_{C(\bar Q)} \leq e^{\lambda_1 T } \left(T\| f\|_{C(\bar Q)} + \|g\|_{C(\bar Q)}\right).
\end{equation}
Let $z_1,z_2\in [0,T]\times\RR^d$ be two points such that
\begin{equation}
\label{eq:AuxEq1}
\frac{|u_t(z_1)-u_t(z_2)|}{\rho^{\alpha}(z_1,z_2)} \geq \frac{1}{2} [u_t]_{C^{\alpha}(\bar Q)}.
\end{equation}
Let $\gamma>0$ be a constant which will be suitably chosen below. We consider two cases.
\setcounter{case}{0}
\begin{case}[$\rho(z_1,z_2)\geq \gamma$] Then, we have
\[
[u_t]_{C^{\alpha}(\bar Q)} \leq 2 \gamma^{-\alpha} |u_t|_{C(\bar Q)},
\]
and, by \cite[Theorem 8.8.1, Inequality (8.8.1)]{Krylov_LecturesHolder}, it follows, for all $\eps>0$, that
\[
[u_t]_{C^{\alpha}(\bar Q)} \leq 2 \gamma^{-\alpha} \left(\eps [u]_{C^{2+\alpha}(\bar Q)} + C\eps^{-\alpha/2} |u|_{C(\bar Q)}\right).
\]
By choosing $\eps:=\gamma^{\alpha}/8$ and by inequality \eqref{eq:BoundOnTheSupNorm}, we obtain
\begin{equation}
\label{eq:EstimateCase1}
[u_t]_{C^{\alpha}(\bar Q)} \leq \frac{1}{4} [u]_{C^{2+\alpha}(\bar Q)} + C\gamma^{-(\alpha+\alpha^2/2)} e^{\lambda_1 T } \left(T\| f\|_{C(\bar Q)} + \|g\|_{C(\bar Q)}\right),
\end{equation}
where $C=C(d,\alpha)$.
\end{case}

\begin{case}[$\rho(z_1,z_2)< \gamma$]
We denote $z=(t,x)$. Let $\zeta:\RR^{d+1}\rightarrow[0,1]$ be a smooth cutoff function such that
\[
\zeta(z) = 1 \quad\hbox{if}\quad \rho(z,z_1)\leq 1 \quad\hbox{and}\quad \zeta(z) = 0  \quad\hbox{if}\quad \rho(z,z_1)\geq 2,
\]
and we define $\varphi$ by
\[
\varphi(z) := \zeta ((t-t_1)/\gamma^2, (x-x_1)/\gamma), \quad \forall\, z \in \RR^{d+1},
\]
so that,
\begin{equation}
\label{eq:SupportVarphi}
\varphi(z) = 1 \quad\hbox{if}\quad \rho(z,z_1)\leq \gamma \quad\hbox{and}\quad \varphi(z) = 0 \quad\hbox{if}\quad \rho(z,z_1)\geq 2\gamma,
\end{equation}
It is straightforward to see that $\varphi$ satisfies
\begin{equation}
\label{eq:HolderNormVarphi}
\|\varphi\|_{C^{2+\alpha}(\RR^{d+1})} \leq  C \left(1+\gamma^{-(2+\alpha)}\right),
\end{equation}
where $C$ is a positive constant. Since $z_2 \in \{\varphi=1\}$, we obtain by \eqref{eq:AuxEq1} that
\begin{equation}
\label{eq:AuxEq2}
[u_t]_{C^{\alpha}(\bar Q)} \leq 2\frac{|u_t(z_1)-u_t(z_2)|}{\rho^{\alpha}(z_1,z_2)} \leq 2[(u\varphi)_t]_{C^{\alpha}(\bar Q)}.
\end{equation}
Let $\bar L_0$ denote the differential operator, with constant coefficients, of the type considered in Lemma \ref{lem:ClassicalEstConstCoeff},
\begin{equation}
\label{eq:DefinitionL_0}
-\bar L_0 = -\partial_t + \sum_{i,j=1}^d \bar a^{ij}(z_1) \partial_{x_ix_j}.
\end{equation}
Estimate \eqref{eq:SchauderEstConstCoeff} shows that there are constants $p_1=p_1(\alpha)$ and $C=C(d,\alpha,T)$ such that
\begin{equation}
\label{eq:AuxEq3}
[(u\varphi)_t]_{C^{\alpha}(\bar Q)} \leq C\left(1+\delta_1^{-p_1}+K_1^{p_1}\right) \left(\|\bar L_0 (u\varphi)\|_{C^{\alpha}(\bar Q)}
 + \| g\varphi\|_{C^{2+\alpha}(\{0\}\times\RR^d)}\right).
\end{equation}
By \eqref{eq:HolderNormVarphi}, we obtain
\begin{equation}
\label{eq:AuxEq4}
\| g\varphi\|_{C^{2+\alpha}(\{0\}\times\RR^d)} \leq C\left(1+\gamma^{-(2+\alpha)}\right) \|g\|_{C^{2+\alpha}(\RR^d)}.
\end{equation}
By writing $\bar L_0(u\varphi) = L(u\varphi) + (\bar L_0-L)(u\varphi)$, we have
\begin{equation}
\label{eq:AuxEq5}
\bar L_0(u\varphi) = L(u\varphi) + \sum_{i,j=1}^d \left(\bar a^{ij}(z)-\bar a^{ij}(z_1)\right) (u\varphi)_{x_ix_j} + \sum_{i=1}^d \bar b^i(z) (u\varphi)_{x_i} + \bar c(z) u\varphi.
\end{equation}
We may write
\[
L(u\varphi) = \varphi Lu + \sum_{i,j=1}^d \bar a^{ij}(z) \varphi_{x_j} u_{x_i}
+ \left(\sum_{i,j=1}^d \bar a^{ij}(z) \varphi_{x_ix_j} + \sum_{i=1}^d \bar b^i(z) \varphi_{x_i} + \bar c(z)\right) u
\]
and so, by \eqref{eq:HolderNormVarphi} and \eqref{eq:GlobalHolderContinuityBoundedCoeff}, we find there is a positive constant $C=C(d)$ such that
\begin{equation}
\label{eq:AuxEq6}
\begin{aligned}
\|L(u\varphi)\|_{C^{\alpha}(\bar Q)} &\leq C \left(1+\gamma^{-(2+\alpha)}\right)\|Lu\|_{C^{\alpha}(\bar Q)}\\
&\quad + C K_1  \left(1+\gamma^{-(2+\alpha)}\right) \left(\|u_{x_i}\|_{C^{\alpha}(\bar Q)}+\|u\|_{C^{\alpha}(\bar Q)}\right).
\end{aligned}
\end{equation}
Notice that we may write the difference as
\begin{align*}
\bar L_0(u\varphi) - L(u\varphi)
&= \sum_{i,j=1}^d \left(\bar a^{ij}(z)-\bar a^{ij}(z_1)\right) \varphi u_{x_ix_j} \\
&\quad + \sum_{i=1}^d \left( \sum_{j=1}^d \left(\bar a^{ij}(z)-\bar a^{ij}(z_1)\right) \varphi_{x_j}+\bar b^i(z)\varphi \right) u_{x_i} \\
&\quad + \left(\sum_{i,j=1}^d \left(\bar a^{ij}(z)-\bar a^{ij}(z_1)\right) (\varphi)_{x_ix_j} + \sum_{i=1}^d \bar b^i(z) \varphi_{x_i} + \bar c(z)\varphi\right)u.
\end{align*}
By \eqref{eq:GlobalHolderContinuityBoundedCoeff}, \eqref{eq:SupportVarphi} and \eqref{eq:HolderNormVarphi}, we see that
\begin{align*}
\|\left(\bar a^{ij}(z)-\bar a^{ij}(z_1)\right) \varphi u_{x_ix_j}\|_{C^{\alpha}(\bar Q)}
\leq C K_1 \gamma^{\alpha} [u_{x_ix_j}]_{C^{\alpha}(\bar Q)} + C K_1(1+\gamma^{-(2+\alpha)}) \|u_{x_ix_j}\|_{C(\bar Q)}.
\end{align*}
From an argument similar to that used to obtain \eqref{eq:AuxEq6}, we have
\begin{equation}
\label{eq:AuxEq7}
\begin{aligned}
\|\bar L_0(u\varphi) - L(u\varphi)\|_{C^{\alpha}(\bar Q)}
&\leq C K_1 \gamma^{\alpha} [u_{x_ix_j}]_{C^{\alpha}(\bar Q)}
\\
&\quad + C K_1(1+\gamma^{-(2+\alpha)}) \left(\|u_{x_ix_j}\|_{C(\bar Q)}+\|u_{x_i}\|_{C^{\alpha}(\bar Q)}+\|u\|_{C^{\alpha}(\bar Q)}\right).
\end{aligned}
\end{equation}
Estimates \eqref{eq:AuxEq6} and \eqref{eq:AuxEq7}, give us, by \eqref{eq:AuxEq5}, that
\begin{equation}
\label{eq:AuxEq8}
\begin{aligned}
\|\bar L_0(u\varphi)\|_{C^{\alpha}(\bar Q)}
&\leq C \left(1+\gamma^{-(2+\alpha)}\right)\|Lu\|_{C^{\alpha}(\bar Q)}
 + C K_1 \gamma^{\alpha} [u_{x_ix_j}]_{C^{\alpha}(\bar Q)}  \\
&\quad + C K_1  \left(1+\gamma^{-(2+\alpha)}\right) \left(\|u_{x_ix_j}\|_{C(\bar Q)}+ \|u_{x_i}\|_{C^{\alpha}(\bar Q)}+\|u\|_{C^{\alpha}(\bar Q)}\right).
\end{aligned}
\end{equation}
Combining the preceding inequality, estimates \eqref{eq:AuxEq3} and \eqref{eq:AuxEq4} in \eqref{eq:AuxEq2}, and using definition \eqref{eq:HigherHolderNorm}, it follows that
\begin{align*}
[u_t]_{C^{\alpha}(\bar Q)}
&\leq C \left(1+\delta_1^{-p_1}+K_1^{p_1}\right)
\left(\left(1+\gamma^{-(2+\alpha)}\right)\|\bar Lu\|_{C^{\alpha}(\bar Q)} \right.\\
& \quad
 + K_1 \gamma^{\alpha} [u]_{C^{2+\alpha}(\bar Q)}  \\
&\quad + K_1  \left(1+\gamma^{-(2+\alpha)}\right) \left(\|u_{x_ix_j}\|_{C(\bar Q)}+ \|u_{x_i}\|_{C^{\alpha}(\bar Q)}+\|u\|_{C^{\alpha}(\bar Q)}\right)\\
&\quad \left.
+\left(1+\gamma^{-(2+\alpha)}\right) \|g\|_{C^{2+\alpha}(\RR^d)}
\right),
\end{align*}
where $C=C(d,\alpha,T)$. The interpolation inequalities \cite[Theorem 8.8.1]{Krylov_LecturesHolder} and the maximum principle \cite[Corollary 8.1.5]{Krylov_LecturesHolder}, give us, for any $\eps>0$,
\begin{align*}
[u_t]_{C^{\alpha}(\bar Q)}
&\leq C  \left(1+\delta_1^{-p_1}+K_1^{p_1}\right)
\\
&\quad\times\left[e^{\lambda_1 T}\left(1+K_1 \eps^{-m}\right)\left(1+\gamma^{-(2+\alpha)}\right)\left(\|f\|_{C^{\alpha}(\bar Q)} +\|g\|_{C^{2+\alpha}(\RR^d)} \right)\right.
\\
& \quad + K_1 \left.\left( \gamma^{\alpha} +\eps\left(1+\gamma^{-(2+\alpha)}\right) \right)[u]_{C^{2+\alpha}(\bar Q)}\right],
\end{align*}
where $m=m(\alpha)$. We choose $\gamma\in (0,1)$ such that
\[
C \left(1+\delta_1^{-p_1}+K_1^{p_1}\right) K_1 \gamma^{\alpha} \leq \frac{1}{16},
\]
as for instance,
\begin{equation}
\label{eq:ChoiceGamma}
\gamma : = \left(\frac{1}{48C} \min \left\{K^{-1}_1, K^{-1}_1 \delta_1^{p_1}, K^{-(1+p_1)}_1 \right\}\right)^{1/\alpha}\wedge 1.
\end{equation}
Then, we choose $\eps>0$ such that
\[
C \left(1+\delta_1^{-p_1}+K_1^{p_1}\right) \left(1+\gamma^{-(2+\alpha)}\right) K_1 \eps \leq \frac{1}{16}.
\]
A suitable choice is
\begin{equation}
\label{eq:ChoiceEpsilon}
\eps:= \frac{1}{96C} (1+\gamma^{2+\alpha})\min\left\{K^{-1}_1, K^{-1}_1 \delta_1^{p_1}, K^{-(1+p_1)}_1\right\}
\end{equation}
Then, we obtain
\begin{equation}
\label{eq:EstimateCase2}
\begin{aligned}
\left[u_t\right]_{C^{\alpha}(\bar Q)}
&\leq \frac{1}{4} [u]_{C^{2+\alpha}(\bar Q)}
+ C e^{\lambda_1 T} \left(1+\delta_1^{-p_1}+K_1^{p_1}\right)\left(1+K_1 \eps^{-m}\right)\left(1+\gamma^{-(2+\alpha)}\right)
\\
&\qquad\times\left(\|f\|_{C^{\alpha}(\bar Q)} +\|g\|_{C^{2+\alpha}(\RR^d)} \right).
\end{aligned}
\end{equation}
\end{case}

By combining inequalities \eqref{eq:EstimateCase1} and \eqref{eq:EstimateCase2} from the preceding two cases, we obtain the global estimate
\begin{equation}
\label{eq:GlobalEstimateCase2}
\begin{aligned}
\left[u_t\right]_{C^{\alpha}(\bar Q)}
&\leq \frac{1}{4} [u]_{C^{2+\alpha}(\bar Q)}
+ C e^{\lambda_1 T} \left(1+\delta_1^{-p_1}+K_1^{p_1}\right)\left(1+K_1 \eps^{-m}\right)\left(1+\gamma^{-(2+\alpha)}\right)
\\
&\qquad \times\left(\|f\|_{C^{\alpha}(\bar Q)} +\|g\|_{C^{2+\alpha}(\RR^d)} \right).
\end{aligned}
\end{equation}
We notice from \eqref{eq:ChoiceGamma} and \eqref{eq:ChoiceEpsilon} that we may find positive constants $N_3=N_3(d,\alpha,T)$ and $p=p(\alpha)$ such that
\begin{equation*}
\left[u_t\right]_{C^{\alpha}(\bar Q)}
\leq \frac{1}{4} [u]_{C^{2+\alpha}(\bar Q)} + N_3 e^{\lambda_1 T} \left(1+\delta_1^{-p}+K_1^{p}\right)
\left(\|f\|_{C^{\alpha}(\bar Q)} +\|g\|_{C^{2+\alpha}(\RR^d)} \right).
\end{equation*}
The similar argument applied to $[u_{x_ix_j}]_{C^{\alpha}(\bar Q)}$ yields
\begin{equation*}
\left[u_{x_ix_j}\right]_{C^{\alpha}(\bar Q)}
\leq \frac{1}{4} [u]_{C^{2+\alpha}(\bar Q)} + N_3 e^{\lambda_1 T} \left(1+\delta_1^{-p}+K_1^{p}\right)
\left(\|f\|_{C^{\alpha}(\bar Q)} +\|g\|_{C^{2+\alpha}(\RR^d)} \right).
\end{equation*}
Therefore, \eqref{eq:HigherHolderNorm} gives us
\begin{equation*}
\left[u\right]_{C^{2+\alpha}(\bar Q)}
\leq N_3 e^{\lambda_1 T} \left(1+\delta_1^{-p}+K_1^{p}\right)
\left(\|f\|_{C^{\alpha}(\bar Q)} +\|g\|_{C^{2+\alpha}(\RR^d)} \right),
\end{equation*}
which concludes the proof of the proposition by the interpolation inequalities \cite[Theorem 8.8.1]{Krylov_LecturesHolder} and the maximum principle estimate \eqref{eq:BoundOnTheSupNorm}.
\end{proof}

%
%

\bibliography{mfpde}
\bibliographystyle{amsplain}

\end{document}